\documentclass[11pt]{article}
\usepackage{amsmath,amssymb,amsthm}
\usepackage{enumerate}  
\usepackage{float}
\usepackage{graphicx}
\usepackage[margin=1.0in]{geometry}
\usepackage{url}
\usepackage{xcolor}
\usepackage{caption}
\usepackage{hhline}
\usepackage{bm}
\usepackage[linesnumbered,ruled,algosection]{algorithm2e}
\usepackage{subcaption}
\usepackage{lineno}
\usepackage{url}
\usepackage{hyperref}
\usepackage{cleveref}

\SetKw{Continue}{continue}

\def \dis{\displaystyle}

\def \R{\mathbb{R}} 

\def \N{\mathbb{N}}

\def \e{\varepsilon}

\def \W{\Omega}
\def \phi{\varphi}

\def \12{\dis\frac{1}{2}}

\def \1{\mathbbm{1}}

\def \<{\left<}
\def \>{\right>}

\def \mA{\mathcal{A}}
\def \mB{\mathcal{B}}
\def \mE{\mathcal{E}}
\def \mT{\mathcal{T}}

\def \bn{n}

\def \mL{\mathcal{L}}

\def \grad{\nabla}

\def \mM{\mathcal{M}}

\def \UDLRA{U_{\rm{\tiny DLRA}}}

\def \UFE{U_{\rm{FE}}^{n+1}}
\def \RLR{R_{\rm{LR}}}
\renewcommand{\tilde}{\widetilde}
\renewcommand{\hat}{\widehat}
\SetKwComment{Comment}{/* }{ */}

\def \lavg{\{\!\{}
\def \ravg{\}\!\}}
\def \ljmp{[\![}
\def \rjmp{]\!]}

\DeclareMathOperator*{\argmin}{argmin}

\DeclareMathOperator*{\Col}{Col}
\def \dx[#1]{\ensuremath{\operatorname{d}\!{#1}}}

\numberwithin{equation}{section}
\newtheorem{defn}{Definition}
\numberwithin{defn}{section}
\newtheorem{remark}{Remark}
\numberwithin{remark}{section}
\newtheorem{prop}{Proposition}
\numberwithin{prop}{section}
\newtheorem{lemma}{Lemma}
\numberwithin{lemma}{section}
\newtheorem{theorem}{Theorem}
\numberwithin{theorem}{section}

\numberwithin{corollary}{section}
\numberwithin{figure}{section}
\numberwithin{table}{section}

\crefname{algocf}{alg.}{algs.}
\Crefname{algocf}{Algorithm}{Algorithms}

\title{A Predictor-Corrector Strategy for Adaptivity in Dynamical Low-Rank Approximations\thanks{Notice: This manuscript has been authored by UT-Battelle, LLC under Contract No. DE-AC05-00OR22725 with the U.S. Department of Energy.  The publisher, by accepting the article for publication, acknowledges that the U.S. Government retains a non-exclusive, paid up, irrevocable, world-wide license to publish or reproduce the published form of the manuscript, or allow others to do so, for U.S. Government purposes. The DOE will provide public access to these results in accordance with the DOE Public Access Plan (\url{http://energy.gov/downloads/doe-public-access-plan}).
}}

\author{
Cory D.\ Hauck\thanks{Mathematics in Computation Section, Computer Science and Mathematics Division, Oak Ridge National Laboratory, Oak Ridge, TN 37831, USA and 
Mathematics Department, University of Tennessee, Knoxville, TN 37996, USA (\href{mailto:hauckc@ornl.gov}{hauckc@ornl.gov}).} 
\and Stefan R.\ Schnake\thanks{Mathematics in Computation Section, Computer Science and Mathematics Division, Oak Ridge National Laboratory, Oak Ridge, TN 37831, USA  (\href{mailto:schnakesr@ornl.gov}{schnakesr@ornl.gov}).}
}

\begin{document}

\maketitle

\begin{abstract}
In this paper, we present a predictor-corrector strategy for constructing rank-adaptive dynamical low-rank approximations (DLRAs) of matrix-valued ODE systems.  The strategy is a compromise between (i) low-rank step-truncation approaches that alternately evolve and compress solutions and (ii) strict DLRA approaches that augment the low-rank manifold using subspaces generated locally in time by the DLRA integrator.  The strategy is based on an analysis of the error between a forward temporal update into the ambient full-rank space, which is typically computed in a step-truncation approach before re-compressing, and the standard DLRA update, which is forced to live in a low-rank manifold.  We use this error, without requiring its full-rank representation, to correct the DLRA solution.  A key ingredient for maintaining a low-rank representation of the error is a randomized singular value decomposition (SVD), which introduces some degree of stochastic variability into the implementation.  The strategy is formulated and implemented in the context of discontinuous Galerkin spatial discretizations of partial differential equations and applied to several versions of DLRA methods found in the literature, as well as a new variant. Numerical experiments comparing the predictor-corrector strategy to other methods demonstrate robustness to overcome short-comings of step truncation or strict DLRA approaches:  the former may require more memory than is strictly needed while the latter may miss transients solution features that cannot be recovered.  The effect of randomization, tolerances, and other implementation parameters is also explored.
\end{abstract}

\section{Introduction}
Low-rank methods are a class of numerical tools that are used to represent high-dimensional data with low-dimensional objects.  In many scenarios, they offer substantial memory savings in the storage and transmission of data which can be leveraged to construct efficient models of high-dimensional systems \cite{grasedyck2013literature}.  Two common low-rank methods for dynamical systems are step-truncation methods and dynamical low-rank approximation (DLRA).   

Step truncation methods are a generic tool for sparse approximation of dynamical systems \cite{schneider1997comparison,schaeffer2013sparse,guo2022low,dektor2021rank}.  These methods update a sparse initial condition in the ambient space and then compress the solution back to a sparse representation through a truncation procedure that removes elements considered unnecessary for maintaining a certain level of accuracy. Low-rank representations are an effective way to enforce a sparse structure.  

Dynamical low-rank approximations solve an approximate model that, in the absence of adaptivity, maintains a fixed rank for all time.  First introduced as the Dirac-Freknel-McLachlan variational principle in the 1930s \cite{dirac1930note,frenkel1934wave}, DLRA methods were revived in the 2000s for the simulation of high-dimensional quantum systems and biological cellular systems \cite{jahnke2008dynamical,beck2000multiconfiguration,lubich2005variational}. Even more recently, DLRA methods have been applied to kinetic equations \cite{einkemmer2018low,peng2022sweep,peng2020low,einkemmer2021efficient,einkemmer2021asymptotic,ding2021dynamical}, hyperbolic problems with uncertainty \cite{kusch2022dynamical}, and even neural network training \cite{schotthofer2022low}.  DLRA methods evolve a tensor-valued ODE on a rank-$r$ manifold $\mM_r$ by projecting the time derivative onto the tangent space of $\mM_r$.  This forces the DLRA to be rank-$r$ for all time and can make certain high-dimensional problems computationally tractable \cite{jahnke2008dynamical}.  

In \cite{lubichkoch2007DLRA}, DLRA methods were rigorously analyzed showing that they are robust and accurate.  Additionally, in \cite{lubichkoch2007DLRA} the authors show the low-rank factors of the DLRA solution solve a coupled system, but this system is often numerically unstable to march forward in time unless an unreasonably small timestep is taken \cite{lubich2014projector}.  Recent work has produced several numerically stable and accurate DLRA temporal integrators \cite{lubich2014projector,kieri2016discretized,kieri2019projection,ceruti2021unconventional}.  Extensions of these methods to asymptotically preserving or conservative schemes have also been of interest \cite{ding2021dynamical,einkemmer2021mass,peng2021high}.

The current work focuses on adaptivity in the DLRA framework.  In practice, the rank required to accurately approximate a dynamical system may change in time, due to non-linear phenomena such as filamentation and turbulence, initial and boundary layers, contributions from sources, and coordinate-based numerical artifacts. Numerical methods that allow the rank of the solution to change in time are called \textit{rank-adaptive} methods and are vital to efficiently resolve these dynamics.  Several rank-adaptive methods have recently been developed \cite{ceruti2022rank,guo2022low,dektor2021rank,hu2022adaptive}, and, in general, these methods take one of two strategies.  Some stay completely in the DLRA regime, where they are prone to modelling error by the low-rank projection \cite{ceruti2022rank}.  Others construct a full-rank object that is accurate and compress back to a low-rank object \cite{guo2022low,guo2022conservative}, but this can be inefficient if the rank of the solution does not change too much.

In the current paper, we develop a rank-adaptive predictor-corrector (RAPC) strategy for existing DLRA methods \cite{ceruti2021unconventional,kieri2019projection} that strikes a balance between the advantages of step truncation methods and a strict DLRA-based approach.  The DLRA serves as the predictor and its error to the full-rank forward Euler update serves as the corrector.  We use a randomized SVD \cite{halko2011finding} to efficiently approximate the corrector by a low-rank object, and then modify the rank of the predictor utilizing low-rank information from the corrector.  This yields a strategy that both allows for information outside of the low-rank space to enter the solution and does not explicitly construct a high-rank object.  The rank-adaptive tensor scheme of \cite{dektor2021rank} follows a similar strategy as the RAPC strategy, but the estimator is only used to determine the possible rank at the next time step; the predictor-corrector strategy uses the corrector information to fill in the rank deficiency of the DLRA update if need be.  Numerical results illustrated in this paper show that the predictor-corrector strategy nearly matches the rank of the rank-adaptive DLRA method of \cite{ceruti2022rank}.  While the RAPC method uses more memory than rank-adaptive DLRA methods, we present an example where the DLRA regime is not rich enough to accurately approximate a solution while the RAPC approach is still accurate.

The paper is organized as follows.  In \Cref{sect:notation}, we introduce preliminary notation, an illustrative two-dimensional solid-body rotation problem, and a standard discontinuous Galerkin discretization that serves as a ``full-rank" numerical reference.  In \Cref{sect:mateval}, we introduce separability of a bilinear form, show how low-rank methods can take advantage of separability, and finally show that the 2D solid-body problem is a sum of separable forms.  In \Cref{sect:DLRA}, we summarize the DLRA method, motivate several numerical integrators recently developed in the DLRA framework, and construct a new integrator that is accurate and has favorable properties with respect to the predictor-corrector strategy.  In \Cref{sect:adapt}, we list several rank-adaptive integrators and give a detailed motivation and construction of the predictor-corrector strategy.  In \Cref{sect:results}, we compare the predictor-corrector strategy against other rank-adaptive integrators on several related test problems.  Additionally, we test the effects of the randomized SVD on the accuracy of the rank-adaptive DLRA solution.

\section{Notation and Preliminaries}\label{sect:notation}

\subsection{Notation}

For $z \in \{x,y\}$, let $\W_z=(-L,L)\subset \R$ for some $L>0$ be an open bounded domain with boundary $\partial\W_z=\{-L,L\}$.  Denote the computational domain $\W=\W_x\times\W_y$ and let $(\cdot,\cdot)_D$ denote the standard $L^2$ inner product on $L^2(D)$ for any domain $D$.  

Given a scalar mesh parameter $h_z > 0$, let $\mathcal{T}_{z,h}:=\mathcal{T}_{z,h_z}$ be a mesh on $\W_z$ with interior skeleton $\mathcal{E}_{z,h}^\mathrm{I}:=\mathcal{E}_{z,h_z}^\mathrm{I}$. Given an edge $e\in\mE_{z,h}^\mathrm{I}$ with $e=T^+\cap T^-$ for $T^+,T^-\in \mT_{z,h}$, let $\bn^\pm$ be the outward normal vector on $T^\pm$.  Given a scalar function $q$ defined on $\partial T^+\cap e$ and $\partial T^-\cap e$ define the average and jump of $q$, respectively, by
\begin{equation}
    \lavg q\ravg = \frac{1}{2}(q\big|_{T^+} + q\big|_{T^-}) \qquad\text{and}\qquad \ljmp q \rjmp = q\big|_{T^+}\bn^+ + q\big|_{T^-x}\bn^-.
\end{equation}

Let $V_{z,h}:=V^{k_z}_{z,h_z}$ be standard the discontinuous Galerkin Lagrange finite element space:
\begin{equation}
V^{k_z}_{z,h_z} = \{q\in L^2(\W_z):q\big|_T \in \mathbb{Q}_{k_z}(T)\quad \forall T\in \mathcal{T}_{z,h}\},
\end{equation}
where $\mathbb{Q}_{k_z}(T)$ is the set of all polynomials on $T$ with degree less than or equal to $k_z$ in each coordinate.  Let $V_h:=V_{x,h}\otimes V_{z,h}$.  Additionally let $\left<\cdot,\cdot\right>_{\mE_{z,h}^\mathrm{I}}:=\sum_{e\in\mE_{z,h}^\mathrm{I}}\left<\cdot,\cdot\right>_e$.

Given two $m\times n$ matrices $A,B$, let
\begin{equation}
(A,B)_F = \text{tr}(B^TA) = \sum_{i=1}^m\sum_{j=1}^na_{ij}b_{ij},
\end{equation}
be the Frobenius inner product with induced norm
$
\|A\|_F = \sqrt{(A,A)_F}.
$
Denote by $\{\sigma_{i}(A)\}_{i=1}^{\min\{m,n\}}$ the singular values of $A$, arranged in a non-increasing order with respect to $i$.  

We utilize MATLAB-style slicing notation for matrices.  If $A$ is an $m\times n$ matrix, $I=\{i_1,\ldots,i_k\}$ is a subset of $\{1,\ldots,m\}$ and $J=\{j_1,\ldots,j_k\}$ is a subset of $\{1,\ldots,n\}$, then $\tilde{A}:=A(I,J)$ is a $k\times l$ matrix defined by $\tilde{A}_{kl} = A_{i_k,j_l}$.  The notation $i\!:\!k\subseteq\{1,\ldots,n\}$ is a shorthand for the subset $\{i,i+1,\ldots,k\}$, $:$ is taken to mean $1\!:\!n$,
and $\text{diag}(A_1,A_2,\ldots,A_N)$ is a block diagonal matrix.

\subsection{The PDE and its Discretization}
\label{subsect:solid-body}

The techniques presented in this paper can be applied to many advective or diffusive PDEs; however, we will focus on a solid-body rotation problem as a model for the topics covered in \Cref{sect:mateval}.  Other operators of physical interest are discussed in \Cref{sect:results}.  

The 2D solid-body rotation is defined by the following system:
\begin{subequations}
\begin{alignat}{3} \label{eqn:pde}
    \frac{\partial u}{\partial t} - y\cdot\grad_xu + x\cdot\grad_y u &= s,
    &&\qquad(x,y)\in\W, &&\quad t>0; \\
    u(x,y,t) &= g_x(x,y,t),
    &&\qquad(x,y) \in\partial\W_x^-, &&\quad  t>0; \\
    u(x,y,t) &= g_y(x,y,t),
    &&\qquad(x,y) \in\partial\W_y^-, &&\quad t>0; \\
    u(x,y,0) &= u_0(x,y),
    &&\qquad(x,y)\in\W, &&
\end{alignat}
\end{subequations}
where $s(\cdot,\cdot,t)\in L^2(\W)$ is a source term and $g_z$ is the inflow data.  Here $\partial\W_z^{\pm}$ is the outflow (+) and inflow (-) boundaries on $\partial\W_z$, given by
\begin{align}
\label{eqn:inflow_y_def}
    \partial\W_x^{\pm}- = \{(x,y)\in\partial\W_x\times\W_y : \pm y\cdot \bn(x) < 0\}~~\text{and}~~
    \partial\W_y^{\pm} = \{(x,y)\in\W_x\times\partial\W_y : \mp x\cdot \bn(y) < 0\}.
\end{align}
We discretize \eqref{eqn:pde} by the standard discontinuous Galkerin method with upwind fluxes. The discretized problem is as follows:  Find $u_h(\cdot,\cdot,t)\in V_h$ such that 

\begin{equation}\label{eqn:num_prob}
    \bigg(\frac{\partial u_h}{\partial t},q_h\bigg)_\W = \mA(u_h,q_h) + \mathcal{G}(q_h,t) + (s,q_h)_\W \quad\forall q_h\in V_h,
\end{equation}
where $\mA:V_h\times V_h\to\R$ and $\mathcal{G}:V_h\times\R^+\to\R$ are given by
\begin{align} \label{eqn:bilinear}
\begin{split}
    \mA(w_h,q_h) &= -(-yw_h,\grad_x q_h)_\W + \left<-y\lavg w_h\ravg + \tfrac{1}{2}|y|\ljmp w_h\rjmp,\ljmp q_h\rjmp\right>_{\mE_{x,h}^\mathrm{I}\times\W_y} \\
    &\quad+\left<-yw_h,\bn q_h\right>_{\partial\W_x^+} \\
    &\quad-(xw_h,\grad_v q_h)_\W + \left<x\lavg w_h\ravg + \tfrac{1}{2}|x|\ljmp w_h\rjmp,\ljmp q_h\rjmp\right>_{\W_x\times\mE_{y,h}^\mathrm{I}} \\
    &\quad+\left<xw_h,\bn q_h\right>_{\partial\W_y^+},
\end{split}
\\
\label{eqn:G_def}
    \mathcal{G}(q_h,t) &= \left<y g_x(\cdot,\cdot,t),\bn q_h\right>_{\partial\W_x^-} + \left<-x g_y(\cdot,\cdot,t),\bn q_h\right>_{\partial\W_y^-}.
\end{align}

\section{Conversion and Evaluation of Matrix Forms}\label{sect:mateval}

\subsection{Low-rank savings}\label{subsect:low_rank_savings}

To extract low-rank features of $u_h$, we convert the coefficients of $u_h$ to a matrix form.  Let $\{\phi_i(x)\}_{i=1}^m$ and $\{\psi_j(y)\}_{j=1}^n$ be a basis for $V_{x,h}$ and $V_{y,h}$, respectively,
that are orthonormal in $L^2(\W_x)$ and $L^2(\W_y)$, respectively. Then $u_h$ has an expansion
\begin{equation}\label{eqn:uh_coords}
    u_h(x,y,t) = \sum_{i=1}^m\sum_{j=1}^n u_h^{ij}(t)\phi_i(x)\psi_j(y)
\end{equation} 
where the coefficients $u_h^{ij}$ are given by
\begin{equation} \label{eqn:u_coeff_formula}
    u_h^{ij} = (u_h,\phi_i\psi_j)_\W.
\end{equation}
These coefficients may be assembled into the $m\times n$ matrix $U_h(t)$ where $(U_h)_{ij} = u_h^{ij}$, in which case 
the expansion for $u_h$ in \eqref{eqn:uh_coords} can be written as
\begin{align} \label{eqn:vec_to_mat}
u_h(x,y,t) =  \Phi(x)U_h(t)\Psi(y)^T
\end{align}
where $\Phi(x) = [\phi_1(x),\ldots,\phi_m(x)]$ and $\Psi(y) = [\psi_1(y),\ldots,\psi_n(y)]$.
For any function $q_h\in V_h$, we use the capitalized version $Q_h$ denote its matrix valued coordinate representation derived as in \eqref{eqn:vec_to_mat}. 

This paper is largely devoted the case when $U_h$ is a low-rank matrix.

\begin{defn}
A matrix $U\in\R^{m\times n}$ is said to be rank-$r$ if the column space
\begin{equation}\label{eqn:col_space_def}
\Col(U) = \{Uz:z\in\R^n\}.
\end{equation}
has dimension $r$.   The manifold of all such matrices is denoted $\mM_r$.
\end{defn}
The following is a standard result from linear algebra
\begin{prop}
If a matrix $U\in\R^{m\times n}$ is rank-$r$, then $r \leq \min\{m,n\}$.  Moreover, $U$ has a decomposition of the form $U=CSD^T$ where $C\in\R^{m\times r}$ and $D\in\R^{n\times r}$ are orthogonal matrices and $S\in\R^{r\times r}$ is non-singular. 
\end{prop}
If $U_h(t)$ is rank-$r$ with decomposition $U_h(t)=C(t)S(t)D(t)^T$, then according to \eqref{eqn:vec_to_mat},
\begin{equation}
    u_h(x,y,t) =  [C(t)\Phi(x)] S(t) [D(t)\Psi(y)]^T.
\end{equation}
The matrices $C$ and $D$ map the time-independent bases functions in $\Phi(x)$ and $\Psi(y)$ to the time-dependent \textit{low-rank basis functions} $\xi_{k}(x,t) := C(:,k)(t)\Phi(x)$ and $\upsilon_k(y,t) := D(:,k)(t)\Psi(y)$.  Thus if $X(x,t) = [\xi_1(x,t),\ldots,\xi_r(x,t)] = C(t)\Phi(x)$ and $Y(y,t) = [\upsilon_1(x,t),\ldots,\upsilon_r(x,t)] = D(t)\Psi(y)$, then
\begin{equation}
u_h(x,y,t) = \sum_{k,l=1}^r S_{kl}(t)\xi_k(x,t)\upsilon_l(y,t) = X(x,t)S(t)Y(y,t)^T.
\end{equation}
The matrix $S$ glues the low-rank basis functions together to reconstruct $u_h$, and for $r \ll {\min\{m,n\}}$, it is much cheaper to store the low-rank factors $C$, $D$, and $S$, which together require $mr+r^2+nr$ floating point values, than $U_h$ directly, which requires $mn$ floating point values.  

Dynamic low-rank approximation (DLRA) \cite{lubichkoch2007DLRA} methods update $X$, $S$, and $Y$ in tandem, but since $\Phi$ and $\Psi$ are independent of $t$, tracking the evolution of $X$ and $Y$ in time is equivalent to tracking  the evolution of $C$ and $D$.  While DLRA methods on paper can be applied to a generic matrix-valued ODE 
\begin{equation}
    \frac{\partial U}{\partial t} = F(U), \qquad U \in \mathbb{R}^{m \times n}, \quad F: \mathbb{R}^{m \times n} \to \mathbb{R}^{m \times n},
\end{equation}
the memory and computational savings provided by DLRA can be realized only if the operator $F$ respects the low-rank framework---that is, only if $F(CSD^T)$ can be evaluated without reconstructing the coefficients of $U_h$ by direct multiplication of the low-rank factors $C$, $S$, and $D$.  Bilinear forms that can separate their actions on the $x$ and $y$ domains tend to respect the low-rank structure.  This notion of separability is codified in \Cref{defn:sep}.

\begin{defn}\label{defn:sep} Let $\mB$ be a bilinear form on $V_h\times V_h$.  Given $N\in\N$ we say that $\mB$ is $N$-separable with respect to the basis $\{\phi_i\psi_j\}$ if for every $\kappa \in\N$ with $1\leq \kappa\leq N$ there exists bilinear forms $\mB_{\kappa}:V_h\times V_h\to\R$, $\mB_{\kappa,x}:V_{x,h}\times V_{x,h}\to\R$, $\mB_{\kappa,y}:V_{y,h}\times V_{y,h}\to\R$ such that
\begin{equation} \label{eqn:b_sum}
\mB(\phi_i\psi_j,\phi_k\psi_l) = \sum_{\kappa=1}^N\mB_\kappa(\phi_i\psi_j,\phi_k\psi_l) = \sum_{\kappa=1}^N\mB_{\kappa,x}(\phi_i,\phi_k)\mB_{\kappa,y}(\psi_j,\psi_l).
\end{equation}
We call $\mB_\kappa$ the terms of $\mB$, and call $\mB_{\kappa,x}$ and $\mB_{\kappa,y}$ the factors of $\mB_{\kappa}$.
\end{defn}

We can exploit separability to evaluate $N$-separable bilinear forms by the following proposition.
\begin{prop}\label{prop2.1}
Given $N\in\N$, let $\mB_\kappa$ be an $N$-separable bilinear form with expansion given in \Cref{defn:sep}.  Let $A_{\kappa} \in \mathbb{R}^{n \times n}$ and $B_{\kappa} \in \mathbb{R}^{m \times m}$ be matrices with elements
$(A_{\kappa})_{ik} = \mB_{\kappa,x}(\phi_k,\phi_i)$
and
$(B_{\kappa})_{jl} = \mB_{\kappa,y}(\psi_l,\psi_j)$ for all $1\leq\kappa\leq N$.
Then for any $w_h,q_h\in V_h$,
\begin{equation} \label{eqn:sep_sum}
\mB(w_h,q_h) = \sum_{\kappa=1}^N (A_{\kappa}W_hB_{\kappa}^T,Q_h)_F,
\end{equation}
where $W_h$ and $Q_h$ are the matrix coefficient representations of $w_h$ and $q_h$ respectively.
\end{prop}

\begin{proof}
Let $w_h=\sum_{i=1}^m\sum_{j=1}^n (W_h)_{ij}\phi_i\psi_j$ and $q_h=\sum_{k=1}^m\sum_{l=1}^n (Q_h)_{kl}\phi_k\psi_l$.  Then for any $1\leq \kappa\leq N$,
\begin{align} \label{eqn:prop2.1:1}
\begin{split}
\mB_\kappa(w_h,q_h) &= \sum_{ijkl} (W_h)_{ij}(Q_h)_{kl}\mB_{\kappa,x}(\phi_i,\phi_k)\mB_{\kappa,y}(\psi_j,\psi_l)
= \sum_{ijkl} (W_h)_{ij}(Q_h)_{kl}(A_{\kappa})_{ki}(B_{\kappa})_{lj} \\
&= \sum_{kl} (A_{\kappa}W_hB_{\kappa}^T)_{kl}(Q_h)_{kl} = (A_{\kappa}W_hB_{\kappa}^T,Q_h)_F.
\end{split}
\end{align}
Summing \eqref{eqn:prop2.1:1} over $\kappa$ yields \eqref{eqn:sep_sum}.  The proof is complete.  
\end{proof}

\begin{defn}\label{def:N-sep-lin}
A linear operator $\mL:\R^{m\times n}\to\R^{m\times n}$ is called $N$-separable provided it is of the form 
\begin{equation}\label{eqn:N-sep-form}
    \mL W = \sum_{\kappa=1}^N A_\kappa W B_{\kappa}^T \quad\forall W\in\R^{m\times n},
\end{equation}
where $A_\kappa\in\R^{m\times m}$ and $B_\kappa\in\R^{n\times n}$ for all $1\leq\kappa\leq N$.
\end{defn}

From \Cref{prop2.1}, \Cref{def:N-sep-lin}, and the Riesz Representation theorem, the existence of the linear operator $\mL:\R^{m\times n}\to\R^{m\times n}$ by
\begin{align}\label{eqn:mL_def}
    (\mL W_h,Q_h)_F = \mB(w_h,q_h) \quad\forall w_h,q_h\in V_h
\end{align}
is immediate.

\Cref{prop:matrix_evals} shows that the low-rank structure of $W_h$ allows for memory efficient evaluations of $\mL W_h$ and the action $Q\to (\mL W_h)$ -- both of which do not require assembly of $U_h$ by multiplication of the low-rank factors.

\begin{prop}\label{prop:matrix_evals}
Suppose $W_h$ has a rank-$r$ decomposition $W_h=CSD^T$ where $C\in\R^{m\times r}, S\in\R^{r\times r}$, and $D\in\R^{n\times r}$.  Furthermore for $1\leq\kappa\leq N$, suppose the matrices $A_{\kappa}$ and $B_{\kappa}$ in \eqref{eqn:N-sep-form} are sparse with $\mathcal{O}(m)$ and $\mathcal{O}(n)$ elements respectively.  Then $\mL W_h$ can be stored using $Nr(m+n+r)$ floating point values and computed in $\mathcal{O}(Nr(m+n))$ floating point operations (flops).
Moreover, suppose $Q$ is an $n\times q$ matrix. Set $\rho=\max\{r,q\}$.  Then the action $Q\to (\mL W_h)Q\in\R^{m\times q}$ requires $\mathcal{O}(Nr\rho(n+m))$ flops.  A similar operation count holds for the action $Q\to Q^T(\mL W_h)$ for $Q\in\R^{m\times q}$.  
\end{prop}

\begin{proof}
By \eqref{eqn:N-sep-form}, 
\begin{equation} \label{eqn:low_rank_eval}
     \mL W_h = \sum_{\kappa=1}^N A_{\kappa}CSD^TB_{\kappa}^T = \sum_{\kappa=1}^N (A_{\kappa}C)S(B_{\kappa}D)^T.
\end{equation}
Since $A_{\kappa}C\in\R^{m\times r}$ and $B_{\kappa}D\in\R^{n\times r}$ for all $\kappa$, we can store a copy of $A_{\kappa}C$, $S$, and $B_{\kappa}D$ which requires $r(m+n+r)$ floating point values for each $\kappa$ and $Nr(m+n+r)$ in total.  The sparsity of $A_{\kappa}$ and $B_{\kappa}$ imply that the operations $A_{\kappa}C$ and $B_{\kappa}D$ require $\mathcal{O}(mr)$ and $\mathcal{O}(nr)$ flops respectively.  Thus storage of $\mL W_h$ in this way requires $\mathcal{O}(Nr(m+n))$ flops.  To evaluate $(\mL W_h)Q$, we use \eqref{eqn:low_rank_eval}:
\begin{equation} \label{eqn:low_rank_eval_DLRA}
     (\mL W_h)Q = \sum_{\kappa=1}^N (A_{\kappa}CS)(D^TB_{\kappa}^TQ).
\end{equation}
The matrix $A_{\kappa}CS\in\R^{m\times r}$ and $D^TB_{\kappa}^TQ\in\R^{r\times q}$ require $\mathcal{O}(mr+mr^2)$ and $\mathcal{O}(nr+nrq)$ flops to evaluate while their product requires $\mathcal{O}(mrq))$ flops to multiply. Dropping the lower order $mr$ and $nr$ terms, the evaluation of $(A_{\kappa}CS)(D^TB_{\kappa}^TQ)\in\R^{m\times q}$ requires $\mathcal{O}{r(mr+mq+nq)}\leq \mathcal{O}{r\rho(m+n)}$ flops.  Therefore \eqref{eqn:low_rank_eval_DLRA} implies that the evaluation of $(\mL W_h)Q$ has a cost of $\mathcal{O}(Nr\rho(m+n))$. The proof is complete.
\end{proof}

\begin{remark}
\begin{enumerate}~
\item \eqref{eqn:low_rank_eval} is used in \cite{guo2022low} in order to store $\mL W_h$ in a low memory fashion as compared to the direct storage of the $m\times n$ matrix.  Memory savings of this storage technique rely on $Nr$ being small.  If either $N$ or $r$ becomes large, then storage by \eqref{eqn:low_rank_eval} becomes impractical. 
\item \eqref{eqn:low_rank_eval_DLRA} shows the memory footprint of $(\mL W_h)Q$ is not dependent on $N$.
\item DLRA methods heavily use \eqref{eqn:low_rank_eval_DLRA} with $Q$ being one of the basis matrices $C$ or $D$.  In addition, the predictor-corrector based DLRA method introduced in this paper also uses \eqref{eqn:low_rank_eval_DLRA} with the columns of $Q$ small.
\end{enumerate}
\end{remark}

\subsection{The low-rank structure of the PDE }\label{subsect:lr_features}

It turns out that the bilinear form $\mA$ defined in \eqref{eqn:bilinear} is a $N$-separable.  The proof of the following result is given in the Appendix.

\begin{prop}
\label{prop:A_low_rank}
Given any basis $\{\phi_i\psi_j\}$ on $V_h$, the bilinear form $\mA$ in \eqref{eqn:bilinear} is $N$-separable with $N=4$.  Moreover, the terms in the decomposition \eqref{eqn:b_sum} are given by
\begin{subequations}
\label{eqn:adv_decomp}
\begin{align}
    \mathcal{B}_{1,x}(\phi_i,\phi_j) &= -(\phi_i,\grad_x\phi_k)_{\W_x} + \left<\lavg\phi_i\ravg - \tfrac{1}{2}\ljmp\phi_i\rjmp,\ljmp\phi_k\rjmp\right>_{\mE_{x,h}^\mathrm{I}} + \left<\phi_in,\phi_k\right>_{\{x=-L\}} \label{eqn:adv_decomp_1x} \\
    \mathcal{B}_{1,y}(\psi_j,\psi_l) &= (-y\psi_j,\psi_l)_{\{y > 0\}} \label{eqn:adv_decomp_1y} \\
    \mathcal{B}_{2,x}(\phi_i,\phi_j) &= -(\phi_i,\grad_x\phi_k)_{\W_x} + \left<\lavg\phi_i\ravg + \tfrac{1}{2}\ljmp\phi_i\rjmp,\ljmp\phi_k\rjmp\right>_{\mE_{x,h}^\mathrm{I}} + \left<\phi_in,\phi_k\right>_{\{x=L\}} \label{eqn:adv_decomp_2x} \\
    \mathcal{B}_{2,y}(\psi_j,\psi_l) &= (-y\psi_j,\psi_l)_{\{y < 0\}} \label{eqn:adv_decomp_2y} \\
    \mathcal{B}_{3,x}(\phi_i,\phi_k) &= (x\phi_i,\psi_k)_{\{x > 0\}} \label{eqn:adv_decomp_3x} \\
    \mathcal{B}_{3,y}(\psi_j,\psi_l) &= -(\psi_j,\grad_y\psi_l)_{\W_y} +\left<\lavg \psi_j\ravg+\tfrac{1}{2}\ljmp \psi_j\rjmp,\ljmp \psi_l\rjmp\right>_{\mE_{y,h}^\mathrm{I}} + \left<\psi_jn,\psi_l\right>_{\{y = L\}} \label{eqn:adv_decomp_3y} \\
    \mathcal{B}_{4,x}(\phi_i,\phi_k) &= (x\phi_i,\psi_k)_{\{x < 0\}} \label{eqn:adv_decomp_4x} \\ 
    \mathcal{B}_{4,y}(\psi_j,\psi_l) &= -(\psi_j,\grad_y\psi_l)_{\W_y} +\left<\lavg \psi_j\ravg-\tfrac{1}{2}\ljmp \psi_j\rjmp,\ljmp \psi_l\rjmp\right>_{\mE_{y,h}^\mathrm{I}} + \left<\psi_jn,\psi_l\right>_{\{y = -L\}} \label{eqn:adv_decomp_4y} 
\end{align}
\end{subequations}
\end{prop}

\begin{remark}
    If central fluxes and periodic boundary conditions were used in the DG discretization, then only two terms would be needed: one for $-y\cdot\grad_x u$ and one for $x\cdot\grad_y u$.  However, more terms are needed to handle the upwind numerical flux and the outflow boundary condition. 
\end{remark}

\begin{remark}
    $\mB_{1,x}$ and $\mB_{2,x}$ are a downwind and upwind flux DG discretizations of the operator $\partial_x$ respectively.  While $\mB_{1,x}$ is a negative semi-definite form due to the downwind numerical flux, $\mB_{1,y}$ is also negative semi-definite.  Therefore their product is a positive semi-definite form on $V_h\times V_h$. Indeed, using the notation in \eqref{eqn:b_sum}, $\mB_{\kappa}$ is positive semi-definite for $\kappa=1,2,3,4$. 
\end{remark}

To handle the other terms in \eqref{eqn:num_prob}, let $s_h\in V_h$ be the $L^2$ projection of the source $s$ in \eqref{eqn:pde} onto $V_h$, defined by the relation
\begin{equation}
    (s_h,q_h)_\W = (s,q_h)\quad\forall q_h\in V_h,
\end{equation}
and let $S_h$ be the matrix representation of $s_h$ with coefficients derived using \eqref{eqn:u_coeff_formula}.  We assume the operations $Q\to S_hQ$ and $Q^T\to Q^TS_h$ are both efficiently computed.  This can be done if the rank of $S_h$ is small.  Additionally, let $g_h(t)\in V_h$ represent the action of $\mathcal{G}$; that is, 
\begin{equation} \label{eqn:g_h_def}
(g_h(t),q_h)_\W = \mathcal{G}(q_h,t) \quad\forall q_h\in V_h. 
\end{equation}
The existence and uniqueness of $g_h$ is given by the Riesz representation theorem.  Let $G_h$ be the matrix representation of $g_h$ with coefficients derived using \eqref{eqn:u_coeff_formula}.  We can show that the rank of $G_h$ is at most four by the following proposition, whose proof is in the Appendix.

\begin{prop} \label{prop:G_low_rank}
$G_h(t)=C(t)S(t)D(t)^T$ where $C(t)\in\R^{m\times 4}, S(t)\in\R^{r\times r}$ and $D(t)\in\R^{n\times 4}$ for any $t\geq 0$.  Additionally, the rank of $G_h$ is at most 4.  
\end{prop}
In matrix form, \eqref{eqn:num_prob} becomes
\begin{align}\label{eqn:num_ode_form}
    \frac{\partial U_h}{\partial t} = F(U_h,t) :=  \mL U_h + G_h(t) + S_h(t).
\end{align}

\section{Dynamic Low-Rank Approximation} \label{sect:DLRA}

In this section we give a brief introduction to Dynamic Low-Rank Approximation (DLRA) and several associated temporal integrators.  As the content in this section does not depend on the DG discretization, we consider the general ODE
\begin{align}\label{eqn:non_h_ode}
\begin{split}
\frac{\partial U}{\partial t}(t) &= F(U(t),t),\quad t>0; \\
U(0) &= U_0.  
\end{split}
\end{align}
Abusing notation, we will suppress the $t$ argument of $F$ and write $F(U,t) = F(U)$.

\subsection{A Short Introduction}

DLRA was developed in \cite{lubichkoch2007DLRA} as efficient way to produce a low-rank approximation to $U(t)$ that doesn't require compression at each step of a time integration scheme. This is achieved by evolving the continuous ODE on the manifold $\mM_r$.  Indeed if $U_\text{DLRA}\in \mM_r$ is a rank-$r$ approximation at time $t$, then one can guarantee $U_\text{DLRA}\in \mM_r$ for all future time by forcing its time derivative to live in $T_{U_\text{DLRA}}\mM_r$ --- the tangent space of $\mM_r$ at $U_\text{DLRA}$.  The rank-$r$ solution $\UDLRA$ is defined by the following evolution equation:
\begin{align}\label{eqn:DLRA_def}
    \frac{\partial U_\text{DLRA}}{\partial t} = \argmin_{Z\in T_{\UDLRA}\mM_r}\|F(U_\text{DLRA})-Z\|_F,
\end{align}
which is equivalent to the following Galerkin condition:
\begin{align}\label{eqn:DLRA_galerkin}
    \left(\frac{\partial U_\text{DLRA}}{\partial t},Z\right)_F =  (F(U_\text{DLRA}),Z)_F \quad\forall Z\in T_{\UDLRA}\mM_r;
\end{align}
which is again equivalent to
\begin{align}\label{eqn:DLRA_proj}
    \frac{\partial \UDLRA}{\partial t} = P_{T_{\UDLRA}\mM_r} F(\UDLRA),
\end{align}
where $P_{T_{\UDLRA}\mM_r}$ is the orthogonal projection onto ${T_{\UDLRA}\mM_r}$.  

Any implementation of \eqref{eqn:DLRA_def}-\eqref{eqn:DLRA_proj} should leverage the low-rank structure of $\UDLRA$ in order to be memory efficient.  If $\UDLRA = C(t)S(t)D(t)^T$, then $T_{U_\text{DLRA}}\mM_r$ is given by \cite{lubichkoch2007DLRA}:
\begin{align}\label{eqn:def_tangent_space}
\begin{split}
    T_{U_\text{DLRA}}\mM_r &= \{\delta CSD^T + C\delta S D^T + CS\delta D^T: \delta C\in\R^{m\times r} \text{ with }\delta C^TC = 0, \\
    &\qquad\delta S\in\R^{r\times r}, \delta D\in\R^{n\times r} \text{ with }\delta D^TD = 0\},
\end{split}
\end{align}
where the gauge conditions $\delta C^TC = 0$ and $\delta D^TD = 0$ guarantee a unique representation.  Moreover, $C,S,D$ satisfy the ODE system \cite{lubichkoch2007DLRA}:
\begin{subequations}\label{eqn:eqn_of_motion}
\begin{align}
C'(t) &= (I-CC^T)F(CSD)DS^{-1}, \label{eqn:eqn_of_motion_C} \\
S'(t) &= C^TF(CSD)D, \label{eqn:eqn_of_motion_S} \\
D'(t) &= (I-DD^T)F(CSD)^TC(S^{-1})^T. \label{eqn:eqn_of_motion_D}
\end{align}
\end{subequations}
It was shown in \cite{lubichkoch2007DLRA} that $\|\UDLRA-U\|_F$ where $U$ solves \eqref{eqn:non_h_ode} is of $\mathcal{O}(\e)$ for small time where $\e$ is the distance from $U$ to $\mM_r$.  Therefore $\UDLRA$ is a quasi-optimal rank-$r$ approximation of $U$.  Additionally, in the formulation of \eqref{eqn:eqn_of_motion}, $F(\UDLRA)$ is evaluated only via products of the form $F(\UDLRA)D$ or $C^TF(\UDLRA)$.  If these products can be evaluated efficiently, like we have shown in \Cref{sect:mateval}, then the solution $\UDLRA$ need not be explicitly constructed; and in such cases, one may expect computational and memory savings over the full-rank system.

While \eqref{eqn:eqn_of_motion} has a lot of promising features, it is often numerically unstable to advance in time.  This is because if the effective rank of $\UDLRA$ is less than the prescribed $r$, then numerically $S$ will contain singular values near machine epsilon.  Thus timestepping \eqref{eqn:eqn_of_motion} will be unstable unless $\Delta t$ is taken to be impracticably small (see \cite{lubich2014projector}).  Because of this, other integrators have been developed to bypass this obstacle while keeping the main advantage of \eqref{eqn:eqn_of_motion} -- the evolution of the low-rank factors.  The remainder of \Cref{sect:DLRA} is devoted to such integrators.

\subsection{The Unconventional Integrator} \label{subsec:uc}

The \textit{unconventional integrator} for \eqref{eqn:eqn_of_motion} was introduced in \cite{ceruti2021unconventional}.  For completeness, we list a motivation and derivation of the algorithm here. 

The goal is to create an integrator that does not invert $S$.  This is achieved by first isolating the combined variable $K=CS$.  Right multiplying \eqref{eqn:eqn_of_motion_C} by $S$, left multiplying \eqref{eqn:eqn_of_motion_S} by $C$, then summing yields a modified system of \eqref{eqn:eqn_of_motion}:
\begin{subequations}\label{eqn:UC_1}
\begin{align}
C'(t)S + CS'(t) &= F(CSD)D, \label{eqn:UC_1K} \\
D'(t) &= (I-DD^T)F(CSD)^TC(S^{-1})^T. \label{eqn:UC_1D}
\end{align}
\end{subequations}
The right hand side of \eqref{eqn:UC_1K} is equal $K'(t)$.  Additionally, the left hand side of \eqref{eqn:UC_1K} can be written in terms of $K$ rather than $C$.  Applying such modifications gives
\begin{subequations}\label{eqn:UC_2}
\begin{align}
K'(t) &= F(KD)D, \label{eqn:UC_2K} \\
D'(t) &= (I-DD^T)F(CSD)^TC(S^{-1})^T. \label{eqn:UC_2D}
\end{align}
\end{subequations}
However, \eqref{eqn:UC_2D} is explicitly dependent on $C$, not $K$.  Given a $\Delta t>0$, we remove the $C$ coupling in \eqref{eqn:UC_2D} by instead solving the approximate system
\begin{subequations} \label{eqn:UC_K}
\begin{align} \label{eqn:UC_K1}
    K'(t) &= F(KD^T)D, \\
    D'(t) &= 0. \label{eqn:UC_K2}
\end{align}
\end{subequations}
on $[t_0,t_0+\Delta t]$.  Since $D$ remains constant in time we can replace $D$  with $D=D_0$ where $C_0S_0D_0$ is the approximation to $U$ at $t$.  Additionally we set the initial condition for $K(t_0)=C_0S_0$.  While we cannot recover $C$ using \eqref{eqn:UC_K}, the purpose of $C$ is to track the evolution of the basis functions.  Therefore, any orthogonal matrix that spans the same column space of $K_1 := K(t_0+\Delta t)$ is sufficient.  Such an updated matrix, denoted as $C_1$, can be obtained by a QR or SVD decomposition of $K_1$.

A similar system holds for evolving the $D$ basis.  Setting $L=DS^T$, then we solve the approximate ODE system
\begin{subequations} \label{eqn:UC_L}
\begin{align} \label{eqn:UC_L1}
    L'(t) &= F(CL^T)^TC, \\
    C'(t) &= 0.  \label{eqn:UC_L2}
\end{align}
\end{subequations}
on $[t_0,t_0+\Delta t_0]$.  Again $C=C_0$, $L(t_0) = D_0S_0^T$, and $D_1$ is an orthogonal matrix with equal column space of $L(t_0+\Delta t)$.  
Once the new bases $C_1$ and $D_1$ are known. $S$ is updated in the new space spanned by  $C_1$ and $D_1$ through a Galerkin projection of \eqref{eqn:non_h_ode}.  It was shown in \cite{ceruti2021unconventional} that the global timestepping error of the unconventional integrator is $\mathcal{O}(\e+\Delta t)$ where $\e$ is the distance from $\mM_r$ to the solution $U$ of \eqref{eqn:non_h_ode}. 

\Cref{alg:uc} details a forward Euler timestepping method applied to the unconventional integrator.  Due to the line 5 of the algorithm, $C_0$ and $C_1$ must be kept in memory at the same time.  Thus the memory requirement for  \Cref{alg:uc} is near double the storage cost of storing $U$ via a low-rank factorization.

\begin{algorithm}[H]
\DontPrintSemicolon
\SetKwInOut{Output}{Output~}
\SetKwInOut{Input}{Input~}
\caption{Forward Euler Timestepping with Unconventional Integrator}\label{alg:uc}
\Input{$C_0\in\R^{m\times r},S_0\in\R^{r\times r},D_0\in\R^{n\times r}$ \tcp*{$U^n=C_0S_0D_0^T$}}
\Output{$C_1\in\R^{m\times r},S_1\in\R^{r\times r},D_1\in\R^{n\times r}$ \tcp*{ $U^{n+1}=C_1S_1D_1^T$}}
\BlankLine
$K_1 = C_0S_0 + \Delta t F(C_0S_0D_0^T)D_0$\;
$[C_1, \sim] = \textsf{qr}(K_1)$\;
$L_1 = D_0S_0^T + \Delta t F(C_0S_0D_0^T)^TC_0$\;
$[D_1, \sim] = \textsf{qr}(L_1)$\;
$\tilde{S} = (C_1^TC_0)S_0(D_0^TD_1)$\;
$S_1 = \tilde{S} + \Delta t C_1^TF(C_1\tilde{S}D_1^T)D_1$
\end{algorithm}

\subsection{Tangent Projection DLRA} \label{subsec:tan}

Another integrator comes from a direct temporal discretization of the low-rank ODE \eqref{eqn:DLRA_proj}.  
Given the decomposition $U=CSD^T$, an explicit representation of the projection operator $P_{T_U}:=P_{T_U\mM_r}$ is given by \cite{lubich2014projector}:
\begin{align}\label{eqn:proj}
\begin{split}
    P_{T_U}Z &= CC^TZ - CC^TZDD^T + ZDD^T.
\end{split}
\end{align}
Due to \eqref{eqn:proj}, we write the forward Euler update of $U^n=C_0S_0D_0^T$ as
\begin{align}\label{eqn:tan_fe}
\begin{split}
    U^{n+1} &= U^n + \Delta tP_{T_{U^n}}F(U^n) \\
    &= U^n + \Delta t(C_0C_0^TF(U^n) - C_0C_0^TF(U^n)D_0D_0^T + F(U^n)D_0D_0^T)
\end{split}
\end{align}
Since $U^n\in T_{U^n}\mM_r$, then $U^{n+1}$ is the projection of the full-rank forward Euler update onto the tangent space.  Rewriting $U^{n+1}$ in a low-rank fashion yields \Cref{alg:tan}.  The algorithm shows the construction of $U^{n+1}$ requires only two evaluations of $F$ -- unlike the three required for the unconventional integrator.  However, one downside comes from lines 3 and 5 that show $U^{n+1}$ must be stored in a rank $2r$ factorization.  Thus the rank will exponentially grow without some culling mechanism (see \Cref{def:cull}).  In \cite{kieri2019projection}, the authors discuss such culling methods as well as extensions to higher order timestepping schemes.
\begin{defn}\label{def:cull}
Suppose $U\in\R^{m\times n}$ is a rank-$r$ matrix $U$ with SVD decomposition $U=CSD^T$, where $C\in\R^{m\times r}$, $S\in\R^{r\times r}$, and $D\in\R^{n\times r}$.  Given $r_1< r$, the cull $U$ to rank $r_1$ to replace $U$ by $\tilde{U}$ where
\begin{equation}
    \tilde{U} = \sum_{i=1}^{r_1}\sigma_i(U)c_id_i^T = \tilde{C}\tilde{S}\tilde{D}^T.   
\end{equation}
Here $\tilde{U}$ is rank $r_1$, and  $C\in\R^{m\times r_1}$, $S\in\R^{r\times r_2}$, and $D\in\R^{n\times r_1}$ are defined by
\begin{equation}
    \tilde{C} = C(:,1\!:\!r_1),\ \tilde{S} = S(1\!:\!r_1,1\!:\!r_1),\ \tilde{D} = D(:,1\!:\!r_1).
\end{equation}
\end{defn}

\begin{algorithm}[H] 
\DontPrintSemicolon
\SetKwInOut{Output}{Output~}
\SetKwInOut{Input}{Input~}
\caption{Forward Euler Timestepping with Tangent Integrator}\label{alg:tan}
\Input{$C_0\in\R^{m\times r},S_0\in\R^{r\times r},D_0\in\R^{n\times r}$ \tcp*{$U^n=C_0S_0D_0^T$}}
\Output{$C_1\in\R^{m\times 2r},S_1\in\R^{2r\times 2r},D_1\in\R^{n\times 2r}$ \tcp*{ $U^{n+1}=C_1S_1D_1^T$}}
\BlankLine
$K_1 = F(C_0S_0D_0^T)D_0$\;
$\tilde{S} = C_0^TK_1$ \;
$[C_1,R_C] = \textsf{qr}(\begin{bmatrix}C_0 & K_1\end{bmatrix})$\;
$L_1 = F(C_0S_0D_0^T)^TC_0$\;
$[D_1,R_D] = \textsf{qr}(\begin{bmatrix}D_0 & L_1\end{bmatrix})$\;
$S_1 = 
R_C
\begin{bmatrix}
S_0-\Delta t\tilde{S} & \Delta t I \\
\Delta t I & 0
\end{bmatrix}
R_D^T$
\end{algorithm}

\subsection{Projected Unconventional Integrator} \label{subsec:proj}

Finally, we present a perturbation of the unconventional integrator from \Cref{subsec:uc} that is beneficial to our predictor-corrector based discussed in \Cref{subsect:rara}.  This integrator evolves the basis functions $C$ and $D$ in the exact same way as the unconventional integrator.  Given the evolved basis $C_1$, $D_1$ from lines 3 and 5 of \Cref{alg:uc}, let $H$ be the subspace of $\R^{m\times n}$ defined by
\begin{equation}
H = \{C_1WD_1^T:W\in\R^{r\times r}\}
\end{equation}
The update of the coefficient matrix $S$ in the unconventional integrator (lines 6 and 7 of \Cref{alg:uc}) can be written using the projector $P_H$ as 
\begin{equation}
 U^{n+1} = P_H(U^n + \Delta tF(P_HU^n))
\end{equation}
where
\begin{align}\label{eqn:H_proj}
    P_HZ = C_1C_1^TZD_1D_1^T.
\end{align}
Again, this is applying a Galerkin projection to the full-rank ODE \eqref{eqn:non_h_ode} and then discretizing in time using forward Euler. However, we can reverse the order of these operations.  We instead project the full-rank forward Euler update onto $H$, that is,
\begin{align}\label{eqn:proj_update}
    U^{n+1} = P_H(U^n + \Delta tF(U^n)).
\end{align}
The projected unconventional integrator is given for a forward Euler timestep in \Cref{alg:proj}.  One advantage of the projected unconventional integrator is that the error between the full-rank foward Euler update $\UFE$ and the low-rank approximation is is the the orthogonal complement of $P_H\UFE$.  This property has a benefit when our rank-adaptive algorithm in \Cref{subsect:rara} is applied to the projected unconventional integrator. 

\begin{algorithm}[H] 
\DontPrintSemicolon
\caption{Forward Euler Timestepping with Projected Unconventional Integrator}\label{alg:proj}
\SetKwInOut{Output}{Output~}
\SetKwInOut{Input}{Input~}
\Input{$C_0\in\R^{m\times r},S_0\in\R^{r\times r},D_0\in\R^{n\times r}$ \tcp*{$U^n=C_0S_0D_0^T$}}
\Output{$C_1\in\R^{m\times r},S_1\in\R^{r\times r},D_1\in\R^{n\times r}$ \tcp*{ $U^{n+1}=C_1S_1D_1^T$}}
\BlankLine
$K_1 = C_0S_0 + \Delta t F(C_0S_0D_0^T)D_0$\;
$[C_1,\sim] = \textsf{qr}(K_1)$\;
$L_1 = D_0S_0^T + \Delta t F(C_0S_0D_0^T)^TC_0$\;
$[D_1,\sim] = \textsf{qr}(L_1)$\;
$S_1 = (C_1^TC_0)S_0(D_0^TD_1) + \Delta t C_1^TF(C_0S_0D_0^T)D_1$
\end{algorithm}

We now show that this method is first order accurate in the following theorem which borrows heavily from Theorem 4 of \cite{ceruti2021unconventional}. 

\begin{theorem}\label{thm:proj_error}
Suppose $U(t)$ is the solution to \eqref{eqn:non_h_ode} with initial condition $U_0$ on the interval $[0,T]$ for some final time $T>0$.  Suppose following assumptions are held:
\begin{enumerate}
    \item $F$ is Lipschitz continuous in $Y$, uniformly in $t$, and  bounded uniformly in $(Y,t)$; that is, there exist $L,B>0$ such that 
    \begin{equation}
    \|F(Y,t)-F(Z,t)\|_F \leq L\|Y-Z\|_F\text{ and }\|F(Y,t)\|_F\leq B
    \end{equation}
    for all $Y,Z\in\R^{m\times n}$ and $0\leq t\leq T$
    \item There is an $\e>0$ such that
    \begin{equation}
        \|F(Y,t) - P_{T_{Y}}F(Y,t)\|_F \leq \e
    \end{equation} 
    for all $Y\in\mM_r$ in a neighbourhood of $U(t)$ and all $0\leq t\leq T$.
    \item There is a $\delta > 0$ such that
    \begin{equation}
        \|U^0-U(t^0)\|_F \leq \delta
    \end{equation}
\end{enumerate}
Suppose $\Delta t$ is sufficiently small such that the the standard forward Euler timestepping iteration is stable w.r.t $\|\cdot\|_F$, and furthermore suppose $U$ is $C^2$ on $(0,T)$.  Then setting $t_n = n\Delta t$ and letting $U^n$ be the result of $n$ steps of the projected unconventional integrator, we have
\begin{align}
\label{eqn:proj_global_err}
    \|U^n-U(t^n)\|_F \leq c_0\delta + c_1\e + c_2\Delta t
\end{align}
provided $t_n\leq T$.  The constants $c_i$ depend only on $L$, $B$, and $T$.
\end{theorem}

Before showing the proof of \Cref{thm:proj_error}, we first derive a local truncation error estimate.
\begin{lemma}\label{lem:proj_lte}
Suppose $U(t^n) = U^n$ where each are defined in \Cref{thm:proj_error}, then assuming all suppositions in \Cref{thm:proj_error}, we have
\begin{align}\label{eqn:proj_lte}
    \|U(t^{n+1})-U^{n+1}\|_F \leq \Delta t(\hat{c}_1\e + \hat{c}_2\Delta t)
\end{align}
where $\hat{c_1}$ and $\hat{c_2}$ are positive constants that depend of $L$, $B$, and $T$.
\end{lemma}

\begin{proof}
We introduce $\vartheta$ from Equation 6 of \cite{ceruti2021unconventional}:
\begin{equation}
    \vartheta = (4e^{L\delta t}BC + 9BL)\Delta t^2 + (3e^{L\Delta t}+4)\e\Delta t + e^{Lh}\delta.
\end{equation}
Since $U(t^n) = U^n$, then $\delta=0$ for this estimate.  Let $U_{\text{FE}}^{n+1} = U^n + \Delta t F(U^n)$.  Since $U^{n+1}=P_HU_{\text{FE}}^{n+1}$ and by a triangle inequality we have
\begin{align}
\begin{split}
    \|U(t^{n+1})-U^{n+1}\|_F &\leq \|U(t^{n+1})-P_HU(t^{n+1})\|_F + \|P_HU(t^{n+1})-P_HU_{\text{FE}}^{n+1}\|_F \\
    &\leq \|(I-P_H)U(t^{n+1}) \|_F + \|U(t^{n+1})-U_{\text{FE}}^{n+1}\|_F =: I_1 + I_2.
\end{split}
\end{align}
By Lemma 3 of \cite{ceruti2021unconventional}, $\|I_1\|_F\leq \vartheta$.  Moreover, since $U$ is sufficiently smooth, the full-rank forward Euler update satisfies 
\begin{equation}
    \|I_2\|_F \leq C(\Delta t)^2
\end{equation}
for $C$ dependent on $L$ and $B$.  Therefore we have \eqref{eqn:proj_lte}.  The proof is complete. 
\end{proof}

Borrowing the stability of the full-rank forward Euler update, we can now prove \Cref{thm:proj_error}

\begin{proof}[Proof of \Cref{thm:proj_error}]
Since the full-rank forward Euler iteration is stable, then so must the projection unconventional integrator since it is a projected version of the full-rank update.  Thus we can extend the local error estimate from \Cref{lem:proj_lte} to the global estimate \eqref{eqn:proj_global_err}.  The proof is complete.
\end{proof}

\section{Adaptive Algorithms} \label{sect:adapt}

The algorithms listed in \Cref{sect:DLRA} all operate on a fixed rank-$r$ manifold $\mM_r$.  Choosing $r$ too small risks polluting the solution with error from the low-rank approximation, while choosing $r$ too large wastes memory and computational resources.  To illustrate this point, we show in \Cref{fig:DLRA_advance} the results for  the discretized 2D solid body rotation rotation problem \eqref{eqn:num_prob} with a box initial condition.  For small values of $r$ accuracy of the solution is severely degraded.  However, for $r$ sufficiently large, but still much less that the full-rank, the accuracy of the solution is reasonable and further increasing $r$ does not significantly improve the solution. 

Unfortunately, an efficient choice for $r$ is rarely known a priori, and adaptive rank algorithms are needed.    In the remainder of this section, we will present two recent adaptive algorithms from the literature and introduce a new predictor-corrector approach.

\begin{figure}
    \centering
    \includegraphics[width=0.48\textwidth]{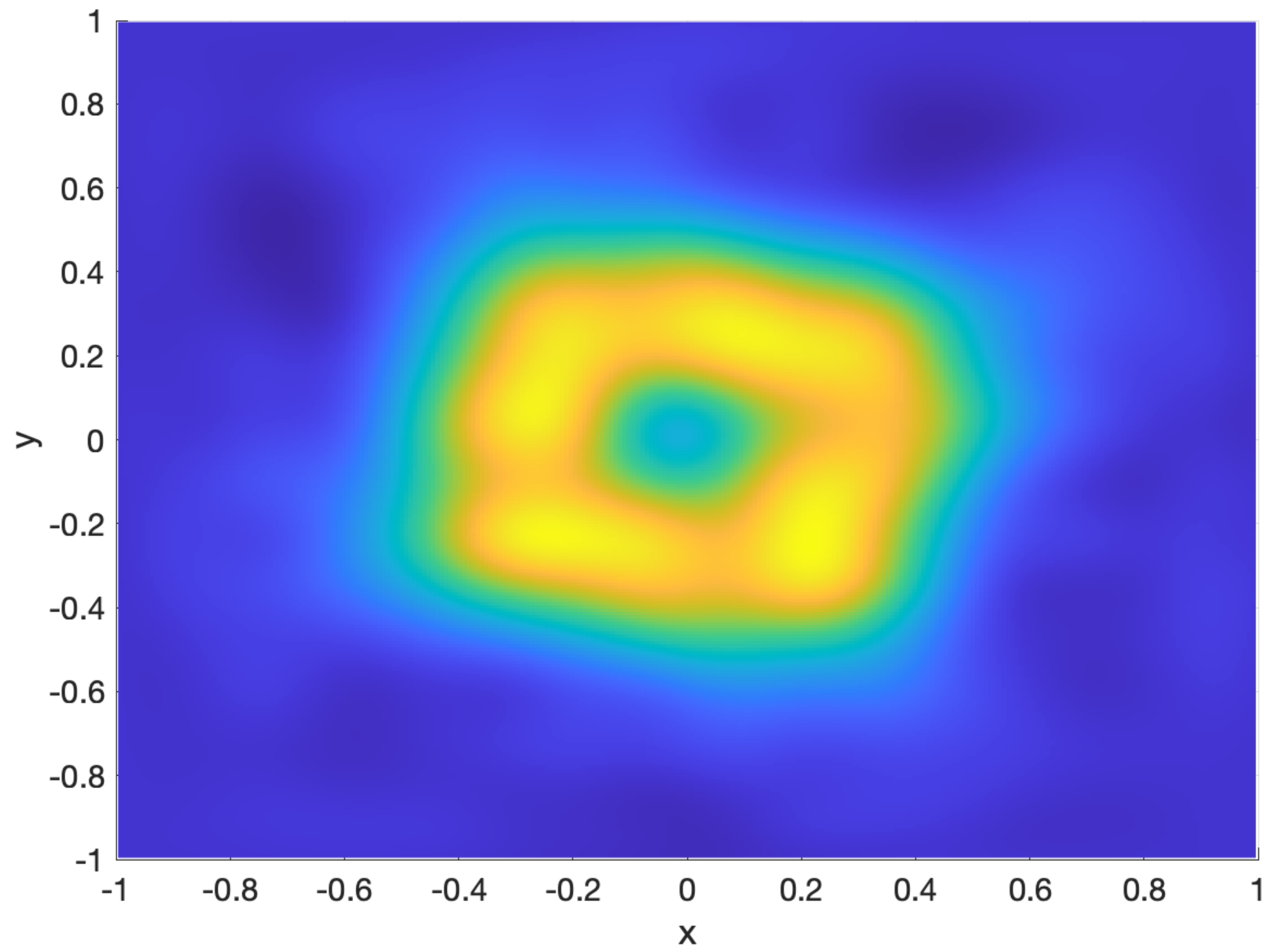}
    \includegraphics[width=0.48\textwidth]{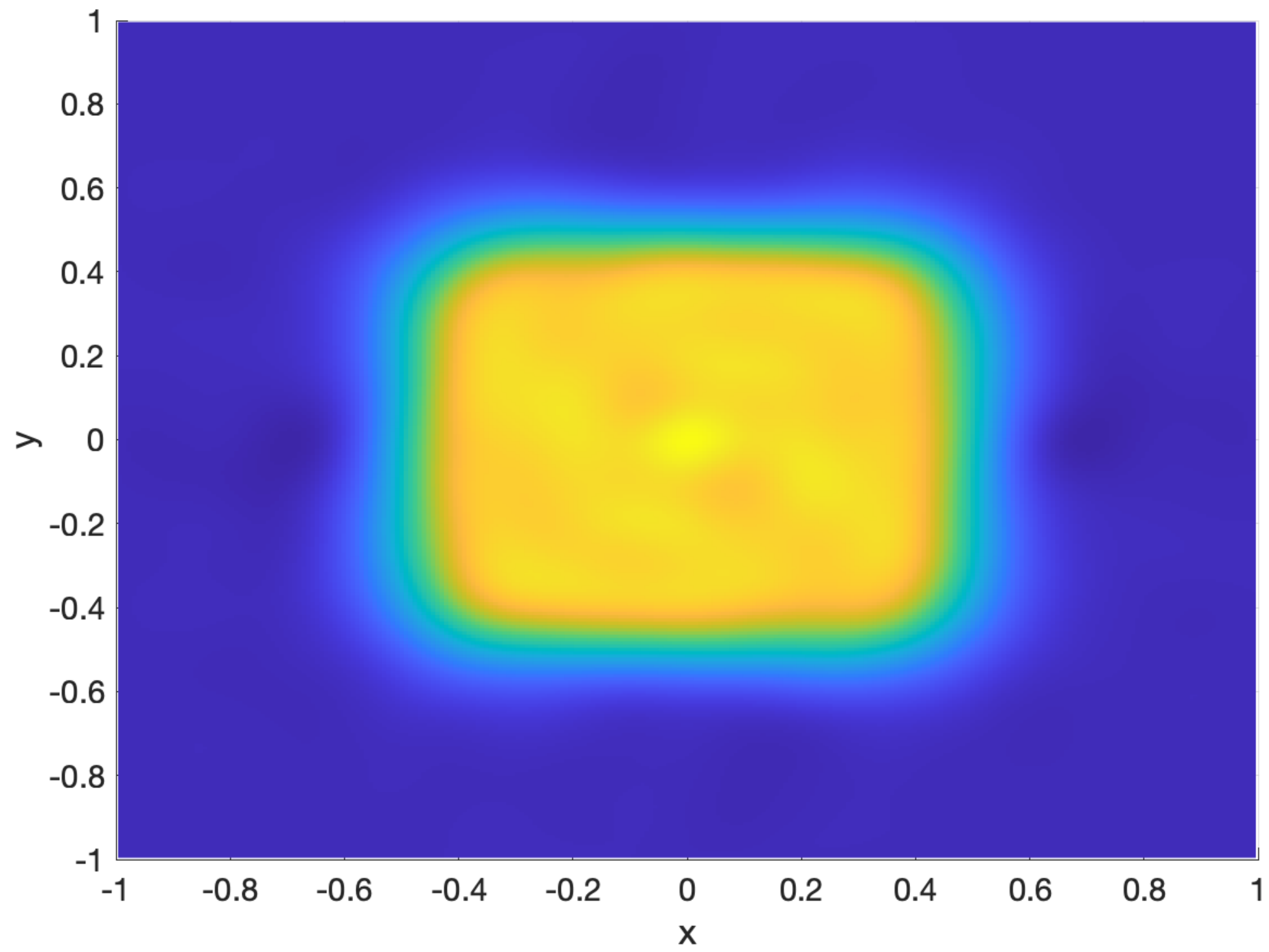}
    
    \includegraphics[width=0.48\textwidth]{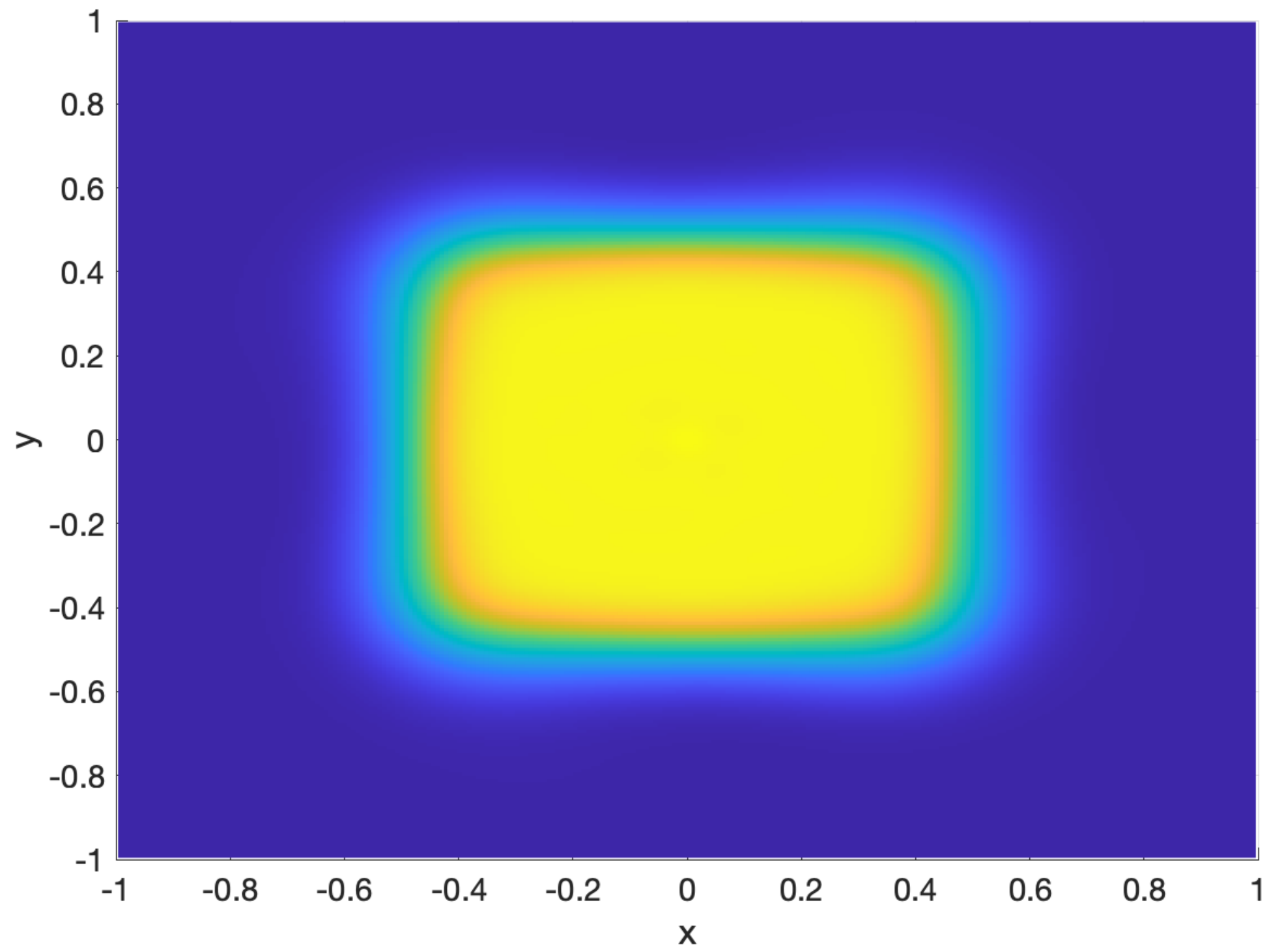}
    \includegraphics[width=0.48\textwidth]{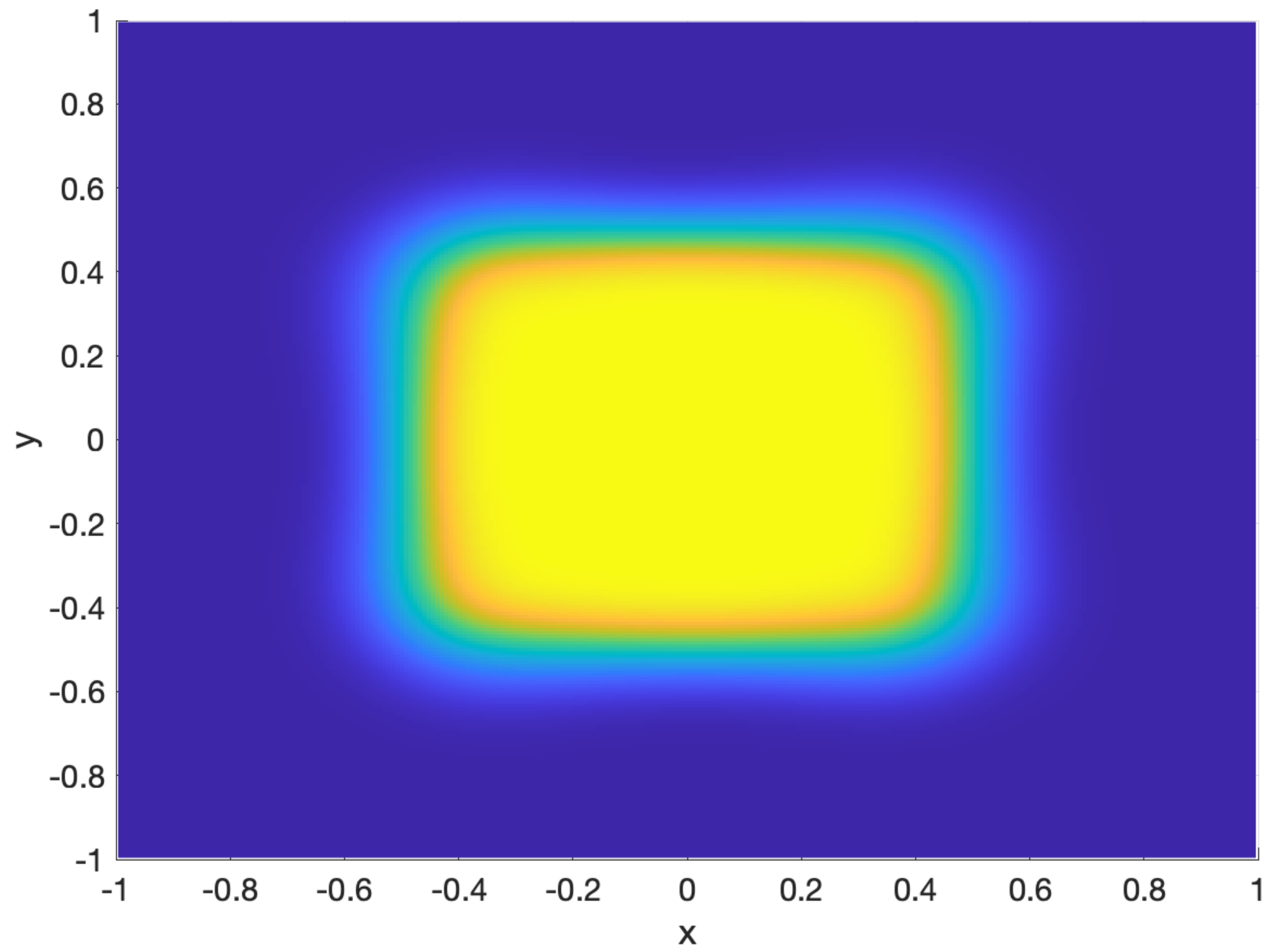}
    \caption{Discrete DLRA solutions for \eqref{eqn:pde} with fixed rank 4, 9, 17, 256 (left to right then top to bottom) at $T=\pi$.  See \Cref{sect:results} for spacial discretization details.  Here the unconventional integrator with forward Euler timestepping was used (see \Cref{alg:uc}) with $\Delta t = 1/4096$. The initial condition \eqref{eqn:solid_body_box_IC} is rank 1 and the rank grows as the box steps through its first $90^{\circ}$ rotation.  It appears that a rank of 17, much lower than the full-rank of 256, is sufficient to accurately approximate the solution.}
    \label{fig:DLRA_advance}
\end{figure}

\subsection{A Rank-Adaptive Unconventional Integrator} \label{subsect:lub_adapt}

A rank-adaptive modification of the unconventional integrator was recently developed in \cite{ceruti2022rank}.  From \Cref{subsec:uc}, the basis at $t^n$, $C_0$ and $D_0$,  and updated basis, $C_1$ and $D_1$, must both be kept in memory at the same time in order to project the solution at $t^n$ into the subspace spanned by the updated basis.  The rank-adaptive unconventional integrator takes advantage of this by modifying lines 5-6 of \Cref{alg:uc} and defining the basis at $t^{n+1}$ to be composed of $\begin{bmatrix}C_0 & C_1\end{bmatrix}$ and $\begin{bmatrix}D_0 & D_1\end{bmatrix}$.    Just like the tangent integrator, since $\begin{bmatrix}C_0 & C_1\end{bmatrix}$ contains $2r$ basis vectors, then the rank would exponentially grow in time; thus culling the rank is required.  The rank-adaptive unconventional integrator with forward Euler timestepping is shown in \Cref{alg:ra_uc}. 

\begin{algorithm}[H] 
\DontPrintSemicolon
\SetKwInOut{Output}{Output~}
\SetKwInOut{Input}{Input~}
\caption{Forward Euler Timestepping with rank-adaptive Unconventional Integrator}\label{alg:ra_uc}
\Input{$\tau > 0$ \tcp*{tolerance}}
\Input{$C_0\in\R^{m\times r_0},S_0\in\R^{r_0\times r_0},D_0\in\R^{n\times r_0}$ \tcp*{ $U^n=C_0S_0D_0^T$}
}
\Output{$C_1\in\R^{m\times r_1},S_1\in\R^{r\times r_1},D_1\in\R^{n\times r_1}$ \tcp*{ $U^{n+1}=C_1S_1D_1^T$}}
\BlankLine
$K_1 = C_0S_0 + \Delta t F(C_0S_0D_0^T)D_0$\;
$C_1 = \textsf{qr}(\begin{bmatrix}C_0 &K_1]\end{bmatrix}$\;
$L_1 = D_0S_0^T + \Delta t F(C_0S_0D_0^T)^TC_0$\;
$D_1 = \textsf{qr}(\begin{bmatrix}D_0 &L_1]\end{bmatrix}$\;
$\tilde{S} = (C_1^TC_0)S_0(D_0^TD_1)$\;
$\tilde{S}_1 = \tilde{S} + \Delta t C_1^TF(C_1\tilde{S}D_1^T)D_1$\;
$[S_C,S_1,S_D]=\text{svd}(\tilde{S_1})$\;
$r_1 = \min\{r\in\N: 1\leq r\leq 2r_0 \text{ and } \sum_{i=r+1}^{2r_0}\sigma_i(S_1)^2 < \tau^2$\;
$C_1 = C_1S_C(:,1\!:\!r_1)$; $D_1 = D_1S_D(:,1\!:\!r_1)$; $S_1 = S_1(1\!:\!r_1,1\!:\!r_1)$  
\end{algorithm}


\subsection{Adaptivity Via Step Truncation} \label{subsect:fra}

Methods that compute the full-rank update, that is storing the evaluation $F(U)$, and then cull in post-processing are often labeled as step truncation methods \cite{dektor2021rank}.  Here we present one such method from \cite{guo2022low} that does not rely on the DLRA framework.  For simplicity, assume $F(U)=AUB^T$ for some $m\times m$ matrix $A$ and $n\times n$ matrix $B$.   If $U^n$ is rank-$r$ with decomposition $U^n=CSD^T$, then 
\begin{align}\label{eqn:fra_1}
    \UFE = CSD^T + \Delta t ACSD^TB^T 
    = \begin{bmatrix}
    C & AC
    \end{bmatrix}
    \begin{bmatrix}
    S & 0 \\ 0 & \Delta tS\end{bmatrix}
    \begin{bmatrix}
    D & BD
    \end{bmatrix}^T.
\end{align}
In terms of the QR factorizations
$
    \begin{bmatrix}
    C & AC
    \end{bmatrix}= C_1R_C $
and
$   \begin{bmatrix}
     D & BD
     \end{bmatrix}= D_1R_D
$,
where $C_1$ and $D_1$ are orthogonal,  \eqref{eqn:fra_1} can be rewritten in the low-rank format
\begin{align}\label{eqn:fra_2}
    \UFE = C_1\tilde{S}D_1^T, \quad \text{ where }\tilde{S}=R_C\begin{bmatrix}
    S & 0 \\ 0 & \Delta tS\end{bmatrix}R_D^T.
\end{align}
Thefactorization in \eqref{eqn:fra_2} shows that $\UFE$ can have rank at most $2r$. 

As with the tangent projection method in \Cref{subsec:tan}, the rank of the forward Euler update will grow exponentially with the number of timesteps, unless culling is applied.  However, unlike the tangent projection method, \eqref{eqn:fra_2} contains no modeling error; thus the only error from the low-rank approximation is made in the culling.  To maintain an $\mathcal{O}(\Delta t)$ method, the tail of the singular values is culled until the portion removed is $\mathcal{O}(\Delta t^2)$.  Extensions to higher order methods can be found in \cite{guo2022low}.  \Cref{alg:fra} shows the rank-adaptive step truncation with forward Euler time-stepping applied to an $N$-separable operator.

\begin{algorithm}[H] 
\DontPrintSemicolon
\SetKwInOut{Output}{Output~}
\SetKwInOut{Input}{Input~}
\caption{Forward Euler Timestepping with rank-adaptive Step Truncation}\label{alg:fra}
\Input{$\tau > 0$ \tcp*{tolerance}}
\Input{$C_0\in\R^{m\times r_0},S_0\in\R^{r_0\times r_0},D_0\in\R^{n\times r_0}$ \tcp*{ $U^n=C_0S_0D_0^T$}
}
\Input{$A_\kappa \in \R^{m\times m}, B_\kappa \in \R^{n\times n}, \kappa = 1,\dots,N$ \tcp*{$F(U)=\sum_{\kappa=1}^N A_\kappa U B_\kappa^T$} 
}
\Output{$C_1\in\R^{m\times r_1},S_1\in\R^{r\times r_1},D_1\in\R^{n\times r_1}$ \tcp*{ $U^{n+1}=C_1S_1D_1^T$}}
\BlankLine
$K_1 = \begin{bmatrix}
 C_0 & A_1C_0 & A_2C_0 & \dots & A_N C_0 
\end{bmatrix}
\in\R^{m\times(N+1)r_0}$
\;
$[C_1,R_C] = \textsf{qr}(K_1)$. \;
$L_1 = \begin{bmatrix}
 D_0 & B_1D_0 & B_2D_0 & \dots & B_N D_0
\end{bmatrix}
\in\R^{n\times(N+1)r_0}$\;
$[D_1,R_D] = \textsf{qr}(L_1)$\;
$\tilde{S} = R_C\text{diag}(S,\Delta t S,\Delta t S,\ldots,\Delta t S)R_D^T$\;
$[S_C,S_1,S_D]=\textsf{svd}(\tilde{S})$ \; 
$r_1 = \min\{r \in \mathbb{N} : 1\leq r\leq (N+1)r_0 \text{ and } \sum_{i=r+1}^{(N+1)r_0}\sigma_i(S_1)^2 < \tau^2\} $\;
$C_1 = C_1S_C(:,1\!:\!r_1)$; $D_1 = D_1S_D(:,1\!:\! r_1)$; $S_1 = S_1(1\!:\! r_1,1\!:\!r_1)$  
\end{algorithm}

An advantage of step truncation is that in a single timestep consistency error created in order to obtain a low-rank solution is only created via a post-processing step that can be easily controlled.  A disadvantage is that the memory footprint for storing the full-rank update is directly proportional to the number of terms in $F$. For a general $N$-separable bilinear form, a rank $(N+1)r$ decomposition needs to be stored even if the solution $U(t^{n+1})$ is close to a rank-$r$ matrix.

\subsection{Predictor-Corrector Rank Adaptivity}\label{subsect:rara}
We now present a new rank-adaptive algorithm that involves two steps.   The predictor step produces a low-rank solution $U_\text{P}^{n+1}$ by any of the numerical DLRA integrators from \Cref{sect:DLRA}.  The corrector step builds a low-rank approximation of the modeling error of the DLRA integrator which we also refer to as the residual $R$: 
\begin{align}\label{eqn:residual_def}
    R = \UFE - U_{\text{P}}^{n+1}.
\end{align}
where 
\begin{equation}\label{eqn:UFE}
    \UFE := U^{n} + \Delta t F(U^n).
\end{equation}
 
We summarize the two steps below; details are provided in \Cref{alg:resid_adapt}. For the remainder of this subsection we assume $U^n = C_0S_0D_0^T$ is a rank-$r$ decomposition of $U^n$ and $U_{\text P}^{n+1}$ has a rank $r_1$ decomposition $U_{\text P}^{n+1}=C_1S_1D_1^T$ where $r_1=r$ if \Cref{alg:uc} or \Cref{alg:proj} is used to compute $U_{\text P}^{n+1}$ and $r_1 = 2r$ if \Cref{alg:tan} is used to compute $U_{\text P}^{n+1}$.

\subsubsection{An randomized algorithm-based error estimator}

The goal is to approximate $\|R\|_F$ and use this approximation to determine whether to add or decrease the rank of the predictor.  Rank can then be added as needed using the principal singular vectors of the singular value decomposition (SVD) for $R$.  Rather than compute the $mn$ individual entries of $R$, we instead seek a low-rank approximation, which we denote $\tilde{R}$, that can be computed efficiently in terms of operation count and memory requirements.  Matrix-free SVD algorithms are memory efficient, but numerical tests of the \textsf{svds} algorithm in MATLAB reveal that the time to compute $\tilde{R}$ is often several orders of magnitude longer than the computation time of the $U_{\text P}^{n+1}$. 

To efficiently compute $\tilde{R}$, we turn to a randomized SVD (R-SVD) approach.  The randomized SVD algorithm shown in \Cref{alg:ra_svd} is sourced  from Algorithms 4.1 and 5.1 of \cite{halko2011finding}.  To summarize, a random sketching matrix $Y \in \mathbb{R}^{n \times (l+p)} $ is created, and then an approximation to the $\Col(R)$ is formed using $RY$.  The parameter $l\in\N$ specifies that $\tilde{R}$ should be close to a rank $l$ approximation of $R$.  The \textit{oversampling parameter} $p\in\N$ is the number of additional samples used in order to increase the probability that $\|R-\tilde{R}\|_F$ is small.

To extract an orthogonal basis for $\Col(RY)$, we use the QR factorization $RY=QP$ where $Q\in\R^{m\times(l+p)}$ is orthogonal and $P \in \R^{(l+p) \times (l+p)}$ is not used.  Then we set $\tilde{R}=QQ^TR$.  Because $Q$ is orthogonal, we only need the SVD of $B^T:=Q^TR\in\R^{(l+p)\times n}$ to form an SVD of $\tilde{R}$.  Specifically if $B = \tilde D \tilde S X^T$ is the SVD of $B$, then
\begin{equation}\label{eq:tildeR}
    \tilde R = QQ^T R = Q B^T = Q X \tilde S \tilde D^T =  \tilde C \tilde S \tilde D^T 
\end{equation}
where $\tilde C = QX$.  Moreover, since $B$ has only $l+p$ columns, a direct SVD can be computed quite cheaply.  This yields a low-rank decomposition of $\tilde{R}$.  

\begin{algorithm}[H] 
\DontPrintSemicolon
\SetKwInOut{Output}{Output~}
\SetKwInOut{Input}{Input~}
\caption{Randomized SVD}\label{alg:ra_svd}
\KwIn{$l\in\N$ \tcp*{Approximation rank}}
\KwIn{$p\in\N$ \tcp*{Oversampling parameter}}
\KwIn{$\RLR$ \tcp*{$\RLR \sim R$ (see \Cref{def:RLR}) }}
\KwOut{$\tilde C\in\R^{m\times l+p},\tilde S\in\R^{l+p\times l+p},\tilde D^{n\times l+p}$ \tcp*{$\tilde{R} = \tilde{C}\tilde{S}\tilde{D}^T$}}
\BlankLine
Construct $Y\in\R^{n\times l+p}$ whose entries are identically and independently distributed Gaussian random variables \;
$[Q,\sim] = \textsf{qr}(\RLR Y)$\;
$B = \RLR^TQ$\;
$[\tilde D,\tilde S,X] = \textsf{svd}(B)$\;
$\tilde C = QX$\;
\end{algorithm}

Since the residual $R$ from \eqref{eqn:residual_def} contains the term $\Delta t F(U^n)$, the most expensive operation in \Cref{alg:ra_svd} is the evaluation of the products $F(U^n)Y$ and $F(U^n)^TQ$.  If $l+p\approx r$, then the computational cost of building $\tilde{R}$ will be similar to any of the DLRA updates in \Cref{sect:DLRA} as all of these methods require the products $F(U^n)C_0$ and $F(U^n)^TD_0$.

The statistical accuracy of this approximation is provided by \cite[Theorem 10.7]{halko2011finding}:

\begin{prop} \label{prop:rand-svd-error}
Set $d=\min\{m,n\}$.  Suppose $\sigma_1\geq\sigma_2\geq\cdots\geq\sigma_d$ are the singular values of $R$.  Suppose $l\geq 2$ and $p\geq 4$; then $
\tilde{R}$ created by \Cref{alg:ra_svd} satisfies the following error bound:
\begin{align}
    \|R-\tilde{R}\|_F \leq \left(1+10\sqrt{\tfrac{l}{p}}\right)\bigg(\sum_{j>  l}\sigma_j^2\bigg)^{1/2}+12\frac{\sqrt{p(l+p)}}{p+1}\sigma_{l+1}
\end{align}
with failure probability at most $7e^{-p}$.
\end{prop}

\begin{proof}
This result follows from setting $t=e$ and $u=\sqrt{2p}$ in \cite[Theorem 10.7]{halko2011finding} and bounding the resulting constants by the nearest whole number.
\end{proof}

Using the R-SVD, an error indicator for the predictor and a subsequent corrector can be constructed.  While the rank of $R$ is not known a priori, it can be bounded when $F$ is an $N$-separable linear operator.
\begin{prop} \label{prop:R-rank}
Let $F$ be an $N$-separable linear operator; let $U^n$ have rank-$r$; and let $R$ be given in \eqref{eqn:residual_def}.  Then the maximum rank of $R$ is either (i) $(N+2)r$ if \Cref{alg:uc} or \Cref{alg:proj} are used to compute $U_{\rm{P}}^{n+1}$ or (ii) $Nr$ if \Cref{alg:tan} is used to compute $U_{\rm{ P}}^{n+1}$.
\end{prop}

\begin{proof}
Let $U^n = C_0S_0D_0^T$ be a rank-$r$ decomposition of $U^n$, Since $F$ is $N$-separable, it has a decomposition (see \eqref{eqn:N-sep-form}) of the form
\begin{equation}
F(U^n) = \sum_{\kappa=1}^N A_\kappa U^n B_{\kappa}^T = \sum_{\kappa=1}^N (A_\kappa C_0) S_0 ( B_{\kappa} D_0)^T
\end{equation}
If \Cref{alg:uc} or \Cref{alg:proj} are used to compute $U_{\text P}^{n+1}$, then $U_\text{P}^{n+1} = C_1S_1D_1^T$ also has rank-$r$, and a direct calculation shows that
\begin{align}\label{eqn:R-rank-1}
\begin{split}
    R &= \UFE - U_{\text P}^{n+1} = C_0S_0D_0^T + \Delta t\sum_{\kappa=1}^N (A_\kappa C_0)S_0(B_\kappa D_0)^T - C_1S_1D_1^T\\
    &= \begin{bmatrix}C_0 & A_1C_0 & \cdots & A_NC_0 & C_1 \end{bmatrix}\text{diag}(S_0,\Delta tS_0,\ldots,\Delta tS_0,-S_1)\begin{bmatrix}D_0 & B_1D_0 & \cdots & B_ND_0 & D_1 \end{bmatrix}^T
\end{split}
\end{align}
has a decomposition that is at most rank-$(N+2)r$.  If \Cref{alg:tan} is used to compute $U_{\text P}^{n+1}$, then
\begin{align}\label{eqn:R-rank-2}
\begin{split}
    R &= \UFE - U_{\text P}^{n+1} \\
    &= U^n + \Delta tF(U^n) - \left(U^n + \Delta t(C_0C_0^TF(U^n) - C_0C_0^TF(U^n)D_0D_0^T + F(U^n)D_0D_0^T)\right) \\
    &= \Delta t (I-C_0C_0^T)F(U^n)(I-D_0D_0^T) \\
    &= \Delta t \sum_{\kappa=1}^N \big((I-C_0C_0^T)A_\kappa C_0\big) S_0 \big( (I-D_0D_0^T)B_\kappa D_0\big)^T \\
    &= \begin{bmatrix}\widetilde{A_1C_0} & \cdots & \widetilde{A_NC_0}\end{bmatrix}\text{diag}(\Delta tS_0,\ldots,\Delta tS_0)\begin{bmatrix}\widetilde{B_1D_0} & \cdots & \widetilde{B_ND_0}\end{bmatrix}^T,
\end{split}
\end{align}
where $\widetilde{A_\kappa C_0} = (I-C_0C_0^T)A_\kappa C_0$ and $\widetilde{B_\kappa D_0} = (I-D_0D_0^T)B_\kappa D_0$.  \eqref{eqn:R-rank-2} implies that $R$ has decomposition that is at most rank $Nr$.  The proof is complete.
\end{proof}
\begin{defn}\label{def:RLR}
We denote by $\RLR$ the low-rank representation of $R$ given by the decomposition in \eqref{eqn:R-rank-1} or \eqref{eqn:R-rank-2} and write $\RLR \sim R$.
\end{defn}
The rank bound from \Cref{prop:R-rank} is used in the estimator $\nu$ of $\|R\|_F$.  We use the $l+p$ singular values of $\tilde{R}$ to approximate the first $l+p$ singular values of $R$ and the smallest singular value of $\tilde{R}$ to approximate the remaining non-trivial singular values of $R$.  This gives 
\begin{align}\label{eqn:nu_ref}
    \|R\|^2_F \approx 
        \|\tilde R\|^2_F + \sum_{l+p+1}^{\theta(r)} \sigma^2_i(R) 
     \approx \|\tilde R\|_F  + \left[\theta(r) - (l +p)\right](\sigma^2_{l+p}\tilde R)  =: \nu^2
\end{align}
where $(\ell + p) \leq \theta(r)$ and
\begin{equation}\label{eqn:theta-def}
    \theta(r) = \begin{cases}
    (N+2)r &\text{ if \Cref{alg:uc} or \Cref{alg:proj} is used to compute $U_{\text P}^{n+1}$}, \\
    Nr &\text{ if \Cref{alg:tan} is used to compute $U_{\text P}^{n+1}$}.
    \end{cases}
\end{equation}
Accounting for the additional singular values in \eqref{eqn:nu_ref} gives a conservative estimate that favors accuracy over memory savings; a less conservative estimate would be to simply use $\|\tilde R\|_F$ to approximation $\|R\|_F$.
\subsubsection{Adding and Removing Rank}

Let $\tilde{R}=\tilde{C}\tilde{S}\tilde{D}^T$ be the low-rank approximation that is constructed by \Cref{alg:ra_svd}.  Rank is added or removed to the predictor $U_{\text P}^{n+1}$ based on the size of the estimator $\nu$, defined in \eqref{eqn:nu_ref}, relative to a prescribed tolerance $\tau > 0$.   There are three cases:

\underline{\textit{Case 1:}} $\nu\geq\tau$.  In this case, rank is added to the predictor by moving rank one components from $\tilde{R}$ to $U_{\text P}^{n+1}$ and updating the corrector:
\begin{alignat}{4}\label{eqn:resid_add}
    U^{n+1}_{(j)} &= U^{n+1}_{(j-1)} + \tilde{C}(:,j)\sigma_{j}(\tilde{R})\tilde{D}(:,j)^T, & \qquad \qquad U^{n+1}_{(0)} &= U_\text{P}^{n+1},\\
    \tilde{R}_{(j)} &= \tilde{R}_{(j-1)} - \tilde{C}(:,j)\sigma_{j}(\tilde{R})\tilde{D}(:,j)^T ,& \qquad \qquad \tilde{R}_{(0)} &= \tilde{R}, \\ 
    \nu_{(j)}^2 &= \nu_{(j-1)}^2 - \sigma_{j}^2(\tilde{R}), &\qquad  \qquad  \nu_{(0)}^2&= \nu^2 ,
\end{alignat}
for $j \in \{ 1, \dots, l\}$. These rank one updates are added until $\nu_{(j)} < \tau$ or until $j=l$ and $\nu_{(l)} \geq \tau$.  In the latter case the predictor is updated and a new residual is constructed:
\begin{equation}
     R \leftarrow R - \tilde{C}(:,1\!:\!l)\tilde{S}(1\!:\!l,1\!:\!l)\tilde{D}(:,1\!:\!l)^T, \qquad 
     U_{n+1}^{\rm{P}} \leftarrow U^{n+1}_{(j)}, \qquad
     \theta(r) \leftarrow \theta(r) - l,
\end{equation}
and the process repeats until the tolerance criterion in satisfied.  The final low-rank update is then culled to remove possibly redundant information introduced by the correction.

\underline{\textit{Case 2:}} $\nu\leq\delta\tau$ for some prescribed constant $\delta \in (0,1)$.  In this case, rank is removed by culling the predictor.  Let $S_1=C_S\hat{S}D_S^T$ be the SVD decomposition of $S_1\in\R^{r_1\times r_1}$.  Thus $U_{\rm{P}}^{n+1}=\hat{C}\hat{S}\hat{D}^T$ where $\hat{C}=C_1C_S$ and $\hat{D}=D_1D_S$.  The algorithm removes components corresponding to the smallest singular values of $S$ from $U_\text{P}^{n+1}$ and appends them to $R$. Specifically let
\begin{align}\label{eqn:nu_neg_def}
    \nu_{(-j)}^2 := \nu^2 + \sum_{\kappa=j+1}^r \hat{S}(\kappa,\kappa)^2
    \quad \text{and} \quad 
    j_* = \min \{j \in \mathbb{N} : 1\leq j\leq r_1 \text{ and } v_{(-j)} \geq \delta \tau \}
\end{align}
Then the updated approximation
\begin{align}
    U^{n+1} = \hat{C}(:,1\!:\!j_*)\hat{S}(1\!:\!j_*,1\!:\!j_*)\hat{D}(:,1\!:\!j_*)^T.
\end{align}
has rank $j_*$ and induces an error 
\begin{equation}
      \| \UFE -  U^{n+1} \| \leq  \| R \|_F + \| U_P^{n+1} - U^{n+1} \|_F \approx \nu^2 +  \sum_{\kappa=j+1}^r \hat{S}(\kappa,\kappa)^2 = \nu_{(-j_*)}.
\end{equation}
In general, $\nu_{(-j)}$ is usually a slight overestimation of the true error $\|R +  (U_P^{n+1} - U^{n+1})\|_F$
because of the conservative estimate of $\nu^2$ by the R-SVD algorithm and the fact that $R$ and $(U_P^{n+1} - U^{n+1})$ are not orthogonal in the Frobenius norm. However, if the projected unconventional integrator from \Cref{subsec:proj} is used, then $R$ and $(U_P^{n+1} - U^{n+1})$ are orthogonal and $\nu_{(-j)}$ is a much better representation of the error.  

\underline{\textit{Case 3:}} $\delta\tau<\nu<\tau$.  In this case, the predictor is considered a sufficiently accurate and efficient approximation.  Thus $U^{n+1} = U_{\rm{P}}^{n+1}$.

\begin{algorithm}[H] 
\SetKwInOut{Output}{Output~}
\SetKwInOut{Input}{Input~}
\DontPrintSemicolon
\caption{predictor-corrector Rank Adaptivity}\label{alg:resid_adapt}
\KwIn{$l\in\N$ \tcp*{Approximation rank for $R$}}
\KwIn{$p\in\N$ \tcp*{Oversampling parameter}}
\KwIn{$\tau>0$ \tcp*{Tolerance}}
\KwIn{$\tau_{\text{cull}}>0$ \tcp*{Redundant Culling tolerance}}
\KwIn{$0<\delta<1$ \tcp*{Rank removal factor}}
\KwIn{$C_0\in\R^{m\times r},S_0\in\R^{r\times r},D_0\in\R^{n\times r}$ \tcp*{$U^n = C_0S_0D_0^T$}}
\KwOut{$C_1\in\R^{m\times r_1},S_1\in\R^{r_1\times r_1},D_1\in\R^{n\times r_1}$ \tcp*{$U^{n+1} = C_1S_1D_1^T$}}
\BlankLine
Create $C_1\in\R^{m\times r_1},S_1\in\R^{r_1\times r_1},D_1\in\R^{n\times r_1}$ \newline 
using Algorithm \ref{alg:uc}, \ref{alg:tan}, or \ref{alg:proj} \tcp*{$U_{\text P}^{n+1} = C_1S_1D_1^T$}
Construct $\RLR \sim R=\UFE-C_1S_1D_1^T$ \tcp*{see \Cref{def:RLR}}
$[\tilde{C},\tilde{S},\tilde{D}]=\textsf{r-svd}(l,p,\RLR)$ using  \Cref{alg:ra_svd} \tcp*{$\tilde{R} = \tilde{C}\tilde{S}\tilde{D}^T$}
$\nu=\sqrt{\sum_{i=1}^{l+p}\tilde{S}_{i,i}^2+(\theta(r)-(l+p))\tilde{S}_{l+p,l+p}^2}$ \tcp*{$\theta(r)$ from \eqref{eqn:theta-def}}
\uIf{$\nu \geq \tau$}{
$c = 0$\;
\While{$\nu \geq \tau$}{
    Find minimum $1\leq j\leq l$ such that $\nu_{(j)}<\tau$ where 
    $\nu_{(j)}=\sqrt{\sum_{i=1}^{l+p}\tilde{S}_{i,i}^2+(\theta(r)-(l+p+c+j))\tilde{S}_{l+p,l+p}^2}$\;
    \uIf{no such $j$ exists}{
        $C_1 = \begin{bmatrix}C_1 &\tilde{C}(:,1\!:\!l)\end{bmatrix}$, $D_1 = \begin{bmatrix}D_1 &\tilde{D}(:,1\!:\!l)\end{bmatrix}$; $S_1 = \text{diag}(S_1,\tilde{S}(1\!:\!l,1\!:\!l))$\;
        Construct $\RLR \sim R = \UFE-C_1S_1D_1^T$ \;
        $[\tilde{C},\tilde{S},\tilde{D}]=\textsf{r-svd}(l,p,\RLR)$ using \Cref{alg:ra_svd}\;
        $c = c + l$\;
        $\nu=\sqrt{\sum_{i=1}^{l+p}\tilde{S}_{i,i}^2+(\theta(r)-(l+p+c))\tilde{S}_{l+p,l+p}^2}$
    }
    \Else{
        $C_1 = \begin{bmatrix}C_1&\tilde{C}(:,1\!:\!j)\end{bmatrix}$, $D_1 = \begin{bmatrix}D_1&\tilde{D}(:,1\!:\!j)\end{bmatrix}$; $S_1 = \text{diag}(S_1,\tilde{S}(1\!:\!j,1\!:\!j))$
    }
}
    $[C_1,R_C] = \textsf{qr}(C_1); [D_1,R_D] = \textsf{qr}(D_1)$\;
    $S_1 = R_CS_1R_D^T$\;
    $[S_C,S_1,S_D] = \textsf{svd}(S_1)$; $r_1 = \text{size}(S_1)$ \tcp*{$\textrm{size}(S_1)$ is \# of columns of $S_1$}
    $r_1 = \min\{j\in\N:1\leq j\leq r_1 \text{ and } \sigma_j(S_1) < \tau_{\text{cull}}\}$\;
    $C_1 = C_1C_S(:,1\!:\!r_1)$, $D_1 = D_1D_S(:,1\!:\!r_1)$, $S_1 = S_1(1\!:\!r_1,1\!:\!r_1)$\;
}
\uElseIf{$\nu\leq \delta\tau$}{
    $[C_S,S_1,D_S] = \textsf{svd}(S_1)$\;
    Find maximum $1\leq j\leq r$ such that $\nu^2+\sum_{\kappa=j+1}^r\sigma_{\kappa}(S_1)^2> (\delta\tau)^2$\;
    $C_1 = C_1C_S(:,1\!:\!j)$, $D_1 = D_1D_S(:,1\!:\!j)$\;
    $S_1 = S_1(1\!:\!j,1\!:\!j)$
}
\Else{
    \Continue
}
\end{algorithm}

\subsection{Memory Comparisons}
In this subsection, we give a comparison of the memory requirements to run each of the adaptive algorithms presented above:
the \textit{rank-adaptive unconventional integrator} (RAUC) algorithm, the \textit{rank-adaptive step truncation} (RAST) algorithm and the new (RAPC) \textit{rank-adaptive predictor-corrector} algorithm.  Each of these algorithms primarily involve function evaluations which have been made memory efficient by \Cref{prop:matrix_evals} or modern QR and SVD decompositions that are also memory efficient.  Thus we gauge the memory requirements of each adaptive algorithm by the max of the sum of ranks of all objects held in memory at one time.   As an example, for the unconventional integrator (\Cref{alg:uc}), the rank-$r$ matrix $U^n$ must be stored at the same time as the rank-$r$ matrix $C_1D_1^T$; thus we list the memory requirements as $r+r=2r$.  Using this metric, \Cref{tab:mem_comp} shows the results below.

\begin{table}[h!t]
    \centering
    \begin{tabular}{c|c|c|c}\hline
       Method & Section & Algorithm & Memory  \\ \hline
        RAUC \cite{ceruti2022rank} & \Cref{subsect:lub_adapt} & \Cref{alg:ra_uc} & $2r$ \\ 
        RAST \cite{guo2022low}& \Cref{subsect:fra} &\Cref{alg:fra} & $(N+1)r$ \\
        RAPC & \Cref{subsect:rara} & \Cref{alg:resid_adapt} & $r+\max\{r_1,r\}+l+p$ \\ \hline
    \end{tabular}
    \caption{Memory comparisons of the rank-adaptive algorithms list in \Cref{sect:adapt}.  Here $r$ is the rank of $U^n$, $F$ is $N$-separable, $r_1$ is the final rank of the $U^{n+1}$, and the parameters $l$ and $p$ are given in \Cref{alg:ra_svd}. }
    \label{tab:mem_comp}
\end{table}

The RAUC algorithm is the most memory efficient; this is expected, as it reuses information obtained through the DLRA process.  However, because of this convergence is entirely dependent on the modeling error introduced from the DLRA approximation.  The RAST algorithm uses the most memory per run due to storing the full-rank update in a low memory fashion; this allows for accurate approximations to be guaranteed, but use excessive memory in cases were DLRA methods work well.  The RAPC algorithm sits in between the these two.  The maximum memory is not known a-priori in that it is dependent on the tolerance given and the oversampling parameter $p$ used in the R-SVD, but by using the SVD to obtain new basis vectors for $U^{n+1}$ the algorithm seeks to prioritize memory as much as possible.

\section{Numerical Results}\label{sect:results}

In this section we compare the three adaptive strategies listed in \Cref{tab:mem_comp}:  the rank-adaptive unconventional integrator (RAUC) algorithm, the rank-adaptive step truncation (RAST) algorithm and the new (RAPC) rank-adaptive predictor
corrector algorithm.  The RAPC algorithm is applied to all the three DLRA integrators in \Cref{sect:DLRA}: the unconventional integrator (RAPC-UC) in \Cref{subsec:uc}, the tangent projector integrator (RAPC-Tan) in \Cref{subsec:tan}, and the projected unconventional integrator (RAPC-Proj) in \Cref{subsec:proj}, for a total of five distinct methods.

We apply the five methods to three test problems.  The first problem is the 2D solid-body rotation problem from \eqref{eqn:num_prob}; these results are presented in \Cref{test:solid_body}.  A large fraction of the results are dedicated this problem, which assess the ability of the adaptive methods to match the rank of the true solution as it oscillates in time. The second and third test are modifications of solid-body rotation problem.  The problem in \Cref{subsect:rad} introduces a relaxation problem common to radiation transport and gauges how the rank-adaptive methods perform when the rank does not change after a short time.  The problem in \Cref{subsect:sources} also contains a relaxation along with localized sources.   

The RAUC and RAST algorithms require only a tolerance $\tau$ to run while all of the RAPC algorithms require a tolerance $\tau$, a culling fraction $\delta$, and a redundant culling parameter $\tau_{\rm{cull}}$. In our tests RAPC-Tan is much more apt to decrease rank than the other two RAPC algorithms, so we choose $\delta = 0.5$ for RAPC-Tan and $\delta = 0.7$ for RAPC-UC and RAPC-Proj to make up for this difference.  In all tests we set $\tau_{\rm{cull}}=10^{-14}$.   For the randomized SVD in \Cref{alg:resid_adapt}, we use $l=3$ in all cases, but the oversampling parameter $p$ is problem-specific.  All of the algorithms are implemented and run using MATLAB \cite{MATLAB:2020}.   Random vectors are drawn using MATLAB's built-in $\texttt{randn}$ function.  Finally $\W=[-1,1]^2$ is used for all problems and we set the basis $\{\phi_j(x)\}$ for $V_{x,h}$ to be the shifted localized Legendre polynomials on every interval $T\in\mT_{x,h}$ that are scaled to be orthonormal in $L^2(\W_x)$.  The same basis is used for $V_{y,h}$.

\subsection{Test Problem 1 - 2D solid-body Rotation}\label{test:solid_body}

The equation for solid-body is given in \eqref{eqn:pde} with initial condition 
\begin{equation}\label{eqn:solid_body_box_IC}
    u_0(x,y) = \chi_{\{|x|<0.5\}}(x)\chi_{\{|y|<0.5\}}(y),
\end{equation}
where $\chi_E$ is the indicator function on a set $E$.  The DG discretization and low-rank structure for this problem is discussed extensively in \Cref{subsect:solid-body} and \Cref{subsect:lr_features}. We choose a uniform discretization with $h_x=h_y=1/128$ and $k_x=k_y=0$ so that the degrees of freedom in each variable is 256.  The oversampling parameter $p$ is set to 7. 

Because $u_0$ is separable, the $L^2$ projection of $u_0$ onto $V_h$ has a rank of one.  However, as the solution evolves via a counter-clockwise rotation, the numerical rank  of the true solution oscillates with its numerical rank, i.e., the total number of singular values above some tolerance, increasing and decreasing back to rank 1 at $t=\pi/2$. \Cref{fig:full_rank_plot} shows the numerical rank of the discrete full-rank forward Euler update as a function of time.   From $t=\pi/2$ to $t=\pi$, the rank does not increase as much as the first $90$ degree rotation.  This is because the numerical dissipation in the system smooths out the solution and forces the singular values to decay faster.  In order to measure rates of convergence in time, we set $\omega=1/256$ to be our ``base timestep'' and take $\Delta t$ to be some fraction of $\omega$.  

\begin{figure}
    \centering
    \begin{subfigure}[t]{0.45\textwidth}
    \includegraphics[width=0.9\textwidth]{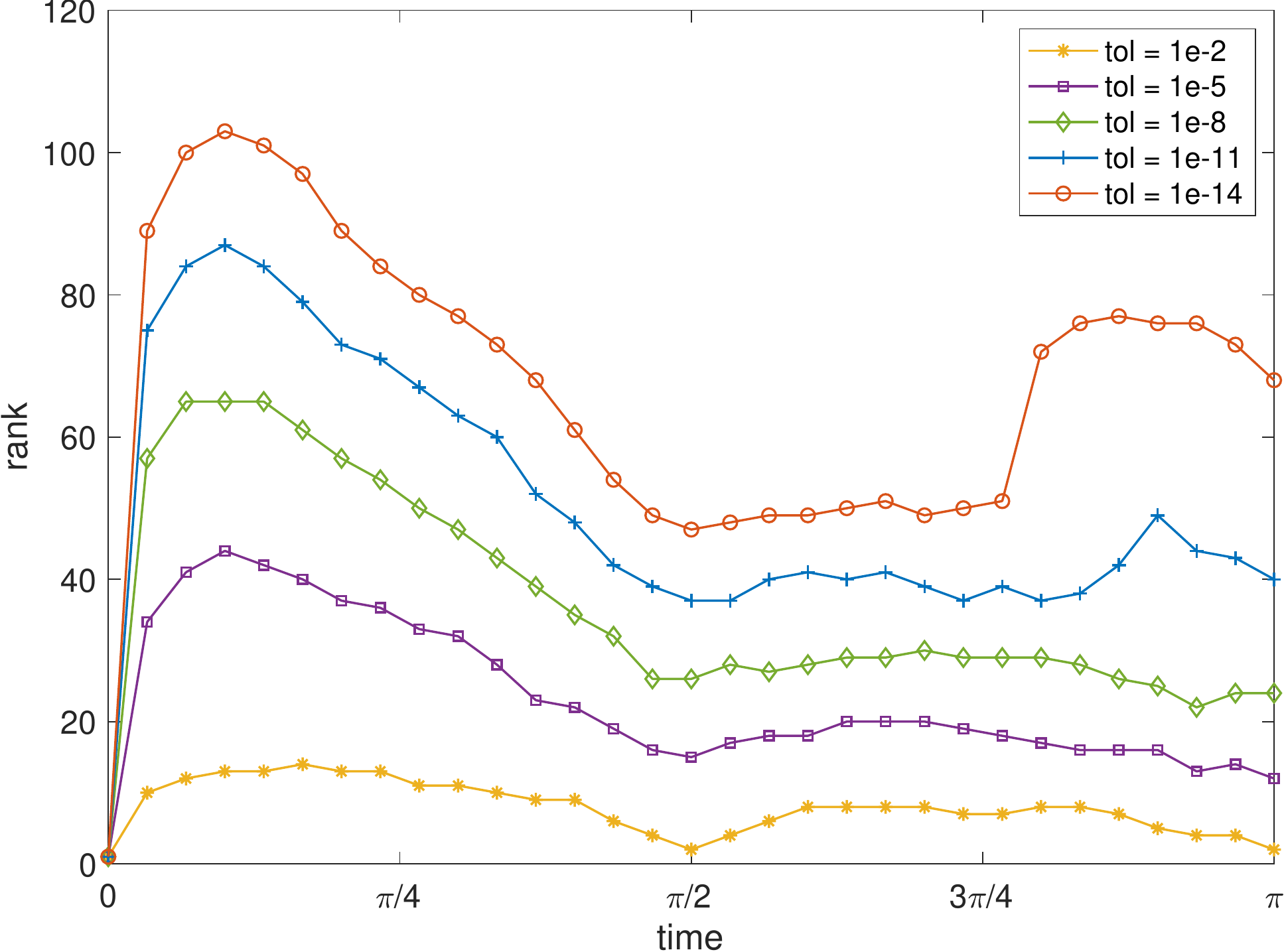}
    \caption{Numerical rank of the forward Euler full-rank iteration of \eqref{eqn:num_prob} with $\Delta t=1/4096$ as a function of time.  Given a tolerance \texttt{tol}, we define the numerical rank of the discrete function by the number of singular values larger than \texttt{tol}.  
    }
    \label{fig:full_rank_plot}
    \end{subfigure}\hfill
    \begin{subfigure}[t]{0.45\textwidth}
    \includegraphics[width=0.9\textwidth]{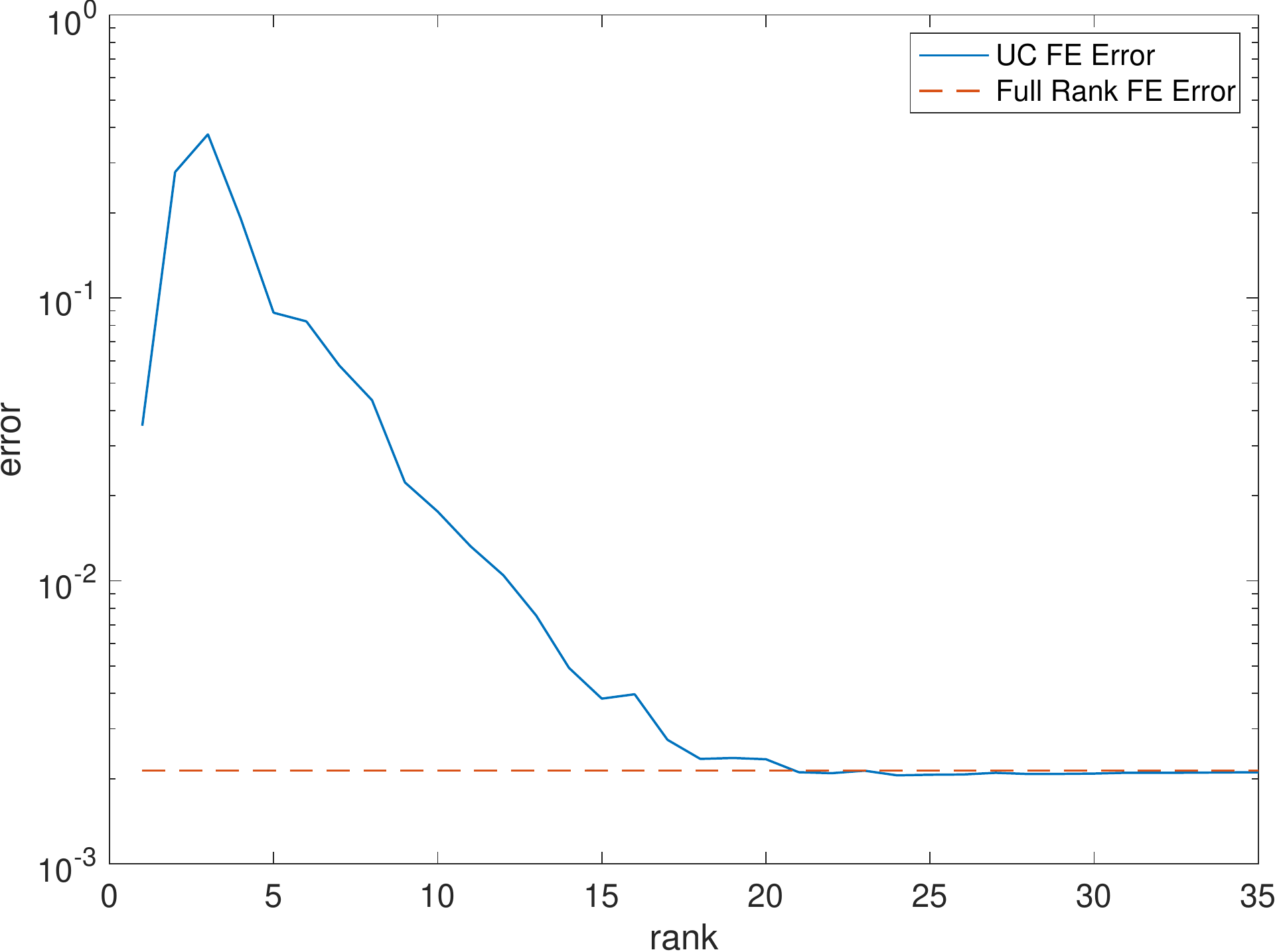}
    \caption{Errors of the unconventional integrator with forward Euler timestepping (see \Cref{alg:uc}) applied to \eqref{eqn:num_prob} at time $T=\pi$ as a function of the initial rank-$r$.  The errors are created by comparison against the full-rank SSP-RK3 approximation with $\Delta t=1/4096$ for all runs.}
    \label{fig:uc_rank_plot}
    \end{subfigure}
    \caption{}
\end{figure} 

To gauge how much rank is expected to resolve the discrete solution up to timestepping error, we plot the error of a rank-$r$ unconventional integrator at time $t=\pi$ against the full-rank SSP-RK3 integrator.  The initial condition in this rank is a rank-$r$ truncation of the SVD of the projected initial condition.  One important detail to observe is that that all of the DLRA integrators in \Cref{sect:DLRA} are dependent on the low-rank components $C,S,D$ rather than the product $CSD^T$; because of this, even though the second through $r$-th singular values are near machine epsilon, all of the singular vectors are equally weighted in creating the updated basis vectors in time (see \Cref{alg:uc}).  Thus the modelling error is dependent on the initial basis vectors $C_0$ and $D_0$ of $U^n$ -- even if that information is essentially not used to construct $U^n$ when the product $U^n=C_0S_0D_0^T$ is formed.   This allows DLRA integrators to perform better with a larger starting rank and can be seen in \Cref{fig:uc_rank_plot}.  As the initial rank grows, the errors decrease until the error is saturated by the forward Euler timestepping error.  The dotted line in \Cref{fig:uc_rank_plot} is the full-rank forward Euler error -- $2.138\times 10^{-3}$ whose error is only dependent on the temporal discretization.  An initial rank of at least 21 is needed to remove the modelling error and produce a DLRA solution that is as good as a full rank solution. 

\subsubsection{Isolating Modelling Error}\label{subsect:test1}

\begin{figure}[ht]
    \centering
    \includegraphics[width=0.49\textwidth]{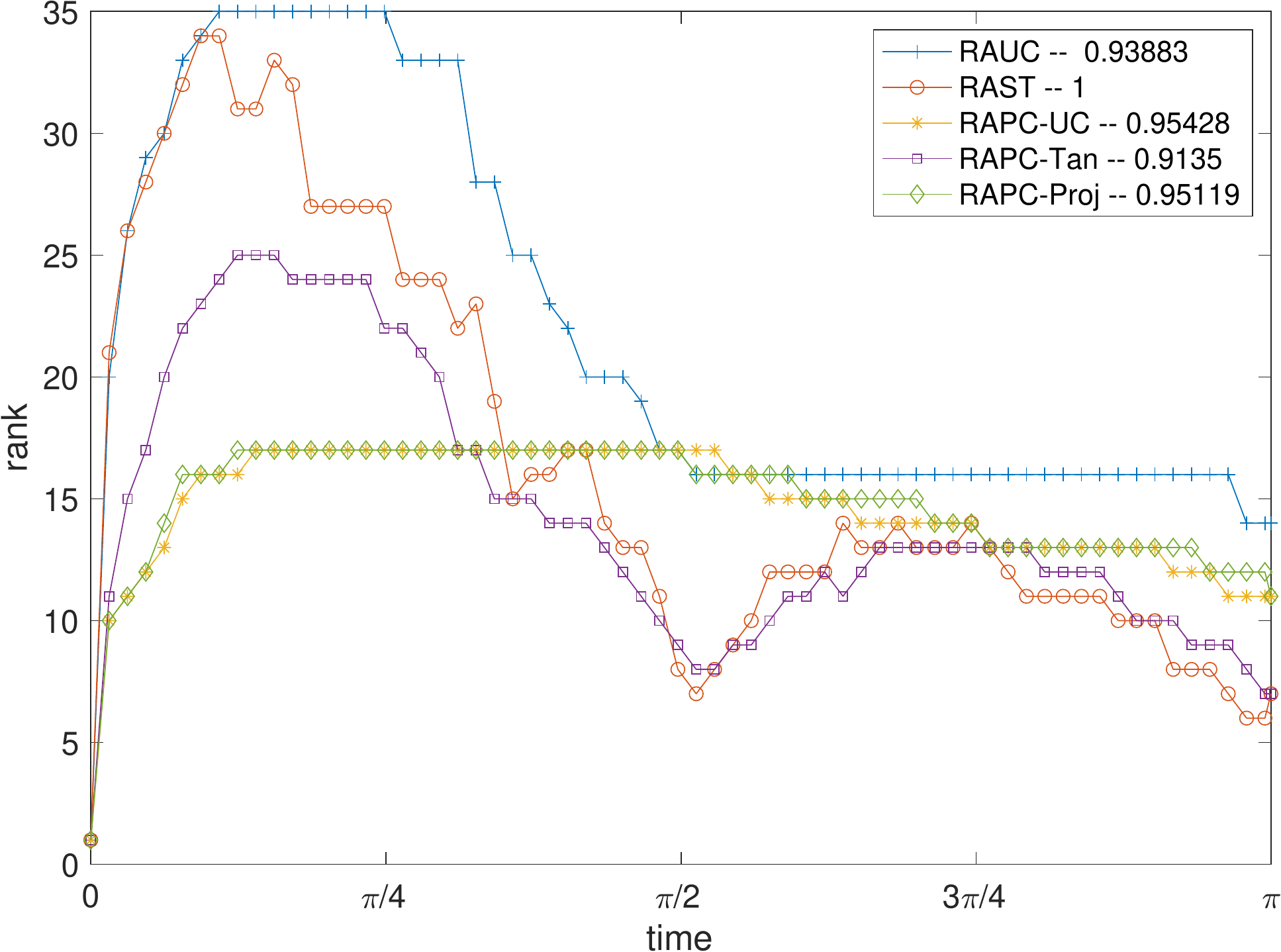}
    \includegraphics[width=0.49\textwidth]{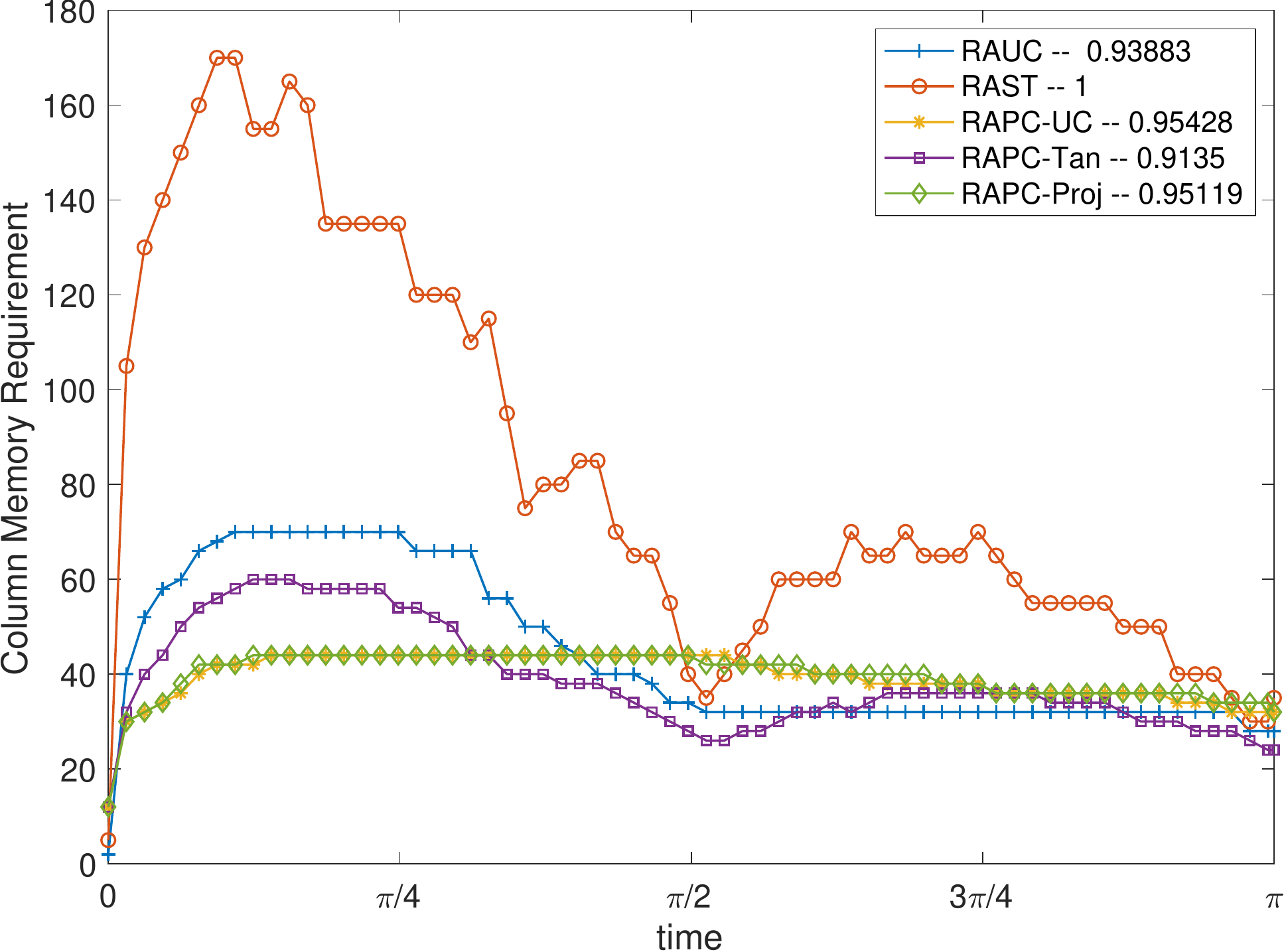}
    \caption{
    Rank (left) and memory requirements (right) of each of the five methods over time for the 2D solid-body rotatin problem (\Cref{test:solid_body}) with $\Delta t=\omega/16$ where $\omega=1/256$. The tolerance $\tau$ for each method is set so that the error at time $T=\pi$ against the full-rank SSP-RK3 method is near $3\times 10^{-3}$ and can be found in \Cref{tab:tol_test1}, and the memory footprint is calculated using \Cref{tab:mem_comp}. For readability, the data is plotted at every 200 timesteps.  The legend labels provide the error of the method, normalized with respect to the error of the RAST algorithm.
    }
    \label{fig:test_1_fig_1}
\end{figure}

\begin{table}[h]
    \centering
    \begin{tabular}{c|c c c c c} \hline
 Method & RAUC & RAST & RAPC-UC & RAPC-Tan & RAPC-Proj \\ \hline
 $\tau$ & $25\Delta t^2$ & $19\Delta t^2$ & $880\Delta t^2$ & $11255\Delta t^2$ & $857\Delta t^2$ \\ 
 & 1.4901e-6 & 1.1325e-6 & 5.2452e-5 & 6.7085e-4 & 5.1081e-5 \\ \hline
Error & 3.039e-3 & 3.237e-3 & 3.089e-3 & 2.957e-3 & 3.079e-3 \\ \hhline{=|=====}
$\tau$ & $\Delta t^2$ & $5\Delta t^2$ & $400\Delta t^2$ & $500\Delta t^2$ & $200\Delta t^2$ \\
& 5.9605e-8 & 2.9802e-7 & 2.3842e-5 & 2.9802e-5 & 1.1921e-5 \\ \hline
Error & 2.592e-3 & 2.108e-3 & 2.1348e-3 & 2.1414e-3 & 2.1642e-3 \\ \hline
\end{tabular}
    \caption{
    Tolerances and Errors for all five rank-adaptive methods in \Cref{subsect:test1}.  The errors are computed at $T=\pi$ against the full-rank SSP-RK3 discretization. $\Delta t = \omega/16$ where $\omega=1/256$.
    }
    \label{tab:tol_test1}
\end{table}

For the first test we wish to compare the rank of the five methods when modelling error is the dominant source of error in the problem.  In this test each of the low-rank solutions are less accurate than the full rank solution.
We set $\Delta t=\omega/16$ and run each method with a tolerance $\tau$ to achieve an error as close to $3\times 10^{-3}$ when compared against the SSP-RK3 approximation.  The tolerances are listed in \Cref{tab:tol_test1} and were chosen by trial and error.  The error of the full-rank forward Euler iteration against the SSP-RK3 iteration is $2.138\times 10^{-3}$ so our choice of $3\times 10^{-3}$ is to judge the modelling error by the low-rank approximation rather than the temporal discretization error.  The errors for each method are given in \Cref{tab:tol_test1}, and  
\Cref{fig:test_1_fig_1} shows the rank of each of the five methods as a function of time.  The maximum rank of each of the methods far less than the full-rank forward Euler update (\Cref{fig:full_rank_plot}).  
However, each of the RAPC methods do much better than other two methods. 
Both the RAPC-UC and RAPC-Proj methods need about half of the rank of the RAUC implementation in order to achieve the same error.  RAPC-UC and RAPC-Proj have a maximum rank of $r=17$ and an error of $3.039\times 10^{-3}$ (see \Cref{tab:tol_test1}).  From \Cref{fig:uc_rank_plot}, the error of $3.039\times 10^{-3}$ fits nicely between errors of the fixed-rank unconventional integrator for the rank-16 and rank-17 runs.  
The RAPC-UC and RAPC-Proj solutions capture the same accuracy as a fixed rank-17 unconventional integrator and these methods only keep a rank 17 solution for less than half of the simulation.   For a majority of the time, the RAPC-UC and RAPC-Proj solutions are below a rank of 17.

The RAST and RAPC-Tan methods best capture the temporal oscillations in rank while the rank of the methods based on the unconventional integrator (RAUC, RAPC-UC, and RAPC-Proj) appear to flatten out over the second $90^\circ$ rotation.  
\subsubsection{Resolving Modelling Error}\label{subsect:test2}

\begin{figure}[ht]
    \centering
    \includegraphics[width=0.49\textwidth]{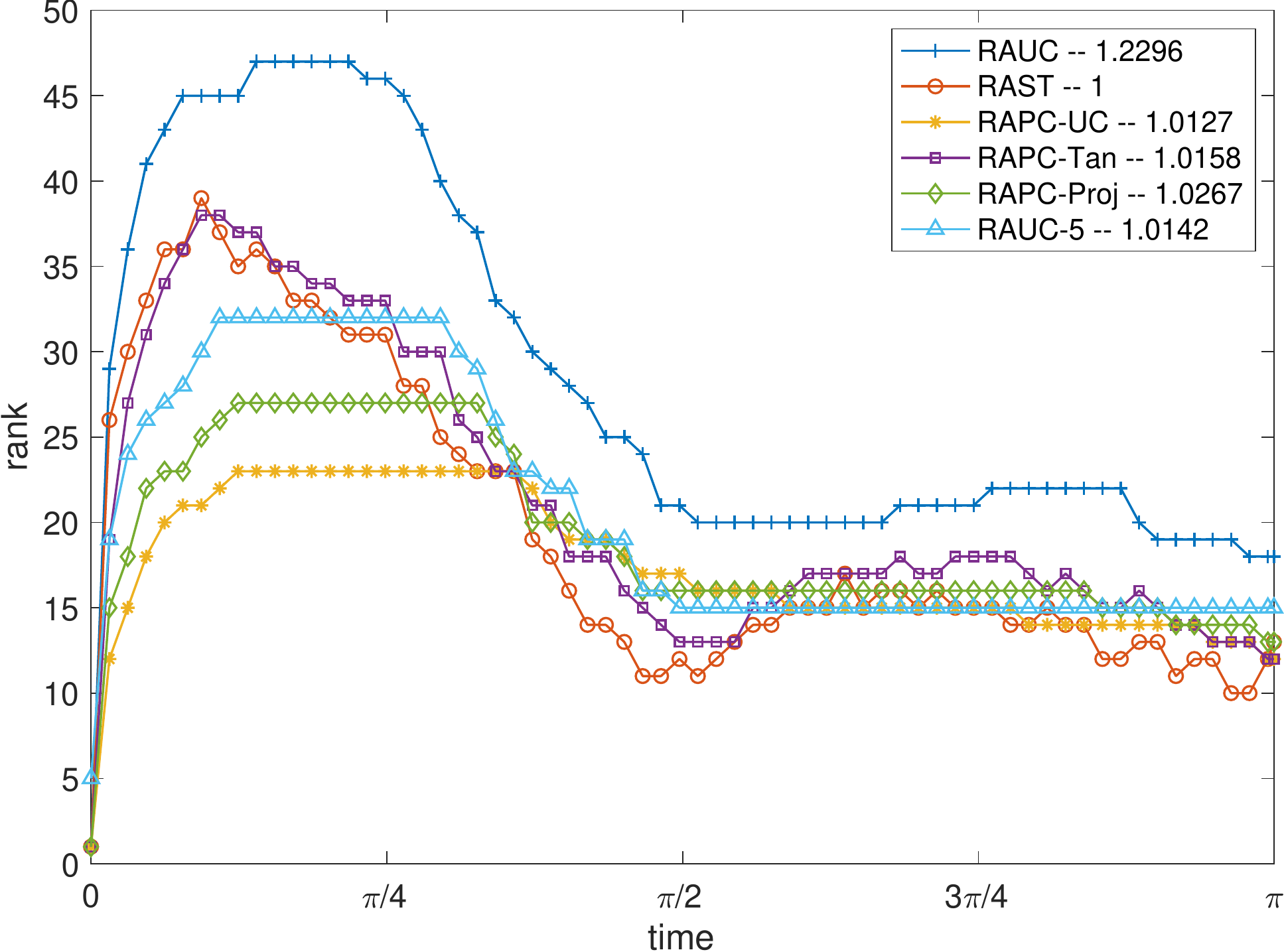}
    \includegraphics[width=0.49\textwidth]{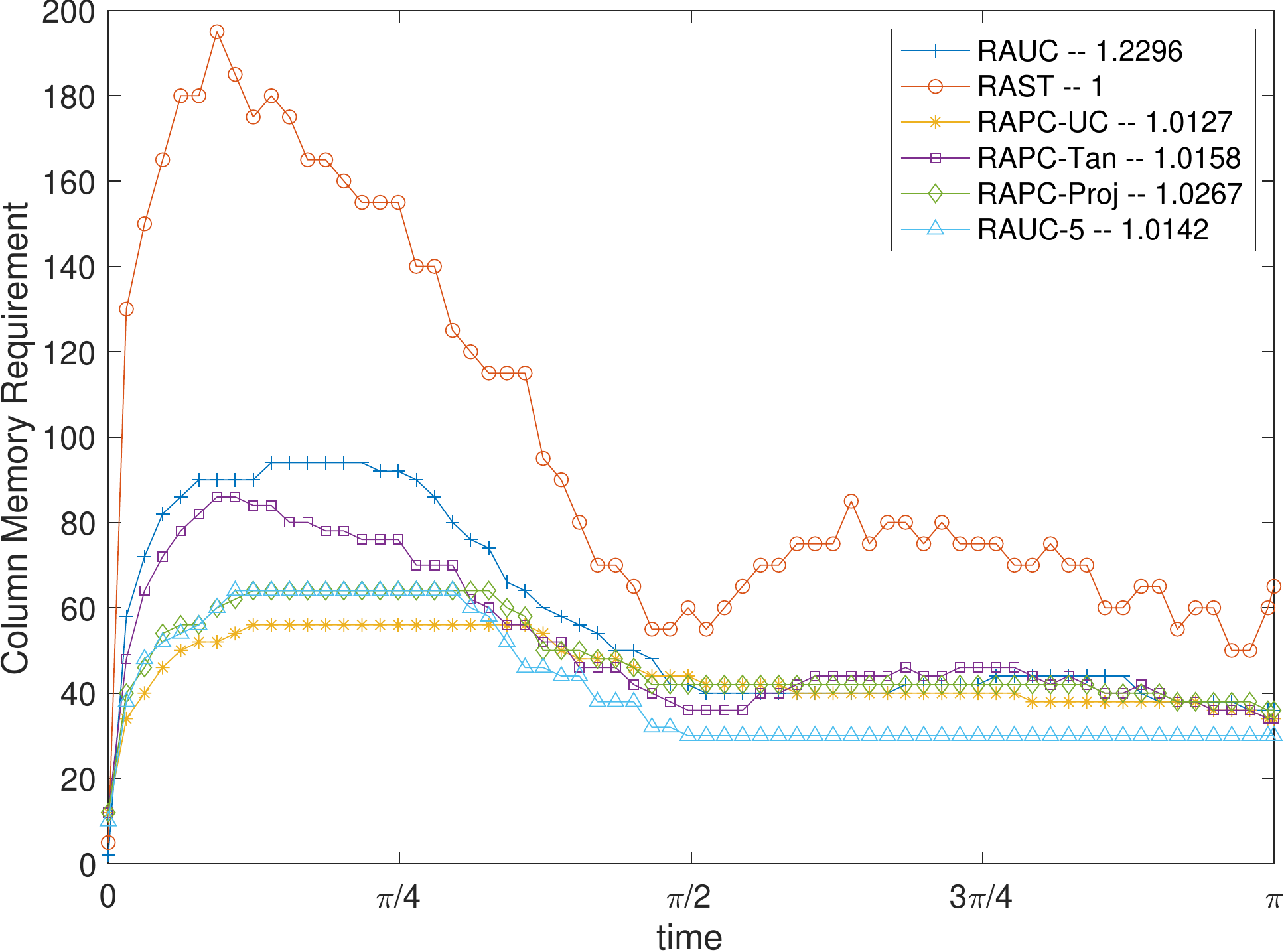}
    \caption{Rank (left) and memory requirements (right) of each of the five methods over time for the 2D solid-body rotation problem (\Cref{test:solid_body}) with $\Delta t=\omega/16$ where $\omega=1/256$. The tolerance $\tau$ for each method was set so that the error at time $T=\pi$ against the full-rank SSP-RK3 method was near $3\times 10^{-3}$ and can be found in \Cref{tab:tol_test1}, and the memory footprint was calculated using \Cref{tab:mem_comp}.   For readability, the data is plotted at every 200 timesteps.  The legend labels provide the error of the method, normalized with respect to the error of the RAST algorithm.
    }
    \label{fig:test_1_fig_2}
\end{figure}

We now run the same test but tune to the largest tolerances, again by trial and error, so that the error of each run is as close to the temporal discretization error, $2.138\times 10^{-3}$, as possible.  Thereby the modelling error is resolved meaning that each of the low-rank solutions is a good of an approximation as the full rank solution.  The rank of the runs are shown in \Cref{fig:test_1_fig_2}.   Similar to the previous run, the RAPC-UC and RAPC-Proj perform the best with the RAPC-UC needing at maximum a rank of 23 for approximately a quarter of the run.  \Cref{fig:uc_rank_plot} shows that this is close to the rank of 21 that was required to capture this error in the non rank-adaptive case.  Also after $t=3\pi/8$, a rank of 23 is not required to resolve the modelling error.  

The RAPC-Tan and RAST integrators perform similarly in rank, but the RAPC-Tan integrator requires storage of only $2r+10$ vectors in each direction (see \Cref{tab:mem_comp}) while the RAST requires $5r$ vectors in each direction.  When $r=38$, this difference is quite large as evidenced in the memory plot of \Cref{fig:test_1_fig_2}; thus the RAPC-Tan integrator is much more memory efficient than the RAST integrator.  RAST is practical when the rank of the solution is guaranteed to be small; indeed, the authors in \cite{guo2022low} propose a different discretization of \eqref{eqn:pde} where the characteristics are instead approximated which leaves the numerical solution low-rank for all time.  

Finally, the rank-adaptive unconventional integrator \cite{ceruti2022rank} (RAUC) described in \Cref{subsect:lub_adapt} was not able to capture the required error for any reasonable rank -- any tighter tolerances did not produce a more accurate result. This is because the method can only gain information from the current and updated basis created by the unconventional integrator.  If this space is not rich enough, then the method cannot resolve the modelling error.  After the first timestep, the numerical rank of the full-rank forward Euler update is 5 and the other four methods correctly have a rank of 5.  However, because the initial condition is only rank 1, the RAUC integrator can only have a maximum of two basis vectors in the updated components $C$ and $D$.  Therefore there is no way to resolve the modelling error.  A simply remedy is to increase the rank of the initial condition to five (using four additional singular values are near machine epsilon, with corresponding singular vectors that result from MATLAB's \texttt{svd} algorithm).  In this case, the algorithm, denoted as RAUC-5 in \Cref{fig:test_1_fig_2}, with $\tau = 50\Delta t^2$ resolves the modelling error with an error of $2.138\times 10^{-3}$.  While the rank of RAUC-5 ($r=32$) is larger than the RAPC-Proj run ($r=27$), the memory requirements are the same because the of the additional memory overhead of the R-SVD corrector estimation call in the RAPC algorithm.  We note that this fix to the unconventional integrator is not guaranteed to work and shows that the performance of the DLRA integrators is highly dependent on the initial basis chosen and not on the initial condition itself.

Similar to the last test, RAPC-UC algorithm is the most memory efficient in the sense that the maximum memory used by RAPC-UC is the least of the five runs. 

\subsubsection{Rates of Convergence}\label{subsect:test_rates}

Here we test that the RAPC methods and the other adaptive low-rank methods are able to recover an $\mathcal{O}(\Delta t)$ method.  This is achieved by setting $\tau = \mathcal{O}(\Delta t^2)$ in order to guarantee that the update is second-order accurate locally and first-order accurate globally.  The results are given in \Cref{fig:rates} and they show setting $\tau=\mathcal{O}(\Delta t^2)$ is sufficient to achieve a second order method.  Tests not included show that a tolerance of order $\Delta t$ does not yield a globally first order accurate method.

\begin{figure}[ht]
    \centering
    \includegraphics[width=0.49\textwidth]{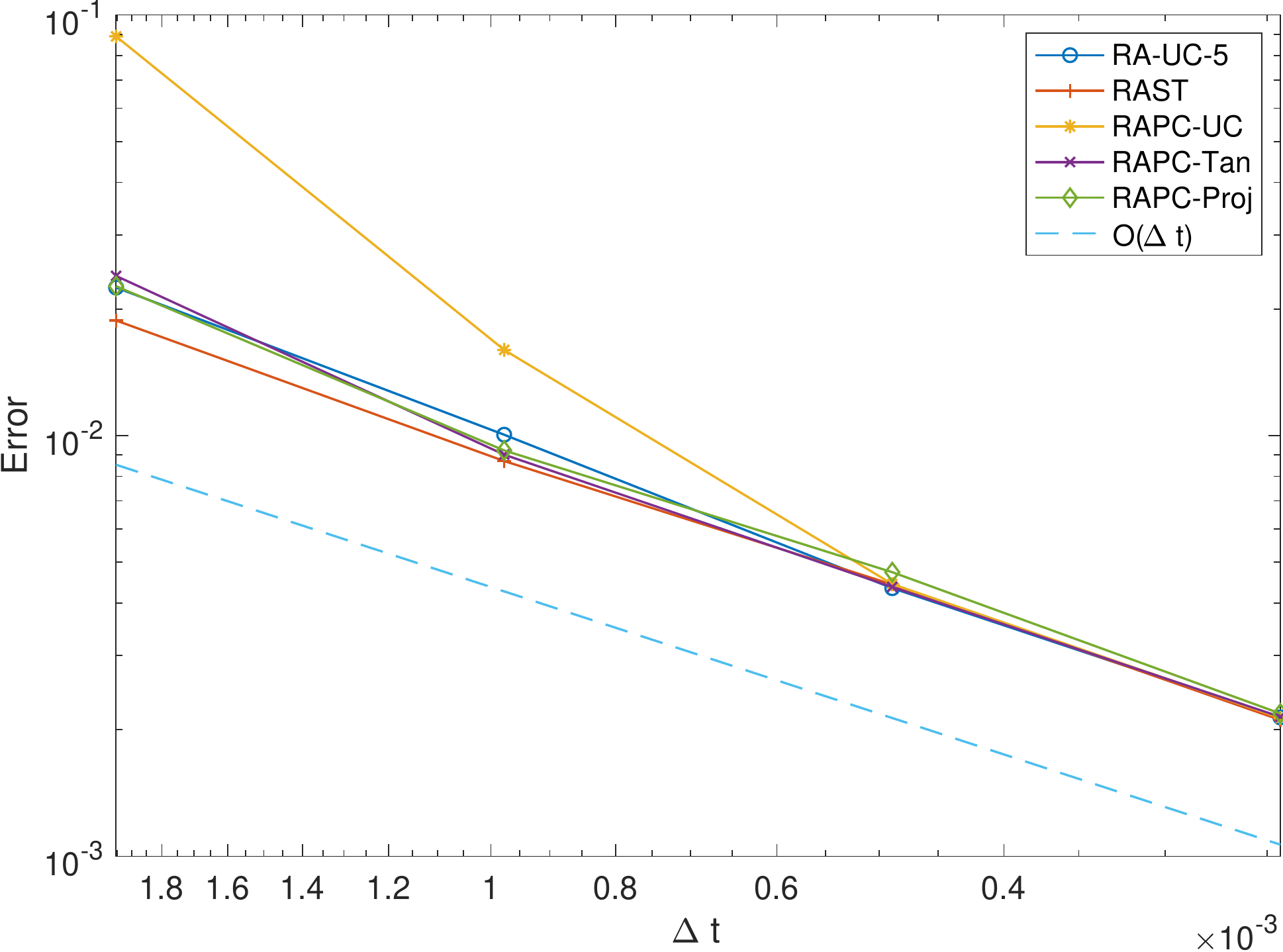}
    \caption{Rate of convergence of each low-rank method the 2D solid-body rotation problem (\Cref{test:solid_body}). $\Delta t = \omega/2^k$ where $k=1,\ldots,4$ and the tolerances are shown in Row 3 of \cref{tab:tol_test1}.  For RAUC-5, we set $\tau = 50\Delta t^2$. }
    \label{fig:rates}
\end{figure}

\subsubsection{Statistical Analysis}\label{subsect:stat_err}

The purpose of this test is to provide statistical information about how variations in the randomized SVD (\Cref{alg:ra_svd}) affect the error of the RAPC algorithm.  Fixing a tolerance, we run \Cref{alg:resid_adapt} with $\Delta t=\omega/16$ from $t=0$ to $T=\pi$ for 50 runs and measure the error between the RAPC-UC, RAPC-Tan, and RAPC-Proj solution and the full-rank SSP-RK3 approximation.  We then compute the sample mean and variance by the standard formulae:
\begin{equation}
    \mu = \frac{1}{N}\sum_{i=1}^{N}X_i \quad\text{ and }\quad \sigma^2 = \frac{1}{N-1} \sum_{i=1}^{N} (X_i-\mu)^2,
\end{equation}
where $\{X_i\}_{i=1}^N$ is the collection of random samples and $N=50$.

\begin{table}[h]
    \centering
    \begin{tabular}{c|c c c c c} \hline
 Method  & RAPC-UC & RAPC-Tan & RAPC-Proj \\ \hline
 $\tau$ & $880\Delta t^2$ & $11255\Delta t^2$ & $857\Delta t^2$ \\ 
 & 5.2452e-5 & 6.7085e-4 & 5.1081e-5 \\ \hline
$\mu$ &  2.9862e-3 & 2.9644e-3 & 3.052e-3 \\ \hline
$\sigma$ & 6.449e-5 & 1.594e-5 & 9.214e-5 \\ \hhline{=|===}
$\tau$ & $400\Delta t^2$ & $500\Delta t^2$ & $200\Delta t^2$ \\
& 2.3842e-5 & 2.9802e-5 & 1.1921e-5 \\ \hline
$\mu$ & 2.1328e-3 & 2.1423e-3 & 2.1967e-3 \\ \hline
$\sigma$ & 5.563-6 & 9.191e-7 & 2.148e-5 \\ \hline
\end{tabular}
    \caption{
    Mean and standard deviations of 50 samples of the error of the Residual based rank-adaptive methods \Cref{subsect:rara} with all three DLRA integtors in \Cref{sect:DLRA} measured against the standard full-rank SSP-RK3 approximation.  The final time is $T=\pi$ and $\Delta t = \omega/16$. 
    }
    \label{tab:rand_test3}
\end{table}

\Cref{tab:rand_test3} shows the average error and the standard deviation of the runs.  The standard deviation of the runs is at least an order of magnitude from the mean-error.  Thus the randomized algorithm SVD has little impact on the approximation of the solution.

\subsection{Test Problem 2 - Advection with Relaxation}\label{subsect:rad}

We consider an equation with advection and relaxation that is often used to model time-dependent radiation transport. Following, \cite{hauck2009temporal}, the relaxation coefficient is allowed to be spatially dependent and discontinuous. Let $\W_x=\W_y=(-1,1)$, and consider the following initial value problem:
\begin{subequations}
\label{eqn:rad_trans_pde}
\begin{align}
    \frac{\partial u}{\partial t} + \frac{1}{\e}y \cdot \grad_x u + \frac{\sigma(x)}{\e^2}(u-Pu) &= 0 \\
    u(x,y,0) = u_0(x,y) &= \rho_0(x).
\end{align}
\end{subequations}
where $\e>0$, $Pu = \frac{1}{2}\int_{\W_y} u(x,y) \dx[y]$, and 
\begin{align}
    \sigma(x) = 
    \begin{cases}
        0.02 &\text{ if }x\in (-0.75,-0.25)\cup(0.25,0.75) \\
        1    &\text{ if }x\in [1,-0.75]\cup[-0.25,0.25]\cup[0.75,1]
    \end{cases}
\end{align}
We equip \eqref{eqn:rad_trans_pde} with zero-inflow boundary conditions and initial data $\rho_0(x) = \chi_{E}(x)$ where $E=(-0.2,0.2)$, and set $\e=1/5$.

\subsubsection{Matrix Representation}

With upwind numerical fluxes, the DG discretization of \eqref{eqn:rad_trans_pde} gives the following ODE for $u_h(t)\in V_h$:
\begin{align}\label{eqn:rad_trans_method}
\begin{split}
    \left(\tfrac{\partial u_h}{\partial t},q_h\right)_\W &- \tfrac{1}{\e}(yu_h,\grad_x q_h)_\W + \left<y\lavg u_h\ravg +\tfrac{|y\cdot n|}{2}\ljmp u_h\rjmp,
    \ljmp q_h\rjmp\right>_{\mE_x^\mathrm{I}\times\W_y} + \left<y u_h ,
    n q_h\right>_{\{y\cdot n(x) > 0\}}  \\
    &+ \tfrac{1}{\e^2}(\sigma u_h,q_h)_\W-\tfrac{1}{\e^2}(\tfrac{\sigma}{2}(u_h,1)_{\W_y},q_h)_\W = 0
\end{split}
\end{align}
for all $q_h\in V_h$. Using the notation of \Cref{sect:mateval}, we may rewrite \eqref{eqn:rad_trans_method} into a matrix ODE
\begin{equation}
\frac{\partial U_h}{\partial t} = F(U_h)
\end{equation}
where $F$ has three terms: two for the advection operator and one for the relaxation operator $u-Pu$.  Rewriting the advection operator is similar to the process in \Cref{subsect:lr_features}, while the relaxation operator can be expressed with one term.  Given the bases $\{\phi_i(x)\}_{i=1}^m$ and $\{\psi_j(y)\}_{j=1}^n$ for $V_{x,h}$ and $V_{y,h}$, let $\overline{\zeta}\in\R^{n\times 1}$ be the coefficient representation of the function $\zeta(y) = 1$ with respect to $\{\psi_j(y)\}_{j=1}^n$.  Then 
\begin{equation}
    (\sigma u_h,q_h)_\W-(\tfrac{\sigma}{2}(u_h,1)_{\W_y},q_h)_\W = (AU_h(M-\tfrac{1}{2}M\overline{\zeta}\overline{\zeta}^T),Q_h)_F,
\end{equation}
where $M_{ij} = (\psi_j,\psi_i)_{\W_y}$ and $A_{ij} = (\sigma\phi_j,\phi_i)_{\W_x}$.  Depending on the basis, the matrix $\overline{\zeta}\overline{\zeta}^T$ may be dense, but the action $Z\to Z\overline{\zeta}\overline{\zeta}^T$ can be computed without construction of $\overline{\zeta}\overline{\zeta}^T$ thereby avoide the storage of an $\mathcal{O}(m^2)$ object.  

\subsubsection{Numerical Test}\label{subsub:rad_num_test}

We set $h_x=h_y=1/64$ and $k_x=k_y=1$ so that the number of degrees of freedom in each variable is $m=n=256$.
The timestep is $\Delta t = \frac{1}{12800}$, the final time is $T=3$, and the discrete initial condition is given as a true rank-one matrix $U_h\big|_{t=0}= \overline{\rho_0}\overline{\zeta}^T$ where $\overline{\rho_0}\in\R^{m\times 1}$ is the coefficient representation of $\rho_0$ with respect to $\{\phi_i(x)\}_{i=1}^n$.  
The oversampling parameter $p$ is set to 10.

At time $t=T$, the error between the full-rank forward Euler and SSP-RK3 method is $8.95\times 10^{-6}$.  We choose the tolerance in each algorithm, by trial and error, such that the error between the low-rank approximations and the SSP-RK3 reference solution is roughly $2.5\times 10^{-3}$.    Every method is able to recover the specified accuracy, and the tolerances used to do so are given in \Cref{tab:rad_tols}. 

The solution rank as a function of time for each method is plotted in \Cref{fig:rad-fig-rank}.  The rank of every method except PCRA-Tan is similar over time.  The RAUC method uses the least memrory; the RAPC-Tan method and the RAST method use the most memory; and the memory usage of the other RAPC methods lie somewhere in between.  The rank disparity between the RAPC-Tan integrator and the other two RAPC algorithms is larger than in the \Cref{test:solid_body}; this is likely because the rank of the tangent integrator (\Cref{subsec:tan}) can double every timestep.  If the residual error is still not small enough, then the RAPC algorithm will then add more vectors onto to discrete approximation yielding a large rank solution. 

\begin{table}[h]
    \centering
    \begin{tabular}{c|c c c c c} \hline
 Method & RAUC & RAST & RAPC-UC & RAPC-Tan & RAPC-Proj \\ \hline
 $\tau$ & $47700\Delta t^2$ & $47100\Delta t^2$ & $56000\Delta t^2$ & $165000\Delta t^2$ & $58000\Delta t^2$ \\ 
 & 2.91e-4 & 2.87e-4 & 3.4e-4 & 1.00e-3 & 3.54e-4 \\ \hline
Error &  2.51e-3 & 2.71e-3 & 2.36e-3 & 2.22e-3 & 2.39e-3  \\ \hline
\end{tabular}
    \caption{
    Tolerances and errors for all five rank-adaptive methods in \Cref{subsect:rad}.  The errors are computed at $T=\pi$ against the full-rank SSP-RK3 discretization. $\Delta t = 1/128000$.
    }
    \label{tab:rad_tols}
\end{table}

\begin{figure}[ht]
    \centering
    \includegraphics[width=0.49\textwidth]{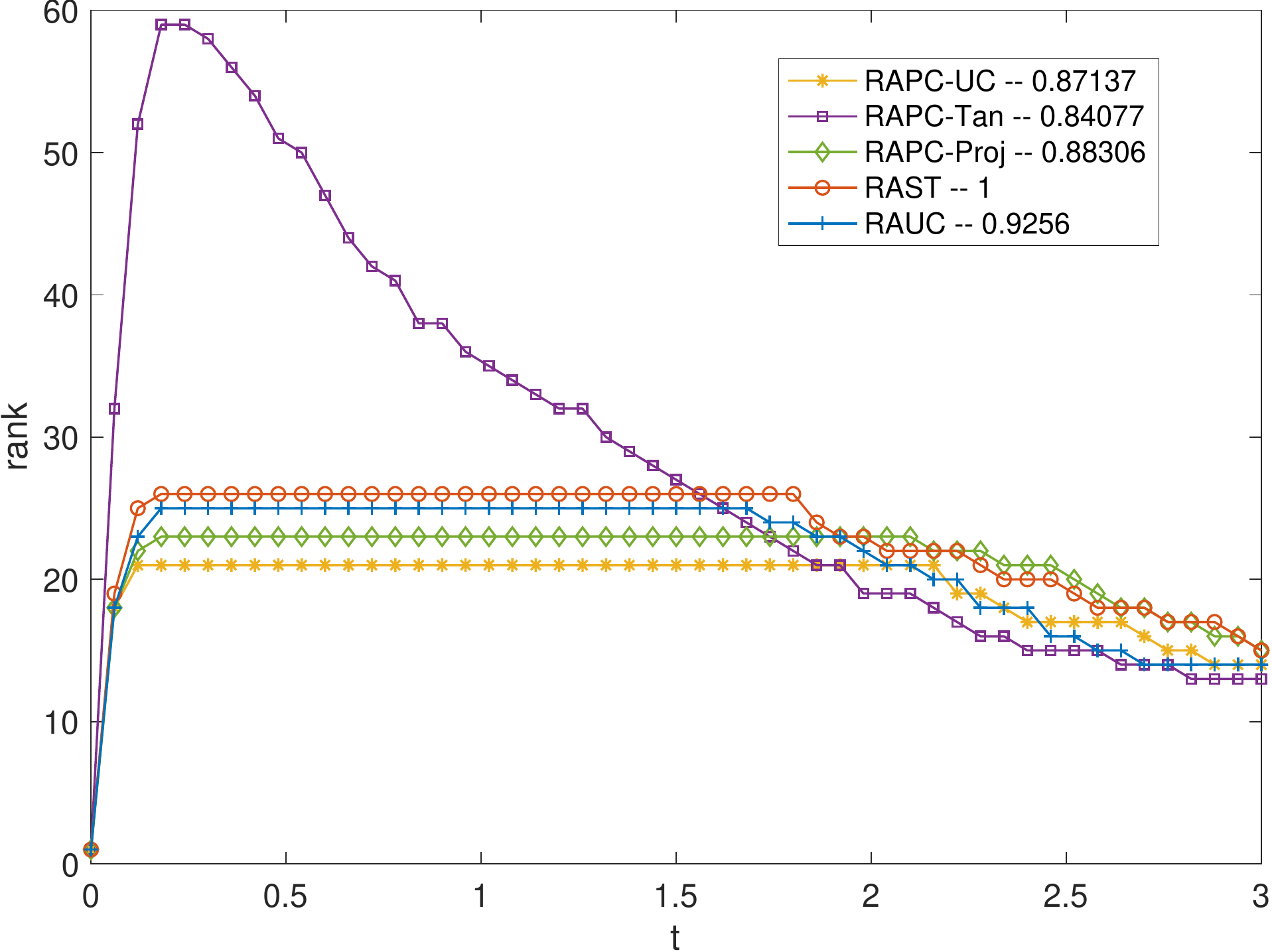}
    \includegraphics[width=0.49\textwidth]{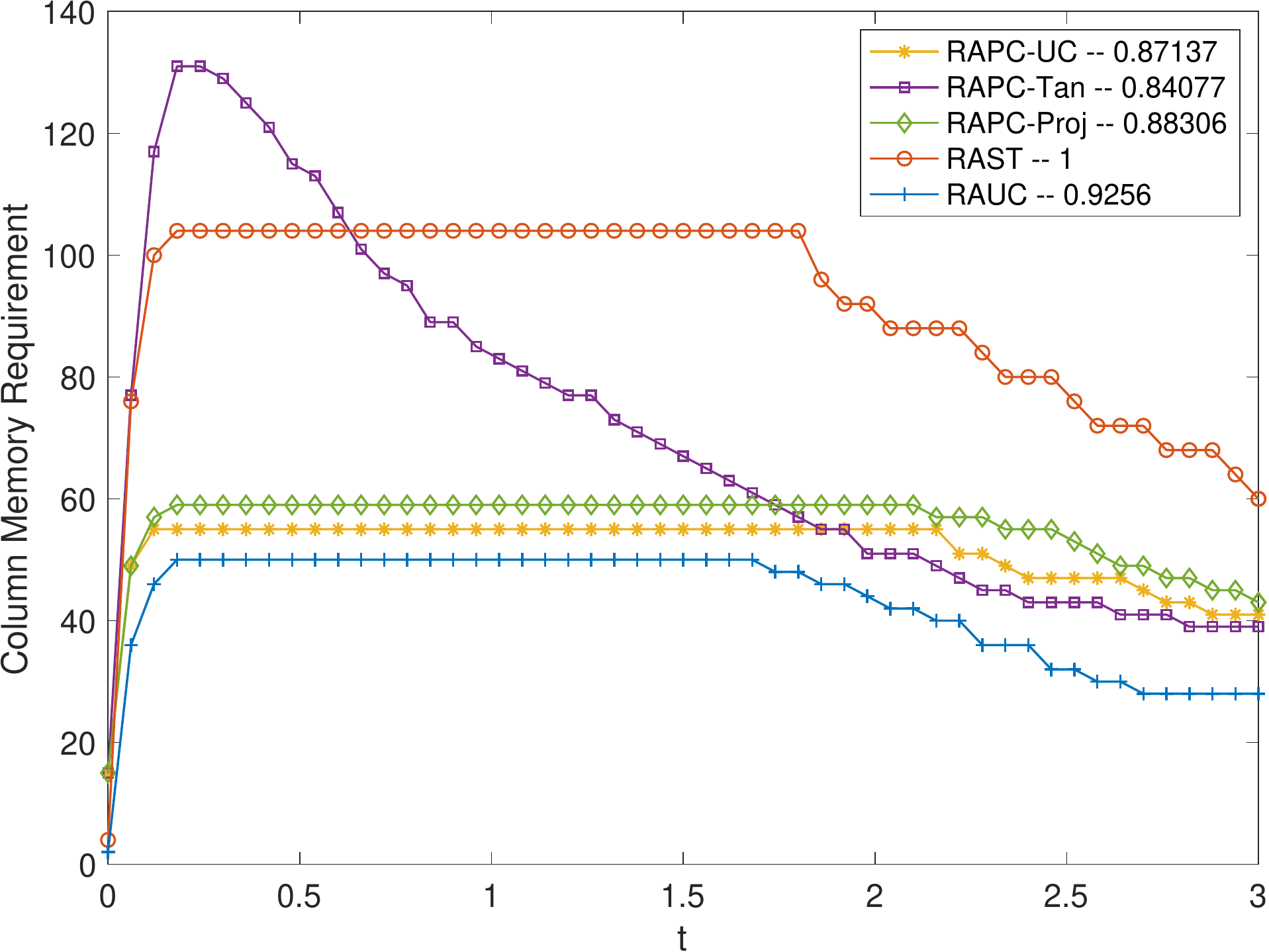}
    \caption{Rank (left) and memory requirements (right) of each of the five methods over time for the advection-relaxation problem in \Cref{subsect:rad}. The tolerances for each adaptive algorithm are given in \Cref{tab:rad_tols} and are chosen to produce errors near $6.5\times 10^{-3}$ when measured against the SSP-RK3 full-rank solution. The memory footprint is calculated using \Cref{tab:mem_comp}. For readability, the data is plotted at every 160 time steps.  The legend labels provide the error of the method, normalized with respect to the error of the RAST algorithm.}
    \label{fig:rad-fig-rank}
\end{figure}

\subsubsection{Statistical Tests}\label{subsub:rad-rand}

It turns out that the advection-relaxation problem is much more sensitive to the errors created by the randomized SVD than the solid-body rotation problem from \Cref{test:solid_body}. To demonstrate this behavior, we consider a simulation with $h_x=h_y=1/64$, $k_x=k_y=1$, $T=2$, and $\Delta t = 1/6400$.  The tolerance for each of the RAPC methods is set to $40000\Delta t^2 \approx 9.77\times 10^{-4}$ which produces errors that are roughly between $6\times 10^{-3}$ and $7\times 10^{-3}$.  For a given RAPC algorthm and oversampling parameter $p$, we run the algorithm 20 times and save the error with respect the SSP-RK3 full-rank reference solution.  We repeat this procedure for all three RAPC algorithms and $p \in \{5,7,10,20,30\}$.  
A statistical summary of the results is plotted in \Cref{fig:rad-rand-data}.  The RAPC-UC method gives the expected result:  increasing the oversampling parameter increases the accuracy in the randomized SVD (see \Cref{prop:rand-svd-error}) and reduces the statistical variation.  The standard deviation for the RAPC-Proj method also correlates with $p$, but not as strongly.  Meanwhile, the RAPC-Tan results do not appear to correlate with $p$ at all. 

\begin{figure}[ht]
    \centering
\begin{subfigure}[t]{0.32\textwidth}
 \includegraphics[width=0.9\textwidth]{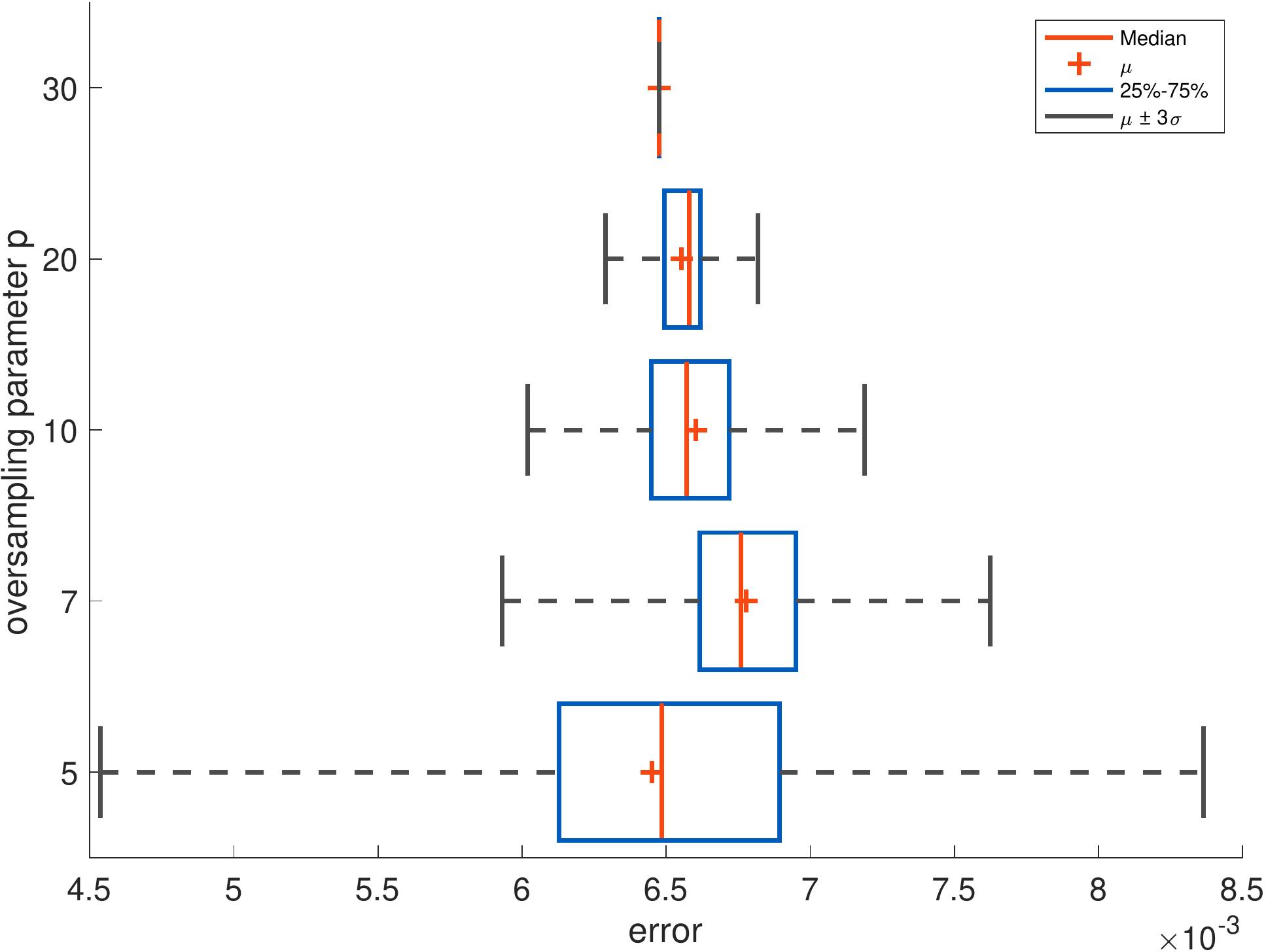}
 \caption{RAPC-UC}
\end{subfigure}
\hfill
\begin{subfigure}[t]{0.32\textwidth}
 \includegraphics[width=0.9\textwidth]{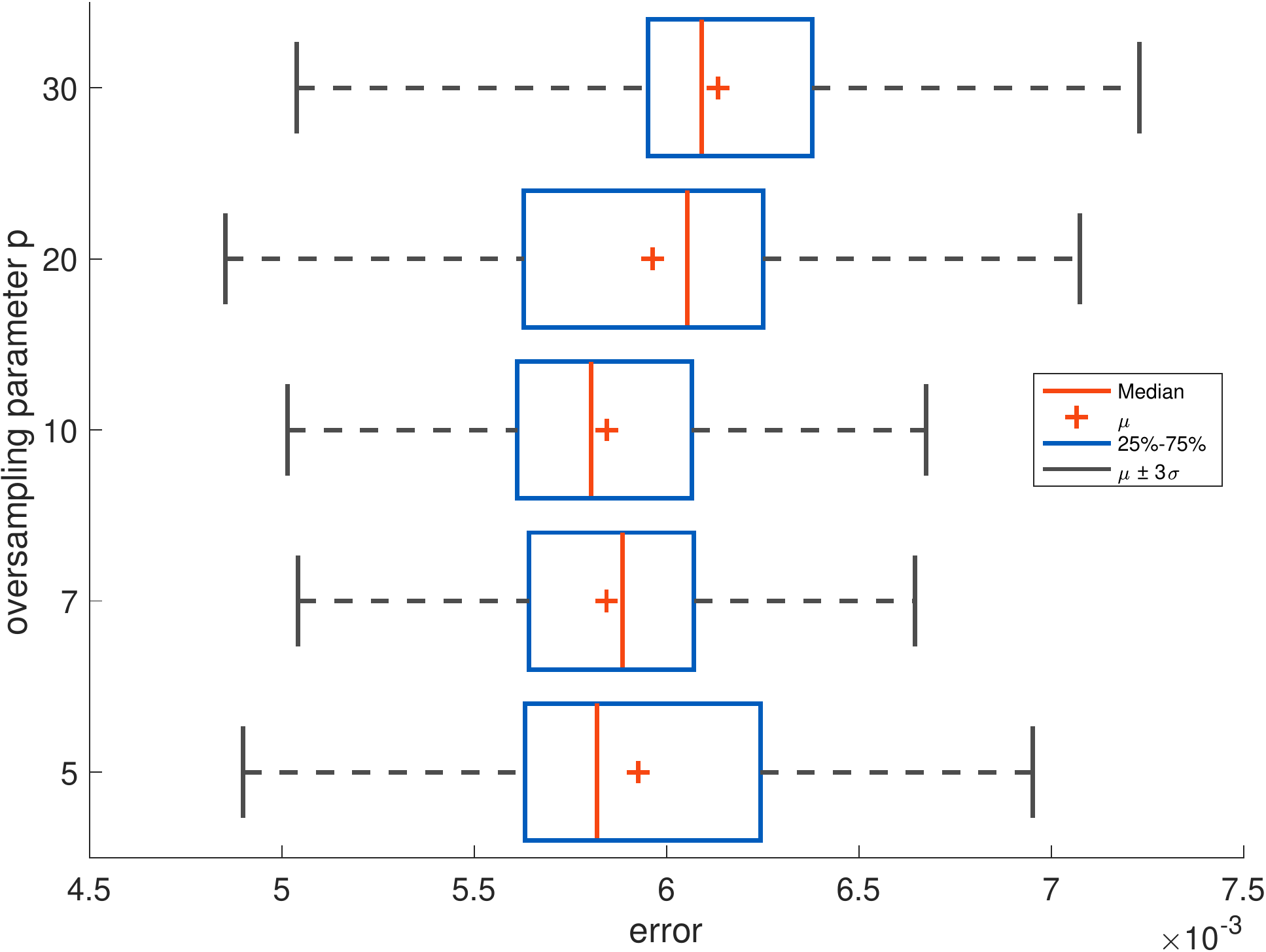}
  \caption{RAPC-Tan}
\end{subfigure}
\hfill
\begin{subfigure}[t]{0.32\textwidth}
 \includegraphics[width=0.9\textwidth]{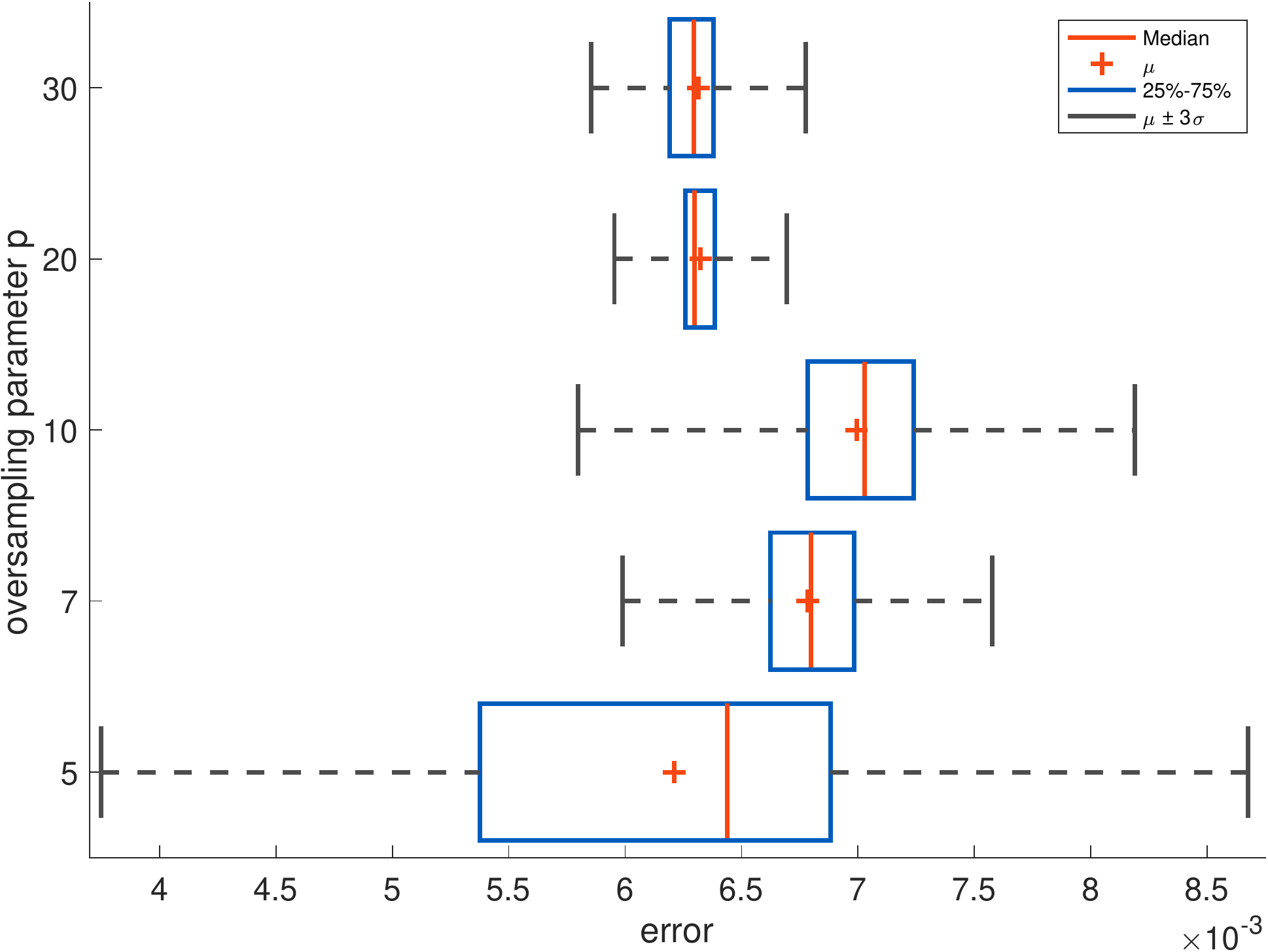}
 \caption{RAPC-Proj}
\end{subfigure}
    \caption{Statistical data for the advection-relaxation problem given in \Cref{subsect:rad}. For each oversampling parameter listed on the $y$-axis, 20 runs are computed.  The errors are created via comparison to the SSP-RK3 full-rank reference solution.  Parameters for the test are listed in \Cref{subsub:rad-rand}. }
    \label{fig:rad-rand-data}
\end{figure}

While the randomness of the randomized SVD may increase the sensitivity of the algorithm, the process of culling process is also very sensitive to the tolerance chosen.  
To demonstrate, using the parameters in \Cref{subsub:rad_num_test}, we only consider the methods which are deterministic: the RAUC and RAST methods.  In \Cref{fig:rad-fra-lub-data} results are displayed for the both of these algorithms and the tolerances used are within 1\% of the tolerances given in \Cref{tab:rad_tols}.  The maximum rank of the RAST test varies significantly as the tolerance is changed by only a slight amount.  As the tolerance increases, the algorithms should be willing to cut more and more rank which should yield lower rank solutions -- this trend does happen on the longer time scales.  However, during the initial rise this trend is reversed: the higher tolerance methods actually require more rank in the beginning.  The errors in both methods are not monotonic with respect to the tolerance as well; the RAUC intregator gives the best error with the largest tolerance of the sample.  This test shows that with all of these algorithms -- random or not -- choosing the tolerance is a delicate process and can greatly affect the rank and accuracy of the corresponding solution.  

\begin{figure}[H]
    \centering
    \includegraphics[width=0.49\textwidth]{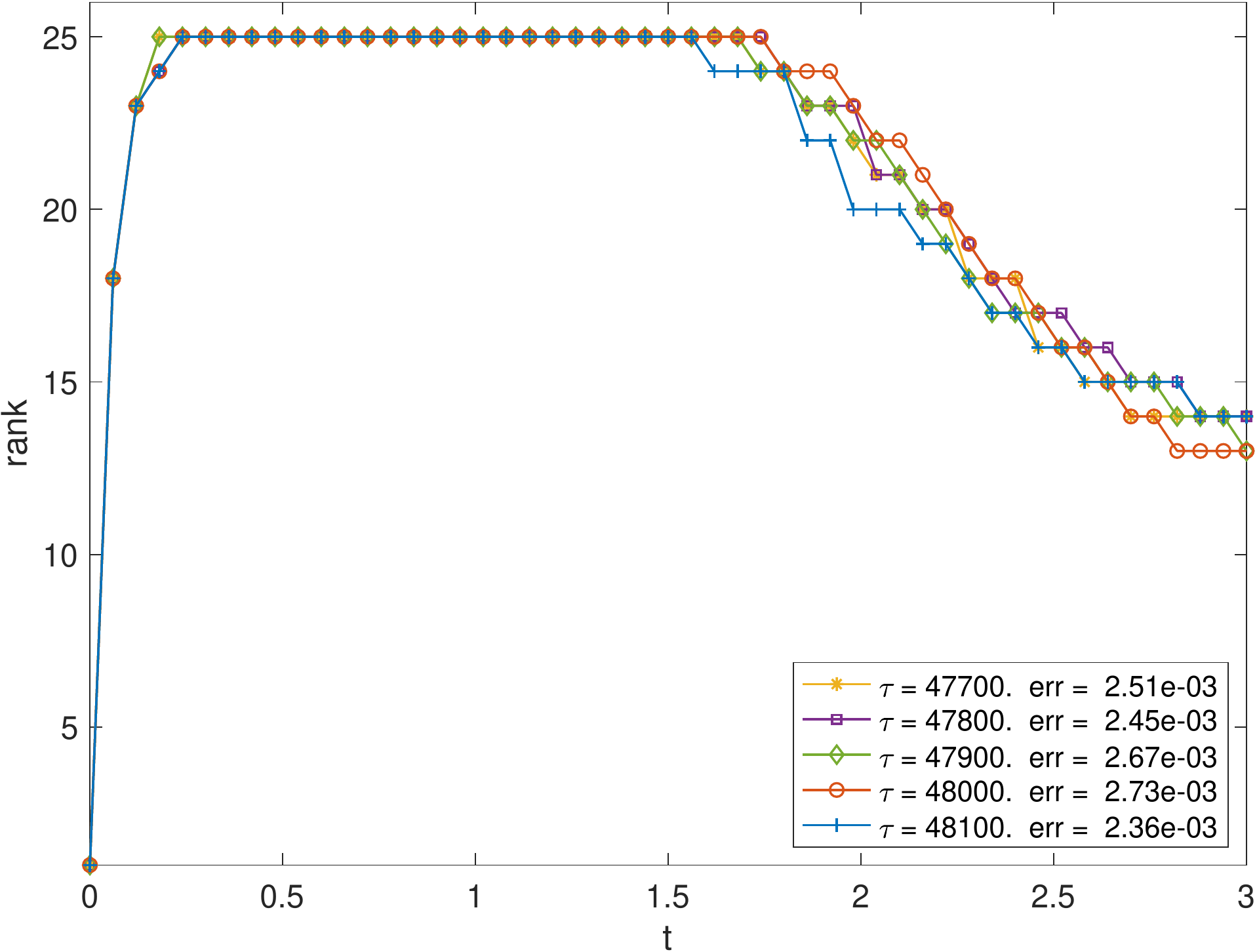}
    \includegraphics[width=0.49\textwidth]{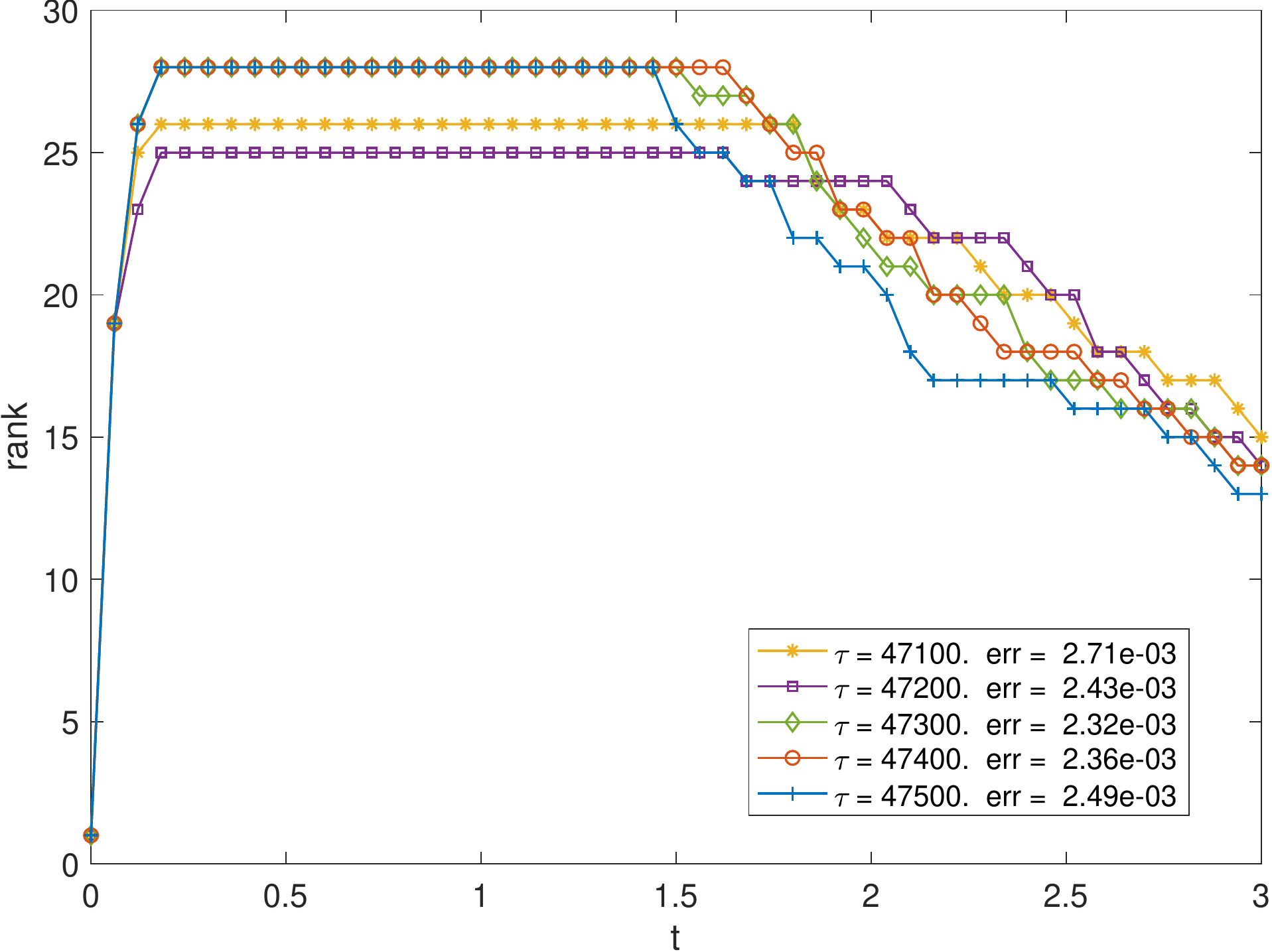}
    \caption{Rank for the radiation transport problem given in \Cref{subsect:rad} for the RAUC (left) and RAST (right) algorithms with varied tolerances given in the legend.    The errors are created via comparison to the SSP-RK3 full-rank reference solution.}
    \label{fig:rad-fra-lub-data}
\end{figure}

\subsection{Test Problem 3 - Rotation, Relaxation, and Localized Sources}\label{subsect:sources}

We consider the following PDE with sources:

\begin{subequations}
\begin{alignat}{3} \label{eqn:sources_pde}
    \frac{\partial u}{\partial t} - y\cdot\grad_xu + x\cdot\grad_y u + \frac{1}{\e}(u-Pu) &= S,
    &&\qquad(x,y)\in\W, &&\quad t>0; \\
    u(x,y,t) &= 0,
    &&\qquad(x,y) \in\partial\W_x^-, &&\quad  t>0; \\
    u(x,y,t) &= 0,
    &&\qquad(x,y) \in\partial\W_y^-, &&\quad t>0; \\
    u(x,y,0) &= u_0(x,y),
    &&\qquad(x,y)\in\W &&
\end{alignat}
\end{subequations}
where $\e>0$, $Pu = \frac{1}{2}\int_{\W_y} u(x,y) \dx[y]$, $u_0$ is given in \eqref{eqn:solid_body_box_IC}, and
$S$ is given by
\begin{align}\label{eqn:source_def}
\begin{split}
    S(x,y) &= -(\chi_{E_1}(x)+\chi_{E_2}(x)+\chi_{E_3}(x)+\chi_{E_4}(x)+\chi_{E_5}(x)) \\
    &\qquad \cdot(\chi_{E_1}(y)+\chi_{E_2}(y)+\chi_{E_3}(y)+\chi_{E_4}(y)+\chi_{E_5}(y))
\end{split}
\end{align}
where $E_i=\left[\frac{20(i-3)-3}{60},\frac{20(i-3)+3}{60}\right]$ for $i=1,\ldots,5$.  The discretization and matrix evaluations of all terms on the left-hand side of \eqref{eqn:sources_pde} have been treated in \Cref{sect:mateval} and \Cref{subsect:rad}. 
Here $S_h$, the matrix representation of the $L^2$ projection of $S$ onto $V_h$ is a rank-one matrix and thus can be stored in a low-memory format. 

For this test we set $h_x=h_y=1/128$ and $k_x=k_y=0$ which yield 256 degrees of freedom in each direction.  Additionally, we set $\e=1/5$, $T=\pi$, and $\Delta t = 1/1024$.  The goal of this test is to see how well the adaptive integrators can resolve the high frequency source $S$.  The oversampling parameter $p$ is set to 15.

At $t=\pi$, the error between the full-rank forward Euler and SSP-RK3 approximations is $3.94\times 10^{-5}$.   The tolerance for each of the five methods is given in \Cref{tab:source_tols} and is chosen to produce an error of approximately $1.50\times 10^{-3}$. \Cref{fig:source-fig-rank-pi} provides the rank of the methods over time.  In this case the rank of each method quickly rises and then plateaus as the projection operator forces the solution to a steady state.  All methods but PCRA-Tan arrive at the same final rank.  Similar to \Cref{subsect:rad}, the RAUC method is able to capture the equilibrium with the smallest memory footprint. 

\begin{table}[h]
    \centering
    \begin{tabular}{c|c c c c c} \hline
 Method & RAUC & RAST & RAPC-UC & RAPC-Tan & RAPC-Proj \\ \hline
 $\tau$ & $90\Delta t^2$ & $100\Delta t^2$ & $125\Delta t^2$ & $1300\Delta t^2$ & $120\Delta t^2$ \\ 
 & 8.58e-5 & 9.54e-5 & 1.19e-5 & 1.24e-3 & 1.14e-4 \\ \hline
Error &  1.5e-3 & 1.50e-3 & 1.50e-3 & 1.72e-3 & 1.50e-3  \\ \hline
\end{tabular}
    \caption{
    Tolerances and Errors for all five rank-adaptive methods in \Cref{subsect:sources}.  The errors are computed at $T=\pi$ against the full-rank SSP-RK3 discretization. $\Delta t = 1/1024$.
    }
    \label{tab:source_tols}
\end{table}

\begin{figure}[ht]
    \centering
    \includegraphics[width=0.49\textwidth]{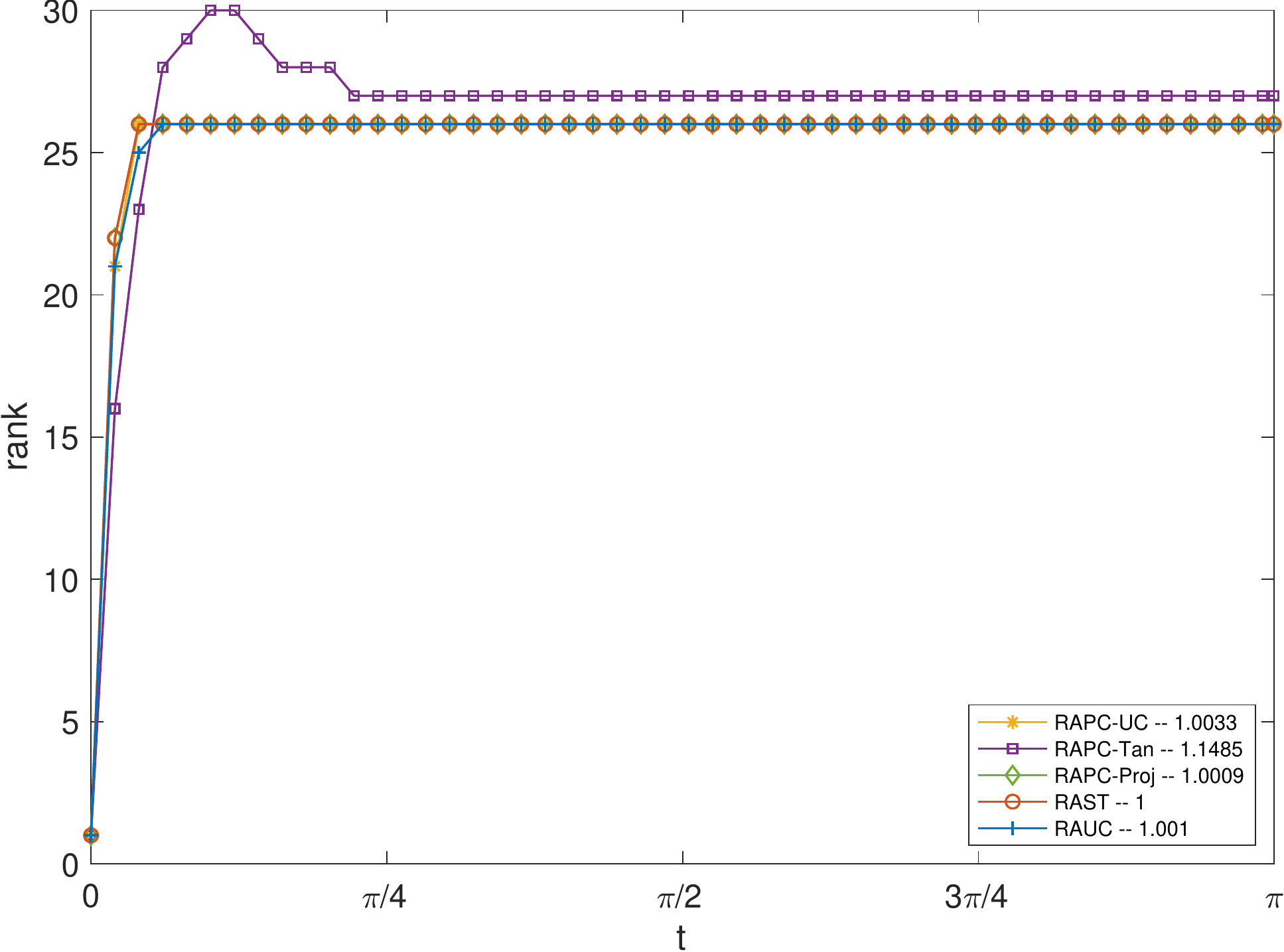}
    \includegraphics[width=0.49\textwidth]{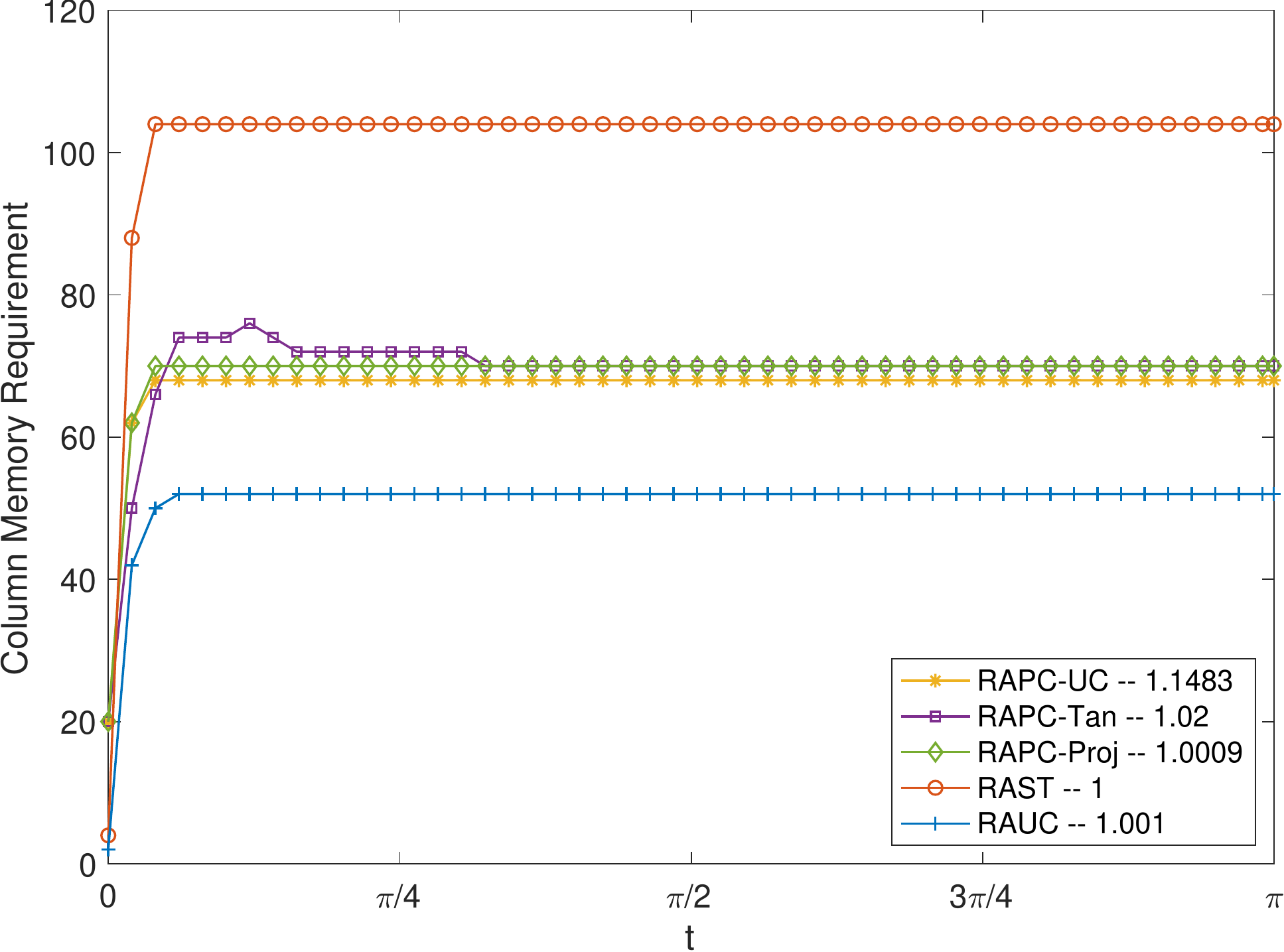}
    \caption{Rank (left) and memory requirements (right) of each of the five methods over time for the problem given in \Cref{subsect:sources} with $\Delta t = 1/1024$. The tolerances for the adaptive algorithms are given in \Cref{tab:source_tols} and is chosen to produce errors near $1.50\times 10^{-3}$ when measured against an SSP-RK3 reference solution.  The memory footprint is calculated using \Cref{tab:mem_comp}.   For readability, the data is plotted at every 65 timesteps. The legend labels each algorithm as well as the ratio of the method's error over the RAST algorithm error which is used as a reference solution.}
    \label{fig:source-fig-rank-pi}
\end{figure}

\begin{figure}[ht]
    \centering
    \includegraphics[width=0.49\textwidth]{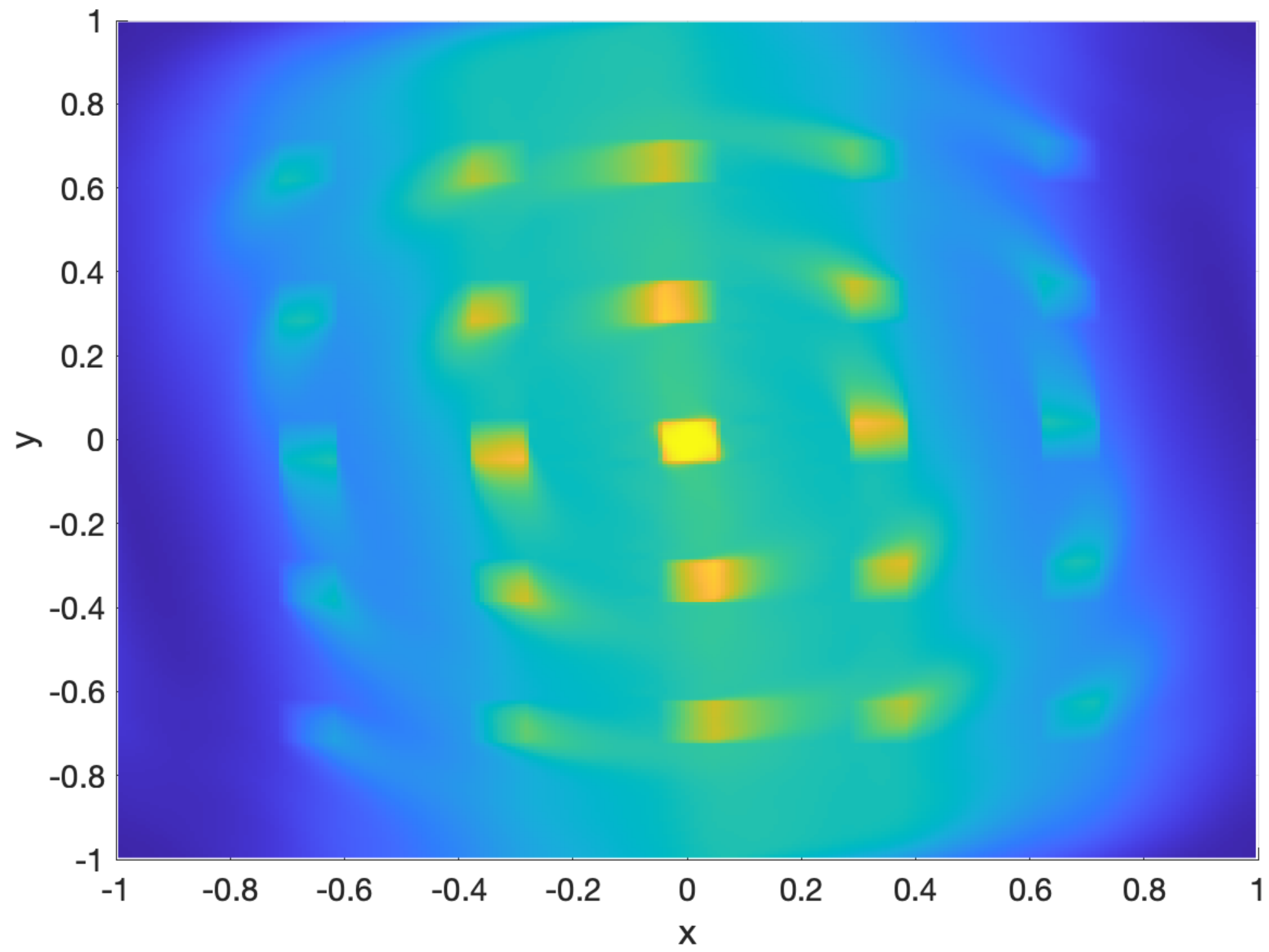}
    \includegraphics[width=0.49\textwidth]{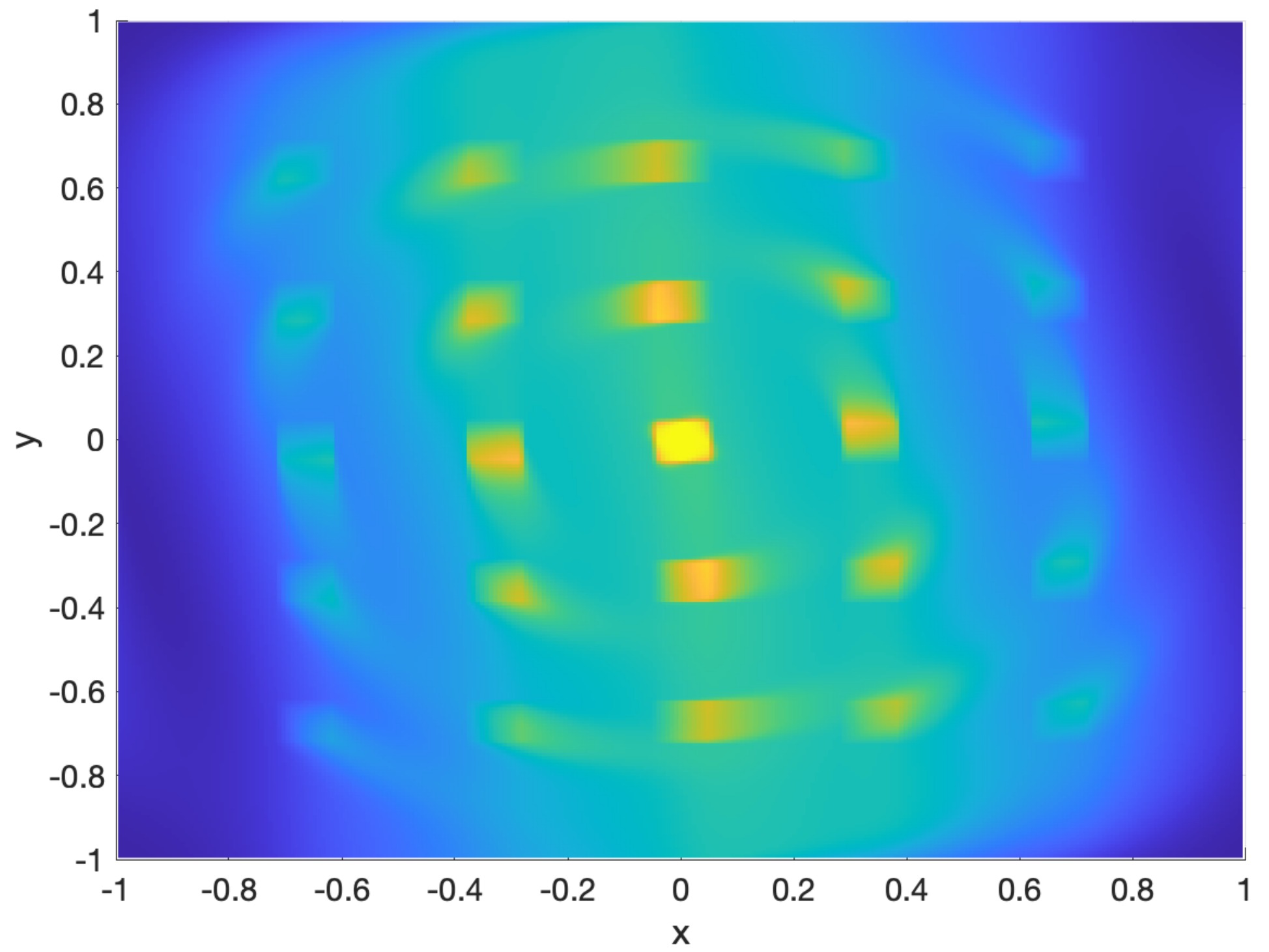}
    \caption{Plots of the discrete approximations at $t=\pi$ for the problem given in \Cref{subsect:sources} with the RAPC-Proj Algorithm (left) with a rank of 26 and the full-rank SSP-RK3 method (right). The tolerance for the adaptive algorithm is given in \Cref{tab:source_tols}.}
    \label{fig:source-solns}
\end{figure}

\section{Conclusion}

In this paper, we introduced a predictor-corrector strategy that modifies the rank of the DLRA solution by building a low-rank approximation of the error made against the full rank forward Euler method.  We presented several numerical results showing our method is accurate in resolving the modelling error of the low-rank projection and that the methods rank is comparable with other rank-adaptive DLRA methods.  The paper also included a discussion on how a variety of PDEs discretized by the discontinuous Galerkin method can be rewritten to fit the DLRA framework. Future topics include:
\begin{itemize}
    \item Building higher order methods in time using the predictor-corrector strategy with forward Euler timestepping as a building block.  This approach is similar to the work in \cite{kieri2019projection}. 
    \item Using more advanced randomized techniques, such as power iterations \cite{halko2011finding}, to diminish the effects of randomness on the solution.
    \item Extensions to dynamic low-rank approximation on higher order tensors \cite{koch2010dynamical}.
\end{itemize}

\section{Appendix}

\begin{proof}[Proof of \Cref{prop:A_low_rank}]
We focus on \eqref{eqn:adv_decomp_3x}-\eqref{eqn:adv_decomp_4y}.  Let $w_h=\phi_i\psi_j$ and $q_h=\phi_k\psi_l$. 
The portion of $\mA$ that \eqref{eqn:adv_decomp_3x}-\eqref{eqn:adv_decomp_4y} separates, namely, 
\begin{align}\label{eqn:flux_x}
- (xw_h,\grad_x q_h)_{\W} + \left<x\lavg w_h\ravg+\tfrac{1}{2}|x|\ljmp w_h\rjmp,\ljmp q_h\rjmp\right>_{\W_x\times\mE_{y,h}^\mathrm{I}} + \left<xw_h,\bn q_h\right>_{\partial\W_y^+}
\end{align}
is not separable because the flux depends on the sign of $x$.  To create the minimal number of separable terms, we split the mass integral $(x\phi_i,\psi_k)_{\W_x}$ about $x=0$ to obtain $(x\phi_i,\psi_k)_{\W_x}= (x\phi_i,\psi_k)_{\{x > 0\}} + (x\phi_i,\psi_k)_{\{x < 0\}}$.  Splitting \eqref{eqn:flux_x} yields
\begin{align}\label{eqn:flux_x_split_I1I2}
- (xw_h,\grad_y q_h)_{\W} + \left<x\lavg w_h\ravg\right.&\left.+\,\tfrac{1}{2}|x|\ljmp w_h\rjmp,\ljmp q_h\rjmp\right>_{\W_x\times\mE_{y,h}^\mathrm{I}} + \left<xw_h,\bn q_h\right>_{\partial\W_y^+} = I_1 + I_2
\end{align} 
where
\begin{align}
    I_1 &= -(xw_h,\grad_y q_h)_{\{x > 0\}\times\W_y} + \left<x\lavg w_h\ravg+\tfrac{1}{2}|x|\ljmp w_h\rjmp,\ljmp q_h\rjmp\right>_{\{x > 0\}\times\mE_{y,h}^\mathrm{I}} + \left<xw_h,\bn q_h\right>_{\partial\W_y^+\cap\{x > 0\}\times\partial\W_y} \label{eqn:flux_x_split_1}\\
    I_2 &= -(xw_h,\grad_y q_h)_{\{x < 0\}\times\W_y} + \left<x\lavg w_h\ravg+\tfrac{1}{2}|x|\ljmp w_h\rjmp,\ljmp q_h\rjmp\right>_{\{x < 0\}\times\mE_{y,h}^\mathrm{I}} + \left<xw_h,\bn q_h\right>_{\partial\W_y^+\cap\{x < 0\}\times\partial\W_y} \label{eqn:flux_x_split_2}
\end{align}
We focus on $I_1$.  For $x>0$, $|x| = x$ and thus the flow direction is independent of $x$. Moreover, \eqref{eqn:inflow_y_def} and the sign of $x$ implies the outflow boundary is only on the top of $\W_y$.  Thus
\begin{align}\label{eqn:outflow_x_change}
    \partial\W_y^+\cap\{x > 0\}\times\partial\W_y = \{x > 0\}\times\{y = L\}
\end{align}
Substituting \eqref{eqn:outflow_x_change} into \eqref{eqn:flux_x_split_1} yields
\begin{align}\label{eqn:x-split1}
\begin{split}
I_1 &= 
-(x\phi_i\psi_j,\grad_y(\phi_k\psi_l))_{\{x > 0\}\times\W_y} + \big<x\lavg \phi_i\psi_j\ravg+\tfrac{1}{2}|x|\ljmp \phi_i\psi_j\rjmp,\ljmp \phi_k\psi_l\rjmp\big>_{\{x>0\}\times\mE_{y,h}^\mathrm{I}}\\
&\quad+ \left<x\phi_i\psi_j,\bn \phi_k\psi_l\right>_{\{x > 0\}\times\{y = L\}} \\
&=-(x\phi_i,\psi_k)_{\{x > 0\}}(\psi_j,\grad_y\psi_l)_{\W_y} \\
&\quad+(x\phi_i,\psi_k)_{\{x > 0\}}\left<\lavg \psi_j\ravg+\tfrac{1}{2}\ljmp \psi_j\rjmp,\ljmp \psi_l\rjmp\right>_{\mE_{y,h}^\mathrm{I}} \\
&\quad+ (x\phi_i,\psi_k)_{\{x > 0\}}\left<\psi_jn,\psi_l\right>_{\{y = L\}} \\
&= (x\phi_i,\psi_k)_{\{x > 0\}}\big(-(\psi_j,\grad_y\psi_l)_{\W_y} +\left<\lavg \psi_j\ravg+\tfrac{1}{2}\ljmp \psi_j\rjmp,\ljmp \psi_l\rjmp\right>_{\mE_{y,h}^\mathrm{I}} + \left<\psi_jn,\psi_l\right>_{\{y = L\}} \big) \\
&= \mB_{3,x}(\phi_i,\phi_k)\mB_{3,y}(\psi_j,\psi_l).
\end{split}
\end{align}
We now focus on $I_2$.  Since $x<0$, $|x|=-x$ and 
\begin{align}\label{eqn:outflow_x_change_2}
    \partial\W_y^+\cap\{x < 0\}\times\partial\W_y = \{x < 0\}\times\{y = -L\}
\end{align}
Therefore similar to the derivation of \eqref{eqn:x-split1} we have
\begin{align}\label{eqn:x-split2}
\begin{split}
I_2 &= 
-(x\phi_i\psi_j,\grad_y(\phi_k\psi_l))_{\{x < 0\}\times\W_y} + \big<x\lavg \phi_i\psi_j\ravg+\tfrac{1}{2}|x|\ljmp \phi_i\psi_j\rjmp,\ljmp \phi_k\psi_l\rjmp\big>_{\{x<0\}\times\mE_{y,h}^\mathrm{I}}\\
&\quad+ \left<x\phi_i\psi_j,\bn \phi_k\psi_l\right>_{\{x < 0\}\times\{y = -L\}} \\
&=-(x\phi_i,\psi_k)_{\{x < 0\}}(\psi_j,\grad_y\psi_l)_{\W_y} \\
&\quad+(x\phi_i,\psi_k)_{\{x < 0\}}\left<\lavg \psi_j\ravg-\tfrac{1}{2}\ljmp \psi_j\rjmp,\ljmp \psi_l\rjmp\right>_{\mE_{y,h}^\mathrm{I}} \\
&\quad+ (x\phi_i,\psi_k)_{\{x < 0\}}\left<\psi_jn,\psi_l\right>_{\{y = -L\}} \\
&= (x\phi_i,\psi_k)_{\{x < 0\}}\big(-(\psi_j,\grad_y\psi_l)_{\W_y} +\left<\lavg \psi_j\ravg-\tfrac{1}{2}\ljmp \psi_j\rjmp,\ljmp \psi_l\rjmp\right>_{\mE_{y,h}^\mathrm{I}} + \left<\psi_jn,\psi_l\right>_{\{y = -L\}} \big) \\
&= \mB_{4,x}(\phi_i,\phi_k)\mB_{4,y}(\psi_j,\psi_l).
\end{split}
\end{align}
\eqref{eqn:adv_decomp_1x}-\eqref{eqn:adv_decomp_2y} can be similarly shown.  The proof is complete.
\end{proof}

\begin{proof}[Proof of \Cref{prop:G_low_rank}]
By \eqref{eqn:u_coeff_formula}, \eqref{eqn:uh_coords}, \eqref{eqn:g_h_def}, and \eqref{eqn:G_def}, we have
\begin{align}
\begin{split} \label{eqn:prop:3.4.1}
    (G_h(t))_{ij} &= g_h(t)^{ij} = (g_h(t),\phi_i\psi_j)_\W = \mathcal{G}(\phi_i\psi_j,t) \\
    &= \left<y g_x(\cdot,t),\bn \phi_i\psi_j\right>_{\partial\W_x^-} + \left<-x g_y(\cdot,t),\bn \phi_i\psi_j\right>_{\partial\W_y^-}.
\end{split}
\end{align}
Spitting the terms in \eqref{eqn:prop:3.4.1} about $x=0$ and $y=0$ yields
\begin{align}
\begin{split} \label{eqn:prop:3.4.2}
    (G_h(t))_{ij}  
    &= \left<y g_x(\cdot,\cdot,t),\bn \phi_i\psi_j\right>_{\partial\W_x^-\cap \partial\W_x\times\{y > 0\}} + \left<y g_x(\cdot,\cdot,t),\bn \phi_i\psi_j\right>_{\partial\W_x^-\cap \partial\W_x\times\{y < 0\}} \\
    &\quad+ \left<-x g_y(\cdot,\cdot,t),\bn \phi_i\psi_j\right>_{\partial\W_y^-\cap \{x > 0\}\times\partial\W_y} + \left<-x g_y(\cdot,\cdot,t),\bn \phi_i\psi_j\right>_{\partial\W_y^-\cap \{x < 0\}\times\partial\W_y} \\
    &= \left<y g_x(\cdot,\cdot,t),\bn \phi_i\psi_j\right>_{\{x = L\}\times\{y > 0\}} + \left<y g_x(\cdot,\cdot,t),\bn \phi_i\psi_j\right>_{\{ x = -L \}\times\{y < 0\}} \\
    &\quad+ \left<-x g_y(\cdot,\cdot,t),\bn \phi_i\psi_j\right>_{\{x > 0\}\times\{ y = -L \}} + \left<-x g_y(\cdot,\cdot,t),\bn \phi_i\psi_j\right>_{ \{x < 0\}\times\{y = L\}} \\
    &= \left<1,\bn \phi_i\right>_{\{x = L\}}(y g_x(L,\cdot,t),\psi_j)_{\{y > 0\}}
    + \left<1,\bn \phi_i\right>_{\{x = -L\}}(y g_x(-L,\cdot,t),\psi_j)_{\{y < 0\}}\\
    &\quad+ (-x g_y(\cdot,-L,t),\phi_i)_{\{x > 0\}}\left<1,\bn \psi_j\right>_{\{ y = -L \}} 
    + (-x g_y(\cdot,L,t),\phi_i)_{\{x < 0\}}\left<1,\bn \psi_j\right>_{\{ y = L \}} \\
    &:= c_i^1d_j^1(t) + c_i^2d_j^2(t) + c_i^3(t)d_j^3 + c_i^4(t)d_j^4. 
\end{split}
\end{align}
where $c^1,c^2,c^3,c^4\in\R^{m}$ and $d^1,d^2,d^3,d^4\in\R^n$.  \eqref{eqn:prop:3.4.2} yields
\begin{equation}\label{eqn:prop:3.4.3}
    G_h(t) = c^1(t)(d^1)^T + c^2(t)(d^2)^T + c^3(d^3(t))^T + c^4(d^4(t))^T = C(t)S(t)D(t)^T
\end{equation}
where
\begin{equation}
    C(t) = [c^1(t),c^2(t),c^3,c^4],\quad S(t) = I_{4\times 4},\quad D(t) = [d^1,d^2,d^3(t),d^4(t)].
\end{equation}
Since $C$ and $D$ are not guaranteed to be linearly independent, then it is readily seen that the rank of $G_h$ is at most 4.  The proof is complete.
\end{proof}

\bibliographystyle{plain}
\bibliography{references}

\end{document}